\documentclass[11pt]{article}
\usepackage{multicol}
\usepackage{amssymb,amsmath,amsthm,bm}
\usepackage[colorlinks=true, urlcolor=blue,linkcolor=blue, citecolor=blue]{hyperref}
\usepackage{mathrsfs,verbatim}
\usepackage{caption,cleveref }
\usepackage{graphicx, float}
\usepackage{color, mathtools, tikz, xcolor}
\usepackage{enumerate}
\usepackage{tikz}
\usetikzlibrary{shapes.geometric}
\usepackage{pdfpages}
\usepackage{circuitikz}

\usetikzlibrary{calc}

\usepackage{bbm}
\usepackage[mathscr]{euscript}
\usepackage{appendix}
\usepackage{geometry}
\newcommand{\eps}{\varepsilon}

\newcommand{\ds}{\displaystyle}
\newcommand{\no}{\noindent}

\numberwithin{equation}{section}

\begin{document}

	\newcommand{\bea}{\begin{eqnarray}}
		\newcommand{\ena}{\end{eqnarray}}
	\newcommand{\beas}{\begin{eqnarray*}}
		\newcommand{\enas}{\end{eqnarray*}}
	\newcommand{\beq}{\begin{align}}
		\newcommand{\enq}{\end{align}}
	\def\qed{\hfill \mbox{\rule{0.5em}{0.5em}}}
	\newcommand{\bbox}{\hfill $\Box$}
	\newcommand{\ignore}[1]{}
	\newcommand{\ignorex}[1]{#1}
	\newcommand{\wtilde}[1]{\widetilde{#1}}
	\newcommand{\mq}[1]{\mbox{#1}\quad}
	\newcommand{\bs}[1]{\boldsymbol{#1}}
	\newcommand{\qmq}[1]{\quad\mbox{#1}\quad}
	\newcommand{\qm}[1]{\quad\mbox{#1}}
	\newcommand{\nn}{\nonumber}
	\newcommand{\Bvert}{\left\vert\vphantom{\frac{1}{1}}\right.}
	\newcommand{\To}{\rightarrow}
	\newcommand{\supp}{\mbox{supp}}
	\newcommand{\law}{{\cal L}}
	\newcommand{\Z}{\mathbb{Z}}
	\newcommand{\mc}{\mathcal}
	\newcommand{\mbf}{\mathbf}
	\newcommand{\tbf}{\textbf}
	\newcommand{\lp}{\left(}
	\newcommand{\limm}{\lim\limits}
	\newcommand{\limminf}{\liminf\limits}
	\newcommand{\limmsup}{\limsup\limits}
	\newcommand{\rp}{\right)}
	\newcommand{\mbb}{\mathbb}
	\newcommand{\rainf}{\rightarrow \infty}
	\newtheorem{theorem}{Theorem}[section]
	\newtheorem{problem}[theorem]{Problem}
	\newtheorem{exercise}[theorem]{Exercise}
	\newtheorem{corollary}[theorem]{Corollary}
	\newtheorem{conjecture}[theorem]{Conjecture}
	\newtheorem{claim}[theorem]{Claim}
	\newtheorem{proposition}[theorem]{Proposition}
	\newtheorem{lemma}[theorem]{Lemma}
	\newtheorem{definition}[theorem]{Definition}
	\newtheorem{example}[theorem]{Example}
	\newtheorem{remark}[theorem]{Remark}
	\newtheorem{solution}[theorem]{Solution}
	\newtheorem{case}{Case}[theorem]
	\newtheorem{condition}[theorem]{Condition}
	\newtheorem{assumption}[theorem]{Assumption}
	\newtheorem{note}[theorem]{Note}
	\newtheorem{notes}[theorem]{Notes}	\newtheorem{observation}[theorem]{Observation}
	\newtheorem{readingex}[theorem]{Reading exercise}
	\frenchspacing

\usetikzlibrary{positioning}
\usetikzlibrary{decorations.markings}
\usetikzlibrary{shapes.geometric}
\usetikzlibrary{patterns}
\usetikzlibrary{calc, babel}

\tikzstyle{vertex}=[circle,fill=red!24,minimum size=20pt,inner sep=0pt]
\tikzstyle{selected vertex} = [vertex, fill=red!24]
\tikzstyle{edge} = [draw,thick,-]
\tikzstyle{weight} = [font=\small]
\tikzstyle{selected edge} = [draw,line width=5pt,-,red!50]
\tikzstyle{ignored edge} = [draw,line width=5pt,-,black!20]

\pgfdeclarelayer{edgelayer}
\pgfdeclarelayer{nodelayer}
\pgfsetlayers{edgelayer,nodelayer,main}

\tikzstyle{none}=[inner sep=0pt]
\definecolor{RED}{rgb}{1.000,0.000,0.000}
\definecolor{ORANGE}{rgb}{1.000,0.700,0.000}
\definecolor{BLUE}{rgb}{0.000,0.000,1.000}
\definecolor{GREEN}{rgb}{0.000,0.700,0.300}
\definecolor{GRAY}{rgb}{0.5,0.5,0.5}
\definecolor{SHADEDGRAY}{rgb}{0.9,0.9,0.9}
\definecolor{WHITE}{rgb}{1.000,1.000,1.000}
\definecolor{BLACK}{rgb}{0.000,0.000,0.000}

\tikzstyle{whitecircle}=[circle,fill=WHITE,draw=BLACK]
\tikzstyle{whiterectangle}=[rectangle,fill=WHITE,draw=BLACK]
\tikzstyle{blackcircle}=[circle,fill=BLACK,draw=BLACK]
\tikzstyle{grayrectangle}=[rectangle,fill=GRAY,draw=BLACK]

\tikzstyle{thinarrow}=[-latex,draw=BLACK,line width=0.700]
\tikzstyle{thickedge}=[-,draw=BLACK,line width=2.500]

\tikzstyle{REDARROW}=[-latex,draw=RED,line width=2.00]
\tikzstyle{ORANGEARROW}=[-latex,draw=ORANGE,line width=2.00]
\tikzstyle{BLUEARROW}=[-latex,draw=BLUE,line width=2.00]
\tikzstyle{GREENARROW}=[-latex,draw=GREEN,line width=2.00]
\tikzstyle{GRAYARROW}=[-latex,draw=GRAY,line width=2.00]
\tikzstyle{BLACKARROW}=[-latex,draw=BLACK,line width=2.00]


	\tikzstyle{level 1}=[level distance=2.75cm, sibling distance=5.65cm]
	\tikzstyle{level 2}=[level distance=3cm, sibling distance=2.75cm]
	\tikzstyle{level 3}=[level distance=3.9cm, sibling distance=1.5cm]
	
	\tikzstyle{bag} = [text width=10em, text centered] 
	\tikzstyle{end} = [circle, minimum width=3pt,fill, inner sep=0pt]

	\title{Hamilton Cycles in Dense Regular Digraphs and Oriented Graphs}
	\author{Allan Lo\footnote{School of Mathematics, University of Birmingham, United Kingdom, Email: s.a.lo@bham.ac.uk. A.~Lo was partially supported by EPSRC, grant no. EP/V002279/1 and EP/V048287/1 (A.~Lo). There are no additional data beyond that contained within the main manuscript.} \hspace{0.2in} Viresh Patel\footnote{Korteweg de Vries Instituut voor Wiskunde, Universiteit van Amsterdam, The Netherlands, Email: v.s.patel@uva.nl. V.~Patel  was supported by the Netherlands Organisation for Scientific Research (NWO)
through the Gravitation Gravitation project NETWORKS (grant no. 024.002.003).} \hspace{0.2in} Mehmet Akif Yıldız\footnote{Korteweg de Vries Instituut voor Wiskunde, Universiteit van Amsterdam, The Netherlands. Email: m.a.yildiz@uva.nl.  M.A.~Yıldız was supported by a Marie Skłodowska-Curie Action from the EC (COFUND grant no. 945045) and by the NWO Gravitation project NETWORKS (grant no. 024.002.003). }   } \vspace{0.2in}
	
	\maketitle
	\begin{abstract}
	\no	We prove that for every $\varepsilon > 0$ there exists $n_0=n_0(\varepsilon)$ such that every regular oriented graph on $n > n_0$ vertices and degree at least $(1/4 + \varepsilon)n$ has a Hamilton cycle. This establishes an approximate version of a conjecture of Jackson from 1981. We also establish a result related to a conjecture of K{\"u}hn and Osthus about the Hamiltonicity of regular directed graphs with  suitable degree and connectivity conditions. 

		\bigskip
		
		
	\end{abstract}
	
	\section{Introduction}\label{sec:INTRO}

	\no
	A Hamilton cycle in a (directed) graph is a (directed) cycle that visits every vertex of the (directed) graph. Hamilton cycles are one of the most intensely studied structures in graph theory. There is a rich body of results that establish (best-possible) conditions  guaranteeing existence of  Hamilton cycles in (directed) graphs. Degree conditions that guarantee Hamiltonicity have been of particular interest, as well as the trade-off between degree conditions and other conditions (e.g. various types of connectivity conditions). \\
	
	\no
	In this paper, we are concerned with directed graphs (or digraphs for short) and oriented graphs. Recall that a digraph can have up to two directed edges between any pair of vertices (one in each direction), while an oriented graph can have only one.\\ 

\no	The seminal result in the area is Dirac's theorem~\cite{Dirac}, which states that every graph on $n\geq 3$ vertices with minimum degree at least $n/2$ contains a Hamilton cycle. The disjoint union of two cliques of equal size or the slightly imbalanced complete bipartite graph shows that the bound is best possible. Ghouila-Houri~\cite{GhouilaHouri} proved the corresponding result for directed graphs, which states that every digraph on $n$ vertices with minimum semi-degree (i.e.\ the smaller of the minimum in- and outdegree) at least $n/2$ contains a Hamilton cycle. The bound here is again tight by doubling the edges in the extremal examples for the graph setting. The proofs of both of these results are relatively short, while the corresponding result for oriented graphs, due to Keevash, K{\"u}hn, and Osthus~\cite{OrientedExactMinDegree} given below, is more difficult and uses the Regularity Lemma together with a stability method. Again the degree threshold is tight as demonstrated by examples given in~\cite{OrientedExactMinDegree}.  
	\begin{theorem}
	\label{thm:KKO}
	    There exists an integer $n_0$ such that any oriented graph $G$ on $n\geq n_0$ vertices with minimum semi-degree $\delta^0(G) \geq \lceil (3n - 4)/8 \rceil$ contains a Hamilton cycle.
	\end{theorem}

	\no Here we consider the question of minimum degree thesholds for Hamiltonicity in \emph{regular} (di)graphs possibly with some mild connectivity constraints. In this direction, for the undirected setting, Bollob\'as  and H{\"a}ggkvist (see \cite{JacksonDirected}) independently conjectured that a $t$-connected regular graph with degree at least $n/(t+1)$ is Hamiltonian. That is, the threshold for Hamiltonicity is significantly reduced compared to Dirac's Theorem if we consider regular graphs (with some relatively mild connectivity conditions). Note that the connectivity conditions without regularity is not enough to guarantee Hamiltonicity due to the slightly imbalanced complete bipartite graph.
Jackson~\cite{JacksonDirected} proved the conjecture for $t=2$, while Jung~\cite{Jung} and Jackson, Li, and Zhu~\cite{JacksonLiZhu} gave an example showing the conjecture does not hold for $t \geq 4$. Finally, K\"uhn, Lo, Osthus, and Staden~\cite{IntoTwoBipartiteExpander, kuhn2016solution} resolved the conjecture by proving the case $t=3$ for large regular graphs.
Results in \cite{PatStr} which use ideas from~\cite{IntoTwoBipartiteExpander, kuhn2016solution} also show that the algorithmic Hamilton cycle problem behaves quite differently for dense regular graphs compared to dense graphs. \\

\no
 Jackson conjectured in 1981 that, for oriented graphs, regularity alone is enough to reduce the semi-degree threshold for Hamiltonicity from $\lceil (3n-4)/8 \rceil$ in Theorem~\ref{thm:KKO} to $n/4$.  

	\begin{conjecture}[\cite{JacksonConjecture}]\label{conj:JacksonMain}
 For each $d>2$, every $d$-regular oriented graph on $n\leq 4d+1$ vertices has a Hamilton cycle.
	\end{conjecture}
	
	\no
	Note that the disjoint union of two regular tournaments shows that the degree bound above cannot be improved. Furthermore, one cannot reduce the degree bound even if the oriented graph is strongly $2$-connected; see Proposition~\ref{thm:counterexample2connected}. 	Our main result is an approximate version of Jackson's conjecture. 

	\begin{theorem}\label{thm:MAIN-1-4-CASE}
	For every $\varepsilon>0$, there exists an integer $n_0(\varepsilon)$ such that every $d$-regular oriented graph on $n\geq n_0(\varepsilon)$ vertices with $d\geq (1/4+\eps)n$ is Hamiltonian.
\end{theorem}

    \no
    Recall that Jackson~\cite{JacksonDirected} proved the $t=2$ case of the Bollob{\'a}s--H{\"a}ggkvist conjecture, namely that every $2$-connected regular graph of degree at least $n/3$ has a Hamilton cycle. K{\"u}hn and Osthus gave a corresponding conjecture for digraphs.
	
	\begin{conjecture}[\cite{SurveyDirected}]\label{conj:directedregularHamilton}
    Every strongly $2$-connected $d$-regular digraph on $n$ vertices with $d\geq n/3$ contains a Hamilton cycle.
	\end{conjecture}
	
	\no We give a counterexample to this conjecture (see Proposition~\ref{thm:counterexample2connected}), but we show that a slight modification of the conjecture is true. In particular, $2$-connectivity is replaced with the following slightly different condition.
	We call a digraph $G$ on at least four vertices \emph{strongly well-connected} if for any partition $(A,B)$ of $V(G)$ with $|A|,|B|\geq 2$, there exist two vertex-disjoint edges $ab$ and $cd$ such that $a,d\in A$ and $b,c\in B$. 
	Note that the property of being strongly well-connected and that of being strongly $2$-connected are incomparable\footnote{A directed cycle (on at least $4$ vertices) is strongly well-connected but not strongly $2$-connected; see Proposition~\ref{thm:counterexample2connected} for the converse example}; on the other hand being strongly well-connected is stronger than being strongly connected but weaker than being strongly $3$-connected.
	Our second result is an approximate version of a slightly modified statement of Conjecture~\ref{conj:directedregularHamilton}.

	\begin{theorem}\label{thm:MAIN-1-3-CASE}
	For every $\eps > 0$, there exists an integer $n_0(\varepsilon)$ such that every strongly well-connected $d$-regular digraph on $n\geq n_0(\varepsilon)$ vertices with $d\geq (1/3+\eps)n$ is Hamiltonian.
	\end{theorem}
	
	\no
	Note that K{\"u}hn and Osthus~\cite{SurveyDirected} give an example that shows the degree bound in Conjecture~\ref{conj:directedregularHamilton} cannot be reduced, i.e.\ an example of a strongly 2-connected regular digraph on $n$ vertices and degree close to $n/3$. The same example is easily seen to be strongly well-connected, showing that we cannot take the degree to be smaller than $n/3$ in Theorem~\ref{thm:MAIN-1-3-CASE}. \\

\no		
Our methods are based on the robust expanders technique of Kühn and Osthus which have been used to resolve a number of conjectures (see \cite{On-Kelly-Conjecture, NEWW}). Any directed dense graph that is a robust expander automatically contains a Hamilton cycle. An important part of this paper is to gain an understanding of dense directed graphs that are not robust expanders. In particular, we are able to construct vertex partitions of such digraphs with useful expansion properties. Although we do not show it directly, such partitions almost immediately allow us to construct very long cycles in the required settings (that is cycles that pass through all but a small proportion of the vertices). The remainder of the paper is devoted to giving delicate balancing arguments to obtain full Hamilton cycles.		\\

\no
The paper is organised as follows.
In the next subsection, we give the counterexample to Conjecture~\ref{conj:directedregularHamilton}.
In Section~\ref{ch:notations} we give notation, preliminaries and a sketch proof. In Section~\ref{ch:partition} we develop the necessary language of partitions and establish some of their basic properties. Section~\ref{ch:balance} is devoted mainly to giving the balancing arguments that will allow us to construct full Hamilton cycles. Section~\ref{ch:partition-to-hamilton} shows how to combine earlier results to show that dense directed and oriented graphs with certain vertex partitions contain Hamilton cycles. In Section~\ref{ch:mainproofs} we prove Theorems~\ref{thm:MAIN-1-3-CASE} and~\ref{thm:MAIN-1-4-CASE}. We pose some open problems in Section~\ref{ch:conclusion}.

\subsection{Counterexample to Conjecture~\ref{conj:directedregularHamilton}}\label{ch:counter-example}

\begin{proposition}\label{thm:counterexample2connected}
	For $n\geq 3$, there exists a strongly $2$-connected $(n-1)$-regular digraph on $2n$ vertices with no Hamilton cycle. 
	For $n \ge 3$, there exists a strongly $2$-connected $(n-1)$-regular oriented graph on $4n+2$ vertices with no Hamilton cycle.
\end{proposition}

\begin{figure}[h]
	\centering
	\includegraphics[width=8cm,angle=0,height=4cm]{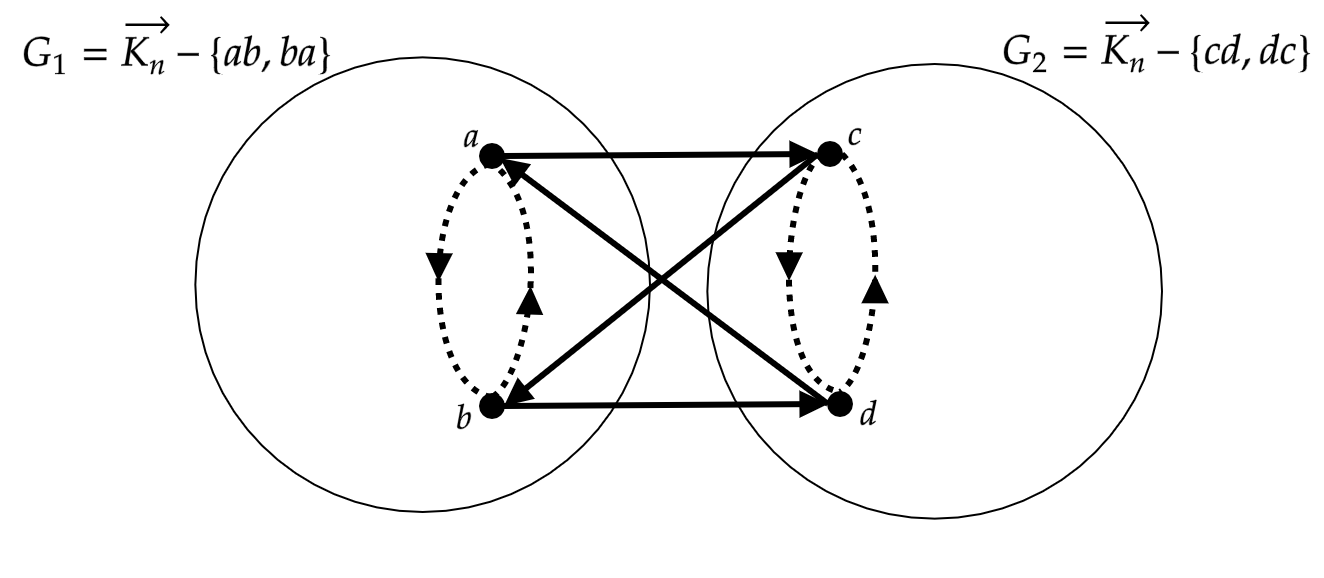}
	\caption{A strongly 2-connected $(n-1)$-regular digraph $G$ on $2n$ vertices}
	\label{fig:counter-example}
\end{figure}

\begin{proof}
	Let $G'$ be the digraph that is the disjoint union of two complete digraphs $G_1$ and $G_2$ each of size $n$. Let $a,b \in V(G_1)$ and $c,d \in V(G_2)$. Let $G$ be the digraph obtained from $G'$ by deleting the edges $ab$, $ba$, $cd$, and $dc$, and adding the edges $ac$, $cb$, $bd$, $da$ (see Figure~\ref{fig:counter-example}). It is clear that $G$ is a strongly $2$-connected $(n-1)$-regular digraph on $2n$ vertices. \\
	
	\no
	It is easy to see that $G$ has no Hamilton cycle. Indeed, any Hamilton cycle $H$ of $G$ must use at least one edge inside one of the cliques (since $n\geq 3$). Let $P$ be a maximal path of $H$ inside one of the cliques (say $G_1$) with at least one edge. Let $e$ and $e'$ be the edges of $H$ that extend $P$ into $G_2$. Then $e$ and $e'$ must be vertex-disjoint edges that cross between $G_1$ and $G_2$ in opposite directions. But $G$ does not have such a pair of edges. \\

\no 
The oriented graph is constructed similarly. 
It is easy to construct a regular tournament of order~$2n+1$ that contains two cycles that together span the tournament and which have exactly two vertices in common. Indeed, we start with the two directed cycles with common vertices say $a$ and $b$. 
The (undirected) complement is Eulerian, that is, all vertices have even degree, and so we orient these edges using an Euler tour.  
This gives the desired tournament. \\

 \no
	Let $G'$ be the disjoint union of two such regular tournaments $G_1$ and $G_2$ each of order $2n+1$.
	Let $C_1$ and $C_1'$ be the two directed cycles in $G_1$ such that $V(C_1) \cup V(C_1') = V(G_1)$ and $V(C_1) \cap V(C_1') = \{a,b\}$. Similarly, let $C_2$ and $C_2'$ be two directed cycles in $G_2$ such that $V(C_2) \cup V(C_2') = V(G_2)$ and $V(C_2) \cap V(C_2') = \{c,d\}$.	Let $G$ be obtained from $G'$ by deleting the edges of $C_1 \cup C_1'\cup C_2 \cup C_2'$, and adding the edges $ac$, $cb$, $bd$, $da$. 
	It is easy to check that $G$ is a strongly $2$-connected, $(n-1)$-regular, oriented graph on $4n+2$ vertices. Note that $G$ is not Hamiltonian by a similar argument as above.
	\end{proof}

	\section{Notation and preliminaries}\label{ch:notations}

	 \no Throughout the paper, we use standard graph theory notation and terminology.
	 For a digraph $G$, we denote its vertex set by $V(G)$ and its edge set $E(G)$. For $a,b \in V(G)$, we write $ab$ for the directed edge from $a$ to $b$. We sometimes write $|G|$ for the number of vertices in $G$ and $e(G)$ for the number of edges in $G$. We write $H \subseteq G$ to mean $H$ is a subdigraph of $G$, i.e.\ $V(H) \subseteq V(G)$ and $E(H) \subseteq E(G)$. We sometimes think of $F \subseteq E(G)$ as a subdigraph of $G$ with vertex set consisting of those vertices incident to edges in $F$ and edge set $F$.
	 For $S \subseteq V(G)$, we write $G[S]$ for the subdigraph of $G$ induced by $S$ and $G-S$ for the digraph $G[V(G) \setminus S]$. For $A,B \subseteq V(G)$ not necessarily disjoint, we define $E_G(A,B):= \{ab \in E(G): a \in A, \; b \in B \}$ and we write $G[A,B]$ for the graph with vertex set $A \cup B$ and edge set $E_G(A,B)$. We write $e_G(A,B) := |E_G(A,B)|$. We often drop subscripts if these are clear from context. For two digraphs $H_1$ and $H_2$, the union $H_1 \cup H_2$ is the digraph with vertex set $V(H_1) \cup V(H_2)$ and edge set $E(H_1) \cup E(H_2)$. 
	 We say that an undirected graph $G$ is bipartite with bipartition $(A,B)$ if $V(G)=A\cup B$ and $E(G) \subseteq \{ab: a \in A, \; b \in B \}$.\\

	 \no For a digraph $G$ and $v\in V(G)$, we denote the set of outneighbours and inneighbours of $v$ by $N_G^{+}(v)$ and $N_G^{-}(v)$ respectively, and we write $d_G^{+}(v)=|N_G^{+}(v)|$ and $d_G^{-}(v)=|N_G^{-}(v)|$ for the out- and indegree of $v$ respectively.  
	 For $S\subseteq V(G)$ we write $d_S^{-}(v):= |N_G^-(v) \cap S|$ and $d_S^{+}(v):= |N_G^+(v) \cap S|$. We write $\delta^+(G)$ and $\delta^-(G)$ respectively for the minimum out- and indegree of $G$, and $\delta^{0}(G):=\min\{\delta^{+}(G),\delta^{-}(G)\}$ for the minimum semi-degree.
	 Similarly, the maximum semi-degree $\Delta^{0}(G)$ of $G$ is defined by $\Delta^{0}(G):=\max\{\Delta^{+}(G),\Delta^{-}(G)\}$ where $\Delta^{+}(G)$ and $\Delta^{-}(G)$ denote the maximum out- and maximum indegree of $G$ respectively. A digraph is called $d$-regular if each vertex has exactly $d$ outneighbours and $d$ inneighbours. For undirected graphs $G$, we write $\Delta(G)$ and $\delta(G)$ respectively for the maximum degree and the minimum degree. A graph is called $d$-regular if each vertex has exactly $d$ neighbours. 
	 \\

	\no A directed path $Q$ in a digraph $G$ is a subdigraph of $G$ where $V(Q) = \{v_1, \ldots, v_k \}$ for some $k \in \mathbb{N}$ and where $E(Q) = \{v_1v_2, v_2v_3, \ldots, v_{k-1}v_k \}$. A directed cycle in $G$ is exactly the same except that it also includes the edge $v_kv_1$. 
	A set of vertex-disjoint directed paths $\mathcal{Q}=\{Q_1,Q_2,\ldots\}$ in $G$ is called a \emph{path system} in $G$. We interchangeably think of $\mathcal{Q}$ as a set of vertex-disjoint directed paths in $G$ and as a subgraph of $G$ with vertex set $V(\mathcal{Q}) = \cup_{i} V(Q_i)$ and edge set $E(\mathcal{Q}) = \cup_{i} E(Q_i)$. We sometimes call this subgraph the graph induced by $\mathcal{Q}$.  
	A matching $M$ in a digraph (or undirected graph) $G$ is a set of edges $M \subseteq E(G)$ such that every vertex of $G$ is incident to at most one edge in $M$. We say that a matching $M$ \emph{covers} $S \subseteq V(G)$ if every vertex in $S$ is incident to some edge in $M$.  \\

	\no For two sets $A$ and $B$, the \textit{symmetric difference} of $A$ and $B$ is the set $A\triangle B := (A\setminus B) \cup (B\setminus A)$.
	For $k\in\mathbb{N}$, we sometimes denote the set $\{1,2,\ldots,k\}$ by $[k]$. 
	For $x, y \in (0,1]$, we often use the notation $x \ll y$ to mean that $x$ is sufficiently small as a function of $y$ i.e.\ $x \leq f(y)$ for some implicitly given non-decreasing function $f:(0,1] \rightarrow (0,1]$. \\

	\subsection{Tools}
	
	\no 
	We will require Vizing's theorem for multigraphs in the proof of Lemma~\ref{lem:MATHCING_LEMMA_I}. Let $H$ be an (undirected) multigraph (without loops). The multiplicity $\mu(H)$ of $H$ is maximum number of edges between two vertices of $H$, and, as usual, $\Delta(H)$ is the maximum degree of $H$.
	A proper $k$-edge-colouring of $H$ is an assignment of $k$ colours to the edges of $H$ such that incident edges receive different colours.
	
	\begin{theorem}[\cite{Vizing}; see e.g.\ \cite{Diestel}]
	\label{thm:Vizing}
	Any multigraph $H$ has a proper $k$-edge colouring with $k = \Delta(H) + \mu(H)$ colours. In particular, by taking the largest colour class, there is a matching in $H$ of size at least $e(H)/(\Delta(H)+\mu(H))$.
	\end{theorem}

	\no
	In Lemma~\ref{lem:MATCHING_LEMMA_II}, we will require a Chernoff inequality for bounding the tail probabilities of binomial random variables.
	For a random variable $X$, write $\mathbb{E}[X]$ for the expectation of $X$. 
	We write $X \sim {\rm Bin}(n,p)$ to mean that $X$ is distributed as a binomial random variable with parameters $n$ and $p$, that is a random variable that counts the number of heads in $n$ independent coin flips where the probability of heads is $p$. In that case we have $\mathbb{E}[X]=np$ and the following bound. 	
	\begin{theorem}[see~\cite{Chernoff}]
	\label{thm:Chernoff}
	Suppose $X_1,X_2,\ldots,X_n$ are independent random variables taking values in $\{0,1\}$ and $X = X_1 + \cdots + X_n$. Then,  for all $0\leq \delta\leq 1$, we have
	\begin{equation*}
	\mathbb{P}(X\leq (1-\delta)\mathbb{E}(X))\leq \exp{\left(-\delta^2 \mathbb{E}(X)/2\right)}.
	\end{equation*}
	In particular, this holds for $X \sim {\rm Bin}(n,p)$. 
 \end{theorem}

	\subsection{Robust expanders}
	\no In this subsection we define robust expanders and discuss some of their useful properties. 
	
	\begin{definition}\label{def:RobustExpander}

	 Fix a digraph $G$ on $n$ vertices and parameters $0<\nu < \tau <1$. For $S\subseteq V(G)$, the \emph{robust $\nu$-outneighbourhood of $S$} is the set \emph{$\text{RN}_{\nu}^{+}(S) := \{v \in V(G): |N^-_G(v) \cap S| \geq \nu n \}$}. We say $G$ is a \emph{robust $(\nu,\tau)$-outexpander} if \emph{$|\text{RN}_{\nu}^{+}(S)|\geq |S|+\nu n$} for all subsets $S\subseteq V(G)$ satisfying $\tau n\leq |S|\leq (1-\tau)n$.
 	\end{definition}

\no If the constant $\nu$ used is clear from  context, we write $\text{RN}^{+}(S)$. The notion of robust expansion has been key to proving numerous conjectures about Hamilton cycles. One of the starting points is the following seminal result which states that robust expanders with certain minimum degree condition are Hamiltonian.
	
\begin{theorem}[\cite{RobustExpanderHamilton}; see also \cite{LoPatel}]
\label{thm:RobustExpanderImpliesHamilton}
Let $1/n \ll \nu\leq \tau\ll\gamma<1$. If $G$ is an $n$-vertex digraph  with $\delta^{0}(G)\geq \gamma n$ such that $G$ is a robust $(\nu,\tau)$-outexpander, then $G$ contains a Hamilton cycle.
\end{theorem}

\no The following straightforward lemma shows that robust expansion is a ``robust" property, i.e.\ if $G$ is a robust $(\nu,\tau)$-outexpander, then adding or deleting a small number of vertices results in another robust outexpander with slightly worse parameters. 
	
	\begin{lemma}[\cite{IntoTwoBipartiteExpander}]
	\label{lem:SYMMETRIC-DIFFERENCE}
	 Let $0<\nu\ll\tau\ll 1$. Suppose that $G$ is a digraph and $U,U'\subseteq V(G)$ are such that $G[U]$ is a robust $(\nu,\tau)$-outexpander and $|U\triangle U'|\leq \nu |U|/2$. Then, $G[U']$ is a robust $(\nu/2,2\tau)$-outexpander.
	\end{lemma}

\no By taking $(U,U')=(V(G)-S,V(G))$, Lemma~\ref{lem:SYMMETRIC-DIFFERENCE} has the following corollary.

	\begin{corollary}\label{cor:AddingSmallPartIntoRobustExpander}
	Let $1/n \ll \nu \ll \tau \ll 1$.
		If $G$ is an $n$-vertex digraph and $S\subset V(G)$ such that $|S|\leq \nu |G|/2$ and $G-S$ is a robust $(\nu,\tau)$-outexpander then $G$ is a robust $(\nu/2,2\tau)$-outexpander.
	\end{corollary}

	\no The next lemma shows that any digraph $G$ with minimum semi-degree slightly higher than $|G|/2$ is a robust outexpander.  

	\begin{lemma}[\cite{On-Kelly-Conjecture}]\label{lem:HigherThanHalfGivesRobustExpander}
	Let $0<\nu\leq \tau\leq \eps<1$ be constants such that $\eps\geq 2\nu/\tau$. Let $G$ be a digraph on $n$ vertices with $\delta^{0}(G)\geq (1/2+\eps)n$. Then, $G$ is a robust $(\nu,\tau)$-outexpander.
	\end{lemma}

\no In fact we can relax the degree condition in Lemma~\ref{lem:HigherThanHalfGivesRobustExpander} and allow a small number of vertices to violate the minimum degree condition.

	\begin{corollary}\label{cor:Higher-Than-Half-For-All-But-Few}
		Let $1/n<\nu,\rho\ll\tau\ll\varepsilon\ll\alpha <1$ be constants. If $G$ is an $n$-vertex digraph  such that $d^{+}(v), d^{-}(v)\geq (1/2+\eps)n$ for all but at most $\rho n$ vertices $v \in V(G)$, then $G$ is a robust $(\nu,\tau)$-outexpander. In particular, if additionally $\delta^{0}(G)\geq \alpha n$, then $G$ contains a Hamilton cycle.
	\end{corollary}  
\begin{proof}
Fix $\nu'$ and $\tau'$ such that $\nu,\rho\ll \nu'\ll \tau'\ll\tau$. Let $W$ be the set of vertices $v$ in $G$ such that $\min\{d^{+}(v), d^{-}(v) \} < (1/2+\eps)n$. Then, observe that $G'=G-W$ satisfies 
\[
d_{G'}^{+}(v), d_{G'}^{-}(v)\geq (1/2+\varepsilon-\rho)n\geq (1/2+\varepsilon-\rho)|G'|
\]
for all $v\in V(G')$. By our choice of parameters, we can conclude that $G'$ is a robust $(\nu',\tau')$-outexpander by Lemma~\ref{lem:HigherThanHalfGivesRobustExpander} since $\tau'\leq\eps-\rho$ and $2\nu'/\tau'\leq \eps-\rho$. Moreover, we have $|W|=\rho n\leq \nu' n/2$. Therefore, $G$ is a robust $(\nu,\tau)$-outexpander by Corollary~\ref{cor:AddingSmallPartIntoRobustExpander}, and the result follows by Theorem~\ref{thm:RobustExpanderImpliesHamilton}.	
\end{proof}

\subsection{Sketch proof}

Note that the sketch proof we give below only makes reference to Definition~\ref{def:RobustExpander}, Theorem~\ref{thm:RobustExpanderImpliesHamilton}, and Lemma~\ref{lem:HigherThanHalfGivesRobustExpander}.
We will sketch the proof of Theorem~\ref{thm:MAIN-1-3-CASE} and then explain how these ideas are generalised and refined to prove Theorem~\ref{thm:MAIN-1-4-CASE}. \\

\no
Let $G=(V,E)$ be an $n$-vertex, $d$-regular digraph with $d \geq (1/3+\eps)n$. If $G$ is a robust $(\nu, \tau)$-outexpander (for suitable parameters $\nu$ and $\tau$), then by Theorem~\ref{thm:RobustExpanderImpliesHamilton}, we know $G$ has a Hamilton cycle. So assume $G$ is not a robust $(\nu, \tau)$-outexpander. We describe a useful vertex partition of $G$. \\

\no
{\bf Partitioning non-robust expanders} -
Since $G$ is not a robust $(\nu, \tau)$-outexpander we know by Definition~\ref{def:RobustExpander} that there exists $S \subseteq V(G)$ such that $\tau n \leq |S| \leq (1- \tau)n$ and $|\text{RN}_{\nu}^{+}(S)| \leq |S| + \nu n$. This immediately gives us a partition of $V(G)$ into four parts given by
\begin{align*}
V_{11} = S \cap \text{RN}^{+}(S), \:\:\:\:\:
V_{12} = S \setminus \text{RN}^{+}(S), \:\:\:\:\: 
V_{21} = \text{RN}^{+}(S) \setminus S, \:\:\:\:\:
V_{22} = V \setminus (S \cup \text{RN}^{+}(S)).
\end{align*}
We see that most outedges from vertices in $S$ go to $\text{RN}^{+}(S)$ by the definition of $\text{RN}^{+}(S)$. Moreover, $S$ and $\text{RN}^{+}(S)$ must be of similar size; indeed we already know $\text{RN}^{+}(S)$ is not significantly bigger than~$S$, and it cannot be significantly smaller because otherwise the degrees in $\text{RN}^{+}(S)$ would be larger than degrees in $S$ violating that $G$ is regular. Also most outedges of vertices in $V \setminus S$ go to $V \setminus \text{RN}^{+}(S)$ because if many of these edges went to $\text{RN}^{+}(S)$, the degrees in $\text{RN}^{+}(S)$ would again be too large violating that $G$ is regular. All of this is straightforward to show and captured in Lemma~\ref{lem:structure_not_expander}. The structure we obtain is depicted in Figure~\ref{fig:partition-expander}. To summarise, we have that
\begin{itemize}
    \item[(a)] $|S| \approx |\text{RN}^{+}(S)|$ so $|V_{12}| \approx |V_{21}|$, 
    \item[(b)] most edges of $G$ are from $S$ to $\text{RN}^{+}(S)$ and from $V \setminus S$ to $V \setminus \text{RN}^{+}(S)$. We call these the good edges of $G$, and
    \item[(c)] (b) implies that we must have $|S|, |V \setminus S| \gtrapprox d$ so that in particular $n/3 \lessapprox |S|, |V \setminus S| \lessapprox 2n/3 $ 
\end{itemize}


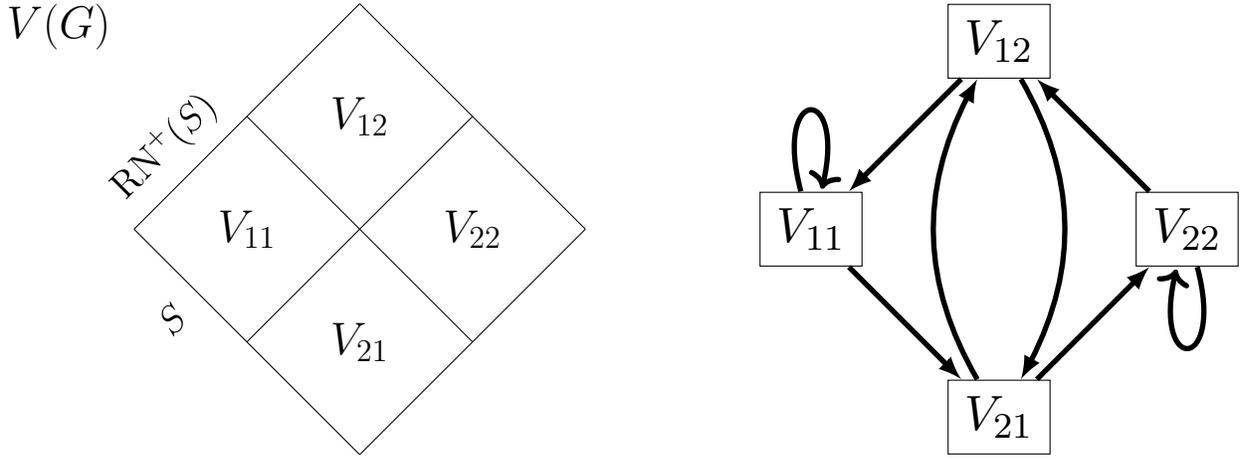
\begin{figure}[h]
\centering
\begin{tikzpicture}
\draw (-7.5,0) node {\LARGE$V_{11}$};
\draw (-4.5,0) node {\LARGE$V_{22}$};
\draw (-6,1.5) node {\LARGE$V_{12}$};
\draw (-6,-1.5) node {\LARGE$V_{21}$};
\draw   (-9,0) -- (-7.5,1.5);
\draw   (-6,3) -- (-7.5,1.5);
\draw   (-6,3) -- (-4.5,1.5);
\draw   (-3,0) -- (-4.5,1.5);
\draw   (-3,0) -- (-4.5,-1.5);
\draw   (-6,-3) -- (-4.5,-1.5);
\draw   (-6,-3) -- (-7.5,-1.5);
\draw   (-9,0) -- (-7.5,-1.5);
\draw   (-4.5,1.5) -- (-6,0);
\draw   (-4.5,-1.5) -- (-6,0);
\draw   (-7.5,1.5) -- (-6,0);
\draw   (-7.5,-1.5) -- (-6,0);

\node at (-10,2.7) {\LARGE$V(G)$};
\node[rotate=45] at (-8.6,1.1) {\Large$\text{RN}^{+}(S)$};
\node[rotate=45] at (-8.5,-1.2) {\Large$S$};

\node [style=whiterectangle,scale=1.7] (1) at (2.5, 2.5) {$V_{12}$};
\node [style=whiterectangle,scale=1.7] (2) at (0, 0) {$V_{11}$};
\node [style=whiterectangle,scale=1.7] (3) at (5, 0) {$V_{22}$};
\node [style=whiterectangle,scale=1.7] (4) at (2.5, -2.5) {$V_{21}$};

\draw [style=BLACKARROW] (1) to (2);
\draw [style=BLACKARROW] (2) to (4);
\draw [style=BLACKARROW] (4) to (3);
\draw [style=BLACKARROW] (3) to (1);
\path [style=BLACKARROW,scale=3] (2) edge [loop above] node {} (2);
\path [style=BLACKARROW,scale=3] (3) edge [loop below] node {} (3);
\path [style=BLACKARROW] (1) edge [bend left] node {} (4);
\path [style=BLACKARROW] (4) edge [bend left] node {} (1);
\end{tikzpicture}
	\caption{The 4-partition of $V(G)$ with $|V_{12}| \approx |V_{21}|$, and directions of the good edges.}
	\label{fig:partition-expander}
\end{figure}

\no
Next we describe the strategy to construct a Hamilton cycle in $G$ using this partition. \\

\no
{\bf Constructing the Hamilton cycle for balanced partitions} - We first describe how to construct the Hamilton cycle in the special case $|V_{12}| = |V_{21}|>0$. In that case, let $V_{12} = \{ x_1, \ldots, x_k \}$ and $V_{21}= \{y_1, \ldots, y_k \}$. Consider the two edge-disjoint subgraphs $G_1$ and $G_2$ of $G$ given by (see Figure~\ref{fig:graphs-G1-G2})
\begin{align*}
    G_1 &= \left( S \cup \text{RN}^{+}(S), \; E_G(S, \text{RN}^{+}(S) \right) \\
    &= (V_{11} \cup V_{12} \cup V_{21}, \; E(V_{12}, V_{11}) \cup E(V_{11}, V_{11}) \cup E(V_{11}, V_{21}) \cup E(V_{12}, V_{21}) ), 
\end{align*}
and 
\begin{align*}
    G_2 &= \left( (V \setminus S) \cup (V \setminus \text{RN}^{+}(S)), \; E_G(V \setminus S, V \setminus \text{RN}^{+}(S) \right) \\
    &= (V_{22} \cup V_{12} \cup V_{21}, \; E(V_{21}, V_{22}) \cup E(V_{22}, V_{22}) \cup E(V_{22}, V_{12}) \cup E(V_{21}, V_{12}) ). 
\end{align*}


\begin{figure}[h]
    \centering

    \begin{tikzpicture}[scale=0.7]

\node at (-1.5,2.7) {\LARGE$G$};
\node [style=whiterectangle,scale=1.25] (1) at (1.5, 2.5) {$V_{12}$};
\node [style=whiterectangle,scale=1.25] (2) at (-1, 0) {$V_{11}$};
\node [style=whiterectangle,scale=1.25] (3) at (4, 0) {$V_{22}$};
\node [style=whiterectangle,scale=1.25] (4) at (1.5, -2.5) {$V_{21}$};

\draw [style=BLACKARROW] (1) to (2);
\draw [style=BLACKARROW] (2) to (4);
\draw [style=BLACKARROW] (4) to (3);
\draw [style=BLACKARROW] (3) to (1);
\path [style=BLACKARROW,scale=3] (2) edge [loop above] node {} (2);
\path [style=BLACKARROW,scale=3] (3) edge [loop below] node {} (3);
\path [style=BLACKARROW] (1) edge [bend left] node {} (4);
\path [style=BLACKARROW] (4) edge [bend left] node {} (1);

\node at (8,2.7) {\LARGE$G_1$};
\node [style=whiterectangle,scale=1.25] (5) at (11, 2.5) {$V_{12}$};
\node [style=whiterectangle,scale=1.25] (6) at (8.5, 0) {$V_{11}$};
\node [style=whiterectangle,scale=1.25] (7) at (11, -2.5) {$V_{21}$};

\draw [style=BLACKARROW] (5) to (6);
\draw [style=BLACKARROW] (6) to (7);
\path [style=BLACKARROW,scale=3] (6) edge [loop above] node {} (6);
\path [style=BLACKARROW] (5) edge [bend left] node {} (7);

\node at (14.5,2.7) {\LARGE$G_2$};
\node [style=whiterectangle,scale=1.25] (8) at (16.5, 2.5) {$V_{12}$};
\node [style=whiterectangle,scale=1.25] (9) at (19, 0) {$V_{22}$};
\node [style=whiterectangle,scale=1.25] (10) at (16.5, -2.5) {$V_{21}$};

\draw [style=BLACKARROW] (10) to (9);
\draw [style=BLACKARROW] (9) to (8);
\path [style=BLACKARROW,scale=3] (9) edge [loop below] node {} (9);
\path [style=BLACKARROW] (10) edge [bend left] node {} (8);
    
    \end{tikzpicture}
    \caption{The edge-disjoint subgraphs $G_1$ and $G_2$ of $G$.}
	\label{fig:graphs-G1-G2}
\end{figure}
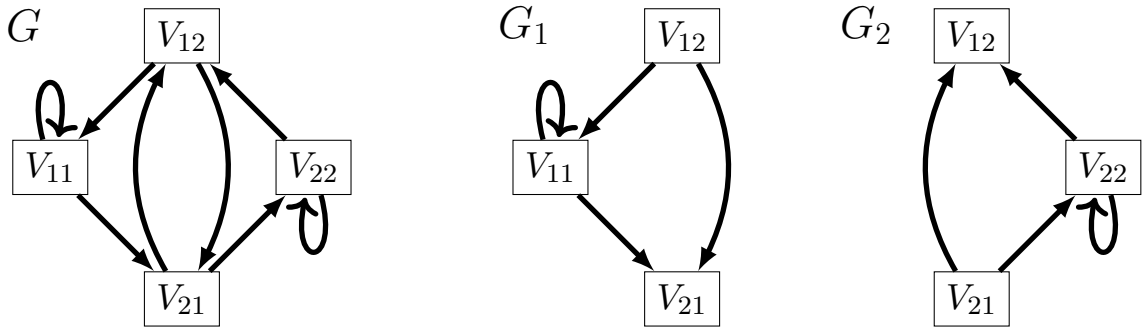

\no
Suppose we can find  
\begin{itemize}
    \item[(i)] vertex-disjoint paths $Q^1_1, \ldots, Q^1_k$ in $G_1$ that together span $V(G_1)$ and where $Q^1_i$ is from $x_i$ to $y_{\sigma(i)}$ for some permutation $\sigma$ on $[k]$,
    \item[(ii)] vertex-disjoint paths $Q^2_1, \ldots, Q^2_k$ in $G_2$ that together span $V(G_2)$ and where $Q^2_i$ is from $y_i$ to $x_{\pi(i)}$ for some permutation $\pi$ on $[k]$,
    \item[(iii)] and where the permutation $\pi \sigma$ is a cyclic permutation.
\end{itemize}
Then it is easy to see that the union of these paths forms a Hamilton cycle. We find these paths as follows. \\

\no
Consider $G_1$ first. We construct the graph $J_1$ from $G_1$ by identifying $x_i$ with $y_i$ for every $i \in [k]$ and keeping all edges (except any self loops). The vertex which replaces $x_i$ and $y_i$ is called $i$. From the structure of $G_1$, it is not hard to see that most vertices in $J_1$ have degree roughly $d = (1/3 + \eps)n$, while $|J_1| = |S| \lessapprox 2n/3$ by (c).
So most vertices in $J_1$ have in- and outdegree at least $(1/2+\eps/2)|J_1|$, which implies $J_1$ is a robust expander by Lemma~\ref{lem:HigherThanHalfGivesRobustExpander}.\footnote{Any enumeration of the vertices in $V_{12}$ and $V_{21}$ would lead to $J_1$ being a robust expander.} Therefore $J_1$ has a Hamilton cycle $H_1$ by Theorem~\ref{thm:RobustExpanderImpliesHamilton}. \\

\no 
Let $\sigma$ be the permutation on $[k]$ where $\sigma(i)$ is the vertex in $[k]$ after $i$ that is visited by $H_1$. Therefore $H_1$ is the union of paths $R_1, \ldots, R_k$ where $R_i$ is from $i$ to $\sigma(i)$, which corresponds in $G_1$ to the path $Q_i^1$ from $x_i$ to $y_{\sigma(i)}$; these paths can easily be seen to satisfy (i) (see Figure~\ref{fig:hamilton-to-paths}). Next, we obtain $J_2$ from $G_2$ by identifying the vertex $x_i$ with $y_{\sigma(i)}$, and labelling the resulting vertex $i$, for every $i \in [k]$ similarly as for $J_1$. Again, we find that $J_2$ is a robust expander and so has a Hamilton cycle $H_2$. Let $\pi$ be the permutation on $[k]$ such that $\pi(i)$ is the next vertex in $[k]$ after $i$ visited by $H_2$. Using the same argument as before, we obtain paths $Q^2_1, \ldots, Q^2_k$ satisfying (ii). By our choice of identification in $J_2$, and since $H_2$ is a Hamilton cycle, it is easy to see that $\pi$ and $\sigma$ satisfy~(iii). \\

\begin{figure}[h]
	\centering
	\begin{tikzpicture}[scale=0.75]
 \draw (-0.5,8.7) node [anchor=north west][inner sep=0.75pt]    {\Large$(A)$};
 \draw (-6.3,3.2) node [anchor=north west][inner sep=0.75pt]    {\Large$V_{11}$};
  \draw (6.1,6.8) node [anchor=north west][inner sep=0.75pt]    {\Large$V_{12}$};
   \draw (6.1,-0.1) node [anchor=north west][inner sep=0.75pt]    {\Large$V_{21}$};
 \draw   (-5.2,0) -- (-5.2,5.8) -- (-2,5.8) -- (-2,0) -- cycle ;
  \draw   (0,7.5) -- (0,5.5) -- (6,5.5) -- (6,7.5) -- cycle ;
\draw   (0,0.5) -- (0,-1.5) -- (6,-1.5) -- (6,0.5) -- cycle ;
 \begin{pgfonlayer}{nodelayer}
	\node [style=whitecircle,scale=0.75] (1) at (-3, 5.2) {$a_1$};
	\node [style=whitecircle,scale=0.75] (2) at (-4.5,3.5) {$a_2$};
	\node [style=whitecircle,scale=0.75] (3) at (-3, 2.5) {$a_3$};
	\node [style=whitecircle,scale=0.75] (4) at (-4,0.5) {$a_4$};
	\node [style=whitecircle,scale=0.75] (5) at (1, 6.5) {$x_1$};
	\node [style=whitecircle,scale=0.75] (6) at (3, 6.5) {$x_2$};
	\node [style=whitecircle,scale=0.75] (7) at (5, 6.5) {$x_3$};
	\node [style=whitecircle,scale=0.75] (8) at (1, -0.5) {$y_1$};
        \node [style=whitecircle,scale=0.75] (9) at (3, -0.5) {$y_2$};
        \node [style=whitecircle,scale=0.75] (10) at (5, -0.5) {$y_3$};
	\end{pgfonlayer}

 \begin{pgfonlayer}{edgelayer}
	\draw [style=REDARROW] (5) to (2);
 	\draw [style=GRAYARROW] (5) to (10);
        \draw[style=BLUEARROW] (6) to (1);
        \draw[style=GREENARROW] (7) to (4);
        \draw[style=GRAYARROW] (6) to (8);	
        \draw [style=GRAYARROW] (7) to (9);
	\draw [style=BLUEARROW] (1) to (10);
	\draw [style=GRAYARROW] (1) to (2);
	\draw [style=REDARROW] (2) to (3);
		\draw [style=REDARROW] (3) to (9);
 	\draw [style=GREENARROW] (4) to (8);
 	\draw [style=GRAYARROW] (3) to (4);
		\end{pgfonlayer}

	\end{tikzpicture}

\begin{tikzpicture}[scale=0.75]
\def \n {7}
\def \radius {4cm}
\def \margin {7} 
\draw (-6,0.5) node [anchor=north west][inner sep=0.75pt]  {\Large$(B)$};
\node[draw, circle,scale=0.75] (11) at ({360/7 * 0}:\radius) {$1$};
\node[draw, circle,scale=0.75] (12) at ({360/7 * 1}:\radius) {$a_2$};
\node[draw, circle,scale=0.75] (13) at ({360/7 * 2}:\radius) {$a_3$};
\node[draw, circle,scale=0.75] (14) at ({360/7 * 3}:\radius) {$2$};
\node[draw, circle,scale=0.75] (15) at ({360/7 * 4}:\radius) {$a_1$};
\node[draw, circle,scale=0.75] (16) at ({360/7 * 5}:\radius) {$3$};
\node[draw, circle,scale=0.75] (17) at ({360/7 * 6}:\radius) {$a_4$};

    \draw [style=GRAYARROW] (15) to (12);
    \draw [style=GRAYARROW] (13) to (17);
    \draw [style=GRAYARROW] (11) to (16);
    \draw [style=GRAYARROW] (14) to (11);
    \draw [style=GRAYARROW] (16) to (14);
    
   \foreach \s in {1,2,3}
{
  \draw[style=REDARROW] ({360/\n * (\s - 1)+\margin}:4) 
    arc ({360/\n * (\s - 1)+\margin}:{360/\n * (\s)-\margin}:4);
}

   \foreach \s in {4,5}
{
  \draw[style=BLUEARROW] ({360/\n * (\s - 1)+\margin}:4) 
    arc ({360/\n * (\s - 1)+\margin}:{360/\n * (\s)-\margin}:4);
}
    
   \foreach \s in {6,7}
{
  \draw[style=GREENARROW] ({360/\n * (\s - 1)+\margin}:4) 
    arc ({360/\n * (\s - 1)+\margin}:{360/\n * (\s)-\margin}:4);
}


\draw (15,0.5) node [anchor=north west][inner sep=0.75pt]  {\Large$(C)$};

\begin{pgfonlayer}{nodelayer}
    \node [style=whitecircle,scale=0.75] (35) at (7.5, 3) {$x_1$};
    \node [style=whitecircle,scale=0.75] (32) at (10,3) {$a_2$};
    \node [style=whitecircle,scale=0.75] (33) at (12.5, 3) {$a_3$};
    \node [style=whitecircle,scale=0.75] (39) at (15, 3) {$y_2$};

    \node [style=whitecircle,scale=0.75] (36) at (7.5, 0) {$x_2$};
    \node [style=whitecircle,scale=0.75] (31) at (10, 0) {$a_1$};
    \node [style=whitecircle,scale=0.75] (40) at (12.5, 0) {$y_3$};

    \node [style=whitecircle,scale=0.75] (37) at (7.5, -3) {$x_3$};
    \node [style=whitecircle,scale=0.75] (34) at (10,-3) {$a_4$}; 
    \node [style=whitecircle,scale=0.75] (38) at (12.5, -3) {$y_1$};
\end{pgfonlayer}

\begin{pgfonlayer}{edgelayer}
	\draw [style=REDARROW] (35) to (32);
        \draw [style=REDARROW] (32) to (33);
		\draw [style=REDARROW] (33) to (39);
        
        \draw[style=BLUEARROW] (36) to (31);
        \draw [style=BLUEARROW] (31) to (40);
        
        \draw[style=GREENARROW] (37) to (34);
		\draw [style=GREENARROW] (34) to (38);
		
\end{pgfonlayer}

\end{tikzpicture}
\caption{An example illustration of (A) $G_1$, (B) the corresponding graph $J_1$ with a Hamilton cycle $H_1$, and (C) the vertex-disjoint paths $Q_1^1, \ldots, Q_k^1$ spanning $G_1$ (with $k=3$ in this case) corresponding to $H_1$. In this case  $\sigma=(231)$, i.e. the cyclic permutation that sends 1 to 2, 2 to 3, and 3 to 1.}
\label{fig:hamilton-to-paths}
\end{figure}
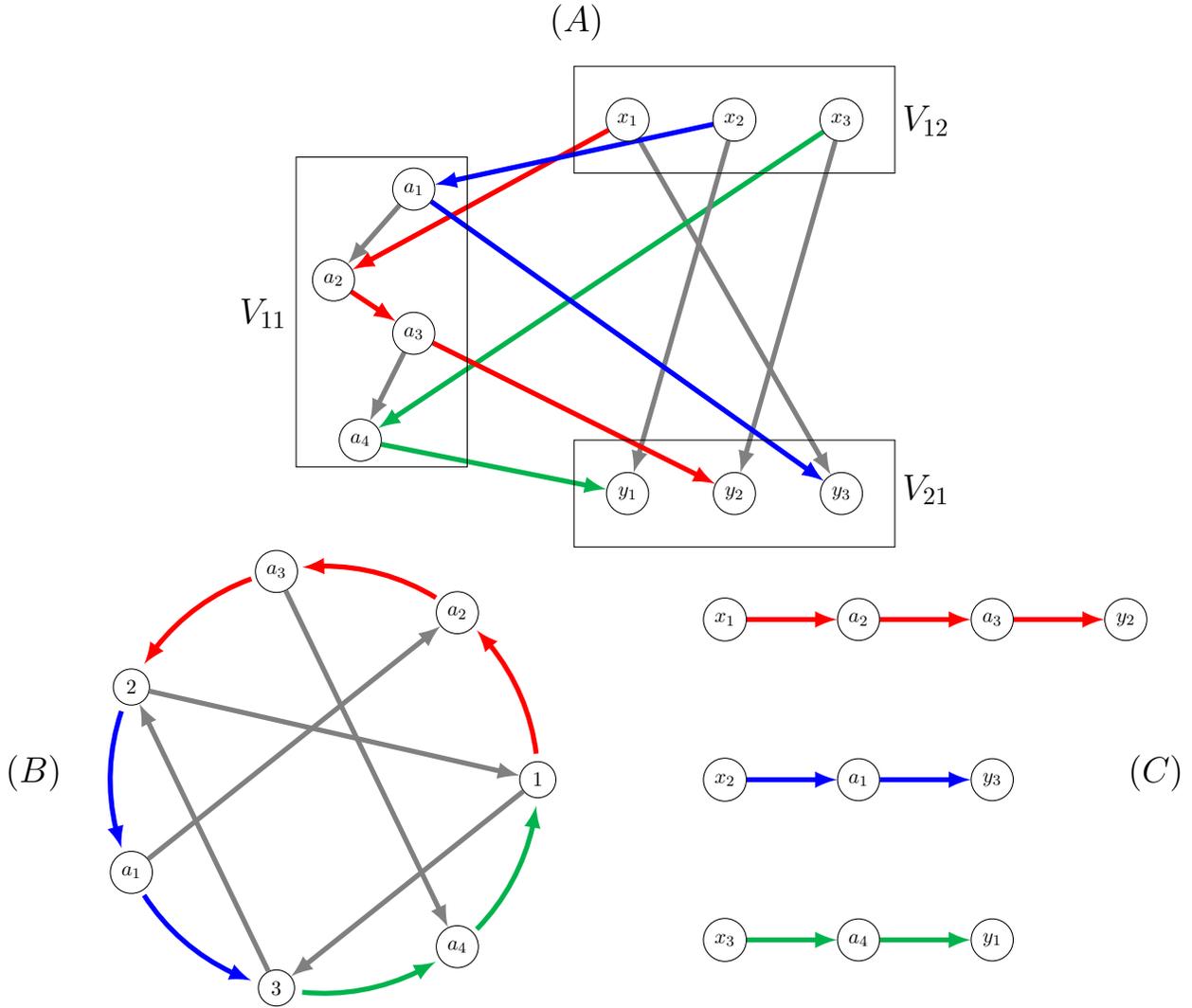

\no
{\bf Constructing the Hamilton cycle for unbalanced partitions} - We have seen how to find the Hamilton cycle when $|V_{12}| = |V_{21}|$. If instead we only have (by (a)) that $|V_{12}| \approx |V_{21}|$, then we will find vertex-disjoint paths $S_1, \ldots, S_{\ell}$ that use only bad edges (and only a relatively small number of bad edges) such that ``contracting'' these paths results in a slightly modified graph $G'$ with a slightly modified vertex partition $V'_{11}, V'_{12}, V'_{21}, V'_{22}$, which has essentially the same properties as before but also that $|V_{12}'| = |V_{21}'|$. 
Here $G'$ is not regular, but almost regular; this however is enough for us. 
So we can find a Hamilton cycle in $G'$ using the previous argument, and ``uncontracting'' the paths $S_1, \ldots, S_{\ell}$ gives a Hamilton cycle in~$G$.\footnote{For Theorem~\ref{thm:MAIN-1-3-CASE}, these paths are constructed directly in the proof of the theorem in Section~\ref{ch:mainproofs}, but in the more complicated case of Theorem~\ref{thm:MAIN-1-4-CASE}, they are constructed in Lemma~\ref{lem:9-partition-path-selection}.} \\

\no 
{\bf The case of regular oriented graphs} - 
For Theorem~\ref{thm:MAIN-1-4-CASE}, i.e. when $G$ is an $n$-vertex regular oriented graph with degree $d \geq (1/4 + \eps)n$, we start by applying the same argument as before. Recall that we construct digraphs $J_1$ and $J_2$ and wish to find Hamilton cycles in these digraphs. However, whereas before, we could guarantee that both $J_1$ and $J_2$ would be robust expanders, this time we find that (at most) one of them, say $J_2$ might not be. This is because $G$ and $J_i$ have lower degree, and so we cannot necessarily apply Lemma~\ref{lem:HigherThanHalfGivesRobustExpander}. It is not too hard to see that the $J_i$ are almost regular and so we can iterate our partition argument on $J_2$. In particular we can partition $V(J_2)$ into four parts $Z_{11}, Z_{12}, Z_{21}, Z_{22}$ that satisfy slightly modified forms of (a) and (b). Again if $|Z_{12}| = |Z_{21}|$, then we can create digraphs $K_1$ and $K_2$ such that Hamilton cycles in $K_1$ and $K_2$ lift to a Hamilton cycle in $J_2$ (just as Hamilton cycles in $J_1$ and $J_2$ lift to a Hamilton cycle in $G$).
This time the increase in density is enough to guarantee that both $K_1$ and $K_2$ are robust expanders, which gives the desired Hamilton cycle by Theorem~\ref{thm:RobustExpanderImpliesHamilton}. If $|Z_{12}| \not= |Z_{21}|$ then, as before, we need to construct paths whose contraction results in a modified graph with a modified partition that is balanced. In fact, we need to be able to find and contract paths in such a way that we simultaneously have $|V'_{12}| = |V'_{21}|$ and $|Z'_{12}| = |Z'_{21}|$. For this purpose, and generally for a cleaner and more transparent argument, rather than working with two iterations of the $4$-partition described earlier, we work equivalently with a $9$-partition of $V(G)$. The required paths are constructed in Lemma~\ref{lem:9-partition-path-selection}.

	\section{Partitions of regular digraphs and oriented graphs}
	\label{ch:partition}

We have seen that (essentially) any dense digraph that is a robust expander is Hamiltonian. If the digraph is not a robust expander, then we will see (Lemma~\ref{lem:structure_not_expander}) that the witness sets to this non-expansion naturally forms a partition of the vertices into $4$ parts. Throughout the paper we will be working with such partitions and their iterations. In this section, we introduce the language of partitions and establish some of their basic properties.

\begin{definition}\label{def:good-edges}
For a given digraph $G$ and $k\in \mathbb{N}$, a partition $\mathcal{P}_k=\{V_{ij}:i,j\in[k] \}$ of $V(G)$ is called a $k^2$-partition of $V(G)$ (we allow the sets $V_{ij}$ to be empty). The set of \emph{good edges with respect to $\mathcal{P}_k$} is defined as $$\mathcal{G}_k(\mathcal{P}_k,G):=\ds\bigcup_{i}E(V_{i*},V_{*i}),$$ where $V_{i*}:=\bigcup_{j}V_{ij}$ and $V_{*j}:=\bigcup_{i}V_{ij}$.
The set of \emph{bad edges with respect to $\mathcal{P}_k$} is defined as $$\mathcal{B}_k(\mathcal{P}_k,G):=E(G)-\mathcal{G}_k(\mathcal{P}_k,G)=\ds\bigcup_{i\neq j}E(V_{i*},V_{*j}).$$
We write $G_{ij}:=G[V_{i*}, V_{*j}]$. 
\end{definition}

\no 
Note that while we define $k^2$-partitions and prove properties for general $k$, in fact we only require the cases $k=2,3$.
For regular digraphs, we have a useful equality relating the sizes of different parts in a $k^2$-partition and the number of bad edges.

\begin{proposition}\label{prop:partition-regular-graphs}
Let $G$ be a $d$-regular digraph, $k\in \mathbb{N}$, and $\mathcal{P}_k=\{V_{ij}:i,j\in [k]\}$ be a $k^2$-partition of $V(G)$. Then, for all $i \in [k]$, we have $$d(|V_{i*}|-|V_{*i}|)= \ds\sum_{j\neq i} \left( e(G_{ij})-e(G_{ji})\right).$$
\end{proposition} 	

\begin{proof}
By considering outneighbours of the vertices in $V_{i*}$, we can write $$d|V_{i*}|=e(V_{i*},V_{*i})+\ds\sum_{j\neq i}e(V_{i*},V_{*j}).$$
Similarly, by considering the inneighbours of the vertices in $V_{*i}$, we have $$d|V_{*i}|=e(V_{i*},V_{*i})+\ds\sum_{j\neq i}e(V_{j*},V_{*i}).$$
By subtracting the second equality from the first one, the result follows.
\end{proof}
	
\no If the number of bad edges is small compared to $E(G)$, then Proposition~\ref{prop:partition-regular-graphs} implies that $V_{i*}$ and $V_{*i}$ are similar in size.

\begin{corollary}\label{cor:Y-and-Z-equal-size}
	Let $k \in \mathbb{N}$ and $\gamma$ be a positive constant. 
	Let $G$ be a $d$-regular digraph on $n$ vertices, and $\mathcal{P}_k=\{V_{ij}:i,j\in [k]\}$ be a $k^2$-partition of $V(G)$. If $|\mathcal{B}_k(\mathcal{P},G)|\leq\gamma n^2$, then we have $\big| |V_{i*}|-|V_{*i}| \big|\leq \gamma n^2/d$ for all $i \in [k]$.
\end{corollary}

\begin{proof}
Fix $i \in [k]$. We have 
$$\Big|\ds\sum_{j\neq i}\left(e(V_{i*},V_{*j})-e(V_{j*},V_{*i})\right)\Big|\leq \ds\sum_{j\neq i}\left(e(V_{i*},V_{*j})+e(V_{j*},V_{*i})\right)\leq|\mathcal{B}_k(\mathcal{P},G)|\leq \gamma n^2.$$ Hence, by Proposition~\ref{prop:partition-regular-graphs}, we know $d\big||V_{i*}|-|V_{*i}|\big| \leq \gamma n^2$, so the result follows.
\end{proof}

\no We will be especially interested in partitions with a small number of bad edges and where certain parts are not too small.

\begin{definition}\label{def:name-good-partition}
For a given digraph $G$ on $n$ vertices and positive constants $\gamma$, $\tau$, and $k\in \mathbb{N}$, we say a $k^2$-partition $\mathcal{P}_k=\{V_{ij}:i,j\in [k]\}$ of $V(G)$ is a \emph{$(k^2,\tau,\gamma)$-partition} if the following hold:
\begin{center}
	$|\mathcal{B}_k(\mathcal{P}_k,G)|\leq \gamma n^2$ and $|V_{i*}|, |V_{*j}|\geq\tau n$ for all $i,j \in [k]$.
\end{center}
\end{definition}

\begin{remark}\label{rem:Y-and-Z-equal-size}
In general, the constants $\gamma$ and $\tau$ are taken to satisfy $1/n\ll \gamma\ll \tau\ll 1$. When working with regular graphs, we sometimes implicitly take the conclusion of Corollary~$\ref{cor:Y-and-Z-equal-size}$ as a property of a $(k^2,\tau,\gamma)$-partition.
\end{remark}

\no Next, we show that every almost regular digraph which is dense and not a robust $(\nu,\tau)$-outexpander admits a $(4,\tau/2,4\nu)$-partition.

\begin{lemma}\label{lem:structure_not_expander}
	Let $1/n\ll\nu\ll\tau\ll\alpha\ll1$, and $G$ be a digraph on $n$ vertices such that $e(G) \ge (\alpha -\nu) n^2 $ and $\Delta^0(G) \le \alpha n$.
	If $G$ is not a robust $(\nu,\tau)$-outexpander, then $G$ admits a $(4,\tau,4\nu)$-partition. 
\end{lemma}

\begin{proof}
	Assume $G$ is not a robust $(\nu,\tau)$-outexpander.
	Then we can find a subset $S\subseteq V(G)$ such that $\tau n\leq |S|\leq (1-\tau)n$ and $|\text{RN}_{\nu}^{+}(S)|<|S|+\nu n$. 
	Let us define $V_{11}=S\cap \text{RN}_{\nu}^{+}(S)$, $V_{12}=S-\text{RN}_{\nu}^{+}(S)$, $V_{21}=\text{RN}_{\nu}^{+}(S)-S$, and $V_{22}=V(G)-(S\cup \text{RN}_{\nu}^{+}(S))$. Therefore $V_{1*}=S$ and $V_{*1}=\text{RN}_{\nu}^{+}(S)$.
	Note that $\mathcal{P}_2=\{V_{ij}:i,j \in [2]\}$ is a $4$-partition of $V(G)$.
	Moreover, since $\tau n\leq |S|\leq (1-\tau)n$, we have $|V_{1*}|,|V_{2*}|\geq \tau n$.
	\\
	
	\no
	We first show that $|\mathcal{B}_2(\mathcal{P}_2,G)| \leq 4\nu n^2$.
	By the definition of $\text{RN}_{\nu}^{+}(S)$, we know that every vertex in $V_{*2}$ has fewer than $\nu n$ inneighbours from $V_{1*}$. Thus, we have
	\begin{align}
	    e(V_{1*},V_{*2}) \le \nu n^2 \label{eq:first}
	\end{align}
	and 
	\begin{align}
	    e(V_{1*},V_{*1}) & = e(V_{1*},V(G)) - e(V_{1*},V_{*2}) \ge e(V_{1*},V(G)) - \nu n^2 
	    = e(G) - e(V_{2*}, V(G)) - \nu n^2  \nonumber \\
	    &\ge (\alpha -\nu) n^2 - \alpha n |V_{2*}| - \nu n^2 = \alpha n |V_{1*}| - 2 \nu n^2 \label{eq:third}
	    .
	\end{align}
	Since $|V_{*1}| = |\text{RN}_{\nu}^{+}(S)|<|S|+\nu n = |V_{1*}|+\nu n$, we have
	\begin{align*}
		e(V(G), V_{*1}) \le  \alpha n |V_{*1}| <  (|V_{1*}|+\nu n) \alpha n \le \alpha n |V_{1*}| + \nu n^2.
	\end{align*} 
	Thus, together with~\eqref{eq:third}, we have
	\begin{align*}
	    e(V_{2*}, V_{*1}) = e(V(G), V_{*1}) - e(V_{1*}, V_{*1}) \le 3 \nu n^2.
	\end{align*}
	Therefore \eqref{eq:first} implies that $ |\mathcal{B}_2(\mathcal{P}_2,G)| = e(V_{1*},V_{*2}) + e(V_{2*},V_{*1}) \leq 4\nu n^2$.\\
	
\noindent We now bound $|V_{*1}|$ and $|V_{*2}|$ from below. 
Let $T$ be the set of vertices with outdegree at most $(\alpha - \sqrt{\nu})n$.
	Then as $\Delta^0(G) \le \alpha n$
	\begin{align*}
	  (\alpha - \nu )n^2 \le e(G) 
	  \le (\alpha - \sqrt{\nu})n|T| + \alpha n (n-|T|)
	  = \alpha n^2 - \sqrt{\nu}n |T|,
	\end{align*}
	which implies that $|T| \le \sqrt{\nu}n$. 
	For $\{i,j\} = [2]$, recall that $|V_{i*}| \ge \tau n$ and so we have
\begin{align*}
    ( |V_{*i}| + 4 \nu /\tau) |V_{i*}|
    &\ge |V_{i*}||V_{*i}| + 4 \nu n^2
    \ge e(V_{i*} ,V_{*i}) + |\mathcal{B}_2(\mathcal{P}_2,G)|
    \ge e(V_{i*} ,V_{*i}) + e(V_{i*},V_{*j}) 
    \\
    &= e(V_{i*},V(G) ) 
    \ge ( \alpha-\sqrt{\nu} ) n | V_{i*} \setminus T |
    \ge ( \alpha-\sqrt{\nu} ) n | V_{i*} | /2
 \end{align*}
As a result, we obtain $|V_{*i}| \geq ( \alpha-\sqrt{\nu})n/2 - 4 \nu /\tau \geq \tau n$, so the result follows.
\end{proof}

\no One can construct an $(\ell^2,\tau,\gamma)$-partition of $G$ from a $(k^2,\tau,\gamma)$-partition of $G$ for $\ell\leq k$.

\begin{proposition}\label{prop:partition-9-to-4}
Let $G$ be a digraph with a $(k^2,\tau,\gamma)$-partition $\mathcal{P}_k=\{V_{ij}:i,j\in [k]\}$.
Let $\{I_1,I_2,\ldots,I_\ell\}$ be a partition of $[k]$ with $I_t\neq \emptyset$ for all $t\in[\ell]$. For $i',j'\in[\ell]$, let $W_{i'j'}=\bigcup_{i\in I_{i'},\,j\in I_{j'}}V_{ij}$. Then, $\mathcal{P}_\ell=\{W_{i'j'}:i',j'\in[\ell]\}$ is an $(\ell^2,\tau,\gamma)$-partition of $G$.
\end{proposition}

\begin{proof}
Let $n = |G|$.
For $i'\in [\ell]$, note that $$W_{i'*}=\ds\bigcup_{j'\in[\ell]}W_{i'j'}=\ds\bigcup_{i\in I_{i'},\, j\in[k]}V_{ij}=\ds\bigcup_{i\in I_{i'}}V_{i*}$$ and so $|W_{i'*}|\geq \tau n$. Similarly, we have $|W_{*j'}|\geq \tau n$ for all $j'\in[\ell]$. Moreover, note that
\begin{align*}
|\mathcal{B}_\ell(\mathcal{P}_\ell,G)|&=\sum_{i',j'\in[\ell]\colon  i'\neq j'}e(W_{i'*},W_{*j'})= \sum_{i',j'\in[\ell]\colon  i'\neq j'  }\,\,\sum_{i\in I_{i'},\,j\in I_{j'}}e(V_{i*},V_{*j})\\
&\leq  \sum_{i,j\in[k]\colon  i \neq j}e(V_{i*},V_{*j})=|\mathcal{B}_k(\mathcal{P}_k,G)|,
\end{align*}
so the result follows.  
\end{proof}

\no Next, we show that if a regular digraph is dense and admits a $(k^2,\tau,\gamma)$-partition, then certain unions of parts have size at least roughly the degree of the digraph.

\begin{proposition}\label{prop:dense-regular-size}
Let $1/n\ll\gamma\ll\tau\ll\eps\ll\alpha\ll1$, $k\in\mathbb{N}$, and $G$ be a $d$-regular digraph on $n$ vertices where $d\geq (\alpha+\eps) n$. 
Suppose that $G$ has a $(k^2,\tau,\gamma)$-partition $\mathcal{P}_k=\{V_{ij}:i,j\in [k]\}$. Then we have $|V_{i*}|,|V_{*i}|\geq d-\eps n/2$ for all $i \in [k]$.
In particular, $\mathcal{P}_k$ is a $(k^2,\alpha+\eps/2,\gamma)$-partition for $G$.
\end{proposition}

\begin{proof}
Let $i\in[k]$. By looking at the outneighbours of $V_{i*}$, we have
$$|V_{*i}|\geq \dfrac{e(V_{i*},V_{*i})}{|V_{i*}|}= \dfrac{d|V_{i*}|-\sum_{j\neq i,\,j\in[k]}e(V_{i*},V_{*j})}{|V_{i*}|} \geq d-\dfrac{|\mathcal{B}_k(\mathcal{P}_k,G)|}{|V_{i*}|}\geq  d-\dfrac{\gamma n^2}{|V_{i*}|} \geq d-\dfrac{\gamma n}{\tau}\geq d-\eps n/2$$ since $|\mathcal{B}_k(\mathcal{P}_k,G)|\leq \gamma n^2$, $|V_{i*}|\geq \tau n$ and $\gamma \ll \tau \ll \eps$. Similarly, we have $|V_{i*}|\geq d-\eps n/2$. 
\end{proof}

\no If a $(k^2,\tau,\gamma)$-partition has the minimum possible number of bad edges among all $(k^2,\tau,\gamma)$-partitions of a digraph, then we give it a special name.

\begin{definition}\label{def:extremal-4-partition}
Let $1/n\ll \gamma\ll \tau\ll 1$, $k\in \mathbb{N}$, and
 $G$ be a digraph on $n$ vertices.  A $(k^2,\tau,\gamma)$-partition $\mathcal{P}_k=\{V_{ij}:i,j\in [k]\}$ of $V(G)$ is called an \emph{extremal $(k^2,\tau,\gamma)$-partition} if $\mathcal{B}_k(\mathcal{P}_k,G)\leq \mathcal{B}_k(\mathcal{P}_k',G)$ for all $(k^2,\tau,\gamma)$-partitions $\mathcal{P}_k'$ of $V(G)$.
\end{definition}

\no We establish some useful degree conditions for extremal $(k^2,\tau,\gamma)$-partitions of dense regular digraphs.  

\begin{proposition}\label{prop:degree-property-extremal}
Let $1/n\ll\gamma\ll\tau\ll\alpha\ll1$, $k\in \mathbb{N}$, and $G$ be a $d$-regular digraph on $n$ vertices with $d\geq \alpha  n$ and an extremal $(k^2,\tau,\gamma)$-partition $\mathcal{P}_k=\{V_{ij}:i,j\in [k]\}$.
Then, for all $i,j \in [k]$ and $w\in V_{ij}$, we have $d_{V_{*i'}}^{+}(w)\leq d_{V_{*i}}^{+}(w)$ and $d_{V_{j'*}}^{-}(w)\leq d_{V_{j*}}^{-}(w)$ for all $ i',j'\in[k]$. In particular, we have $d_{V_{j'*}}^{-}(w), d_{V_{*i'}}^{+}(w)\leq d/2$ for all $i'\neq i$ and $j'\neq j$, and $d_{\mathcal{G}_k(\mathcal{P}_k,G)}^{+}(v),d_{\mathcal{G}_k(\mathcal{P}_k,G)}^{-}(v)\geq d/k$ for all $v\in V(G)$.
\end{proposition}

\begin{proof}
Let $\eps$ be a constant such that $\tau \ll \eps \ll \alpha$.
Let $\alpha' = \alpha - \eps$.
Suppose the contrary and without loss of generality that there exists $w\in V_{ij}$ and $a\in[k]$ such that $d_{V_{*a}}^{+}(w)>d_{V_{*i}}^{+}(w)$. Let $V_{ij}'=V_{ij}\backslash\{w\}$, $V_{aj}'=V_{aj}\cup \{w\}$, and $V_{i'j'}'=V_{i'j'}$ for all $(i',j')\in[k]\times[k]\backslash\{(i,j),(a,j)\}$. Let $\mathcal{P}_{k}'=\{V_{i'j'}':i',j'\in[k]\}$. By Proposition~\ref{prop:dense-regular-size}, $$|V_{i*}'|=|V_{i*}|-1\geq d-\eps n/2-1\geq \tau n$$ since $\tau \ll \eps\ll \alpha$. Similarly, we have $|V_{*j}'|\geq \tau n$. Moreover, for all $i'\neq i$ and $j'\neq j$, we know $|V_{i'*}'|\geq |V_{i'*}|\geq \tau n$ and $|V_{*j'}'|\geq |V_{*j'}|\geq \tau n$. On the other hand, we obtain 
$$|\mathcal{B}_k(\mathcal{P}_k',G)|=|\mathcal{B}_k(\mathcal{P}_k,G)|-d_{V_{*a}}^{+}(w)+d_{V_{*i}}^{+}(w)<|\mathcal{B}_k(\mathcal{P}_k,G)|.$$
Hence $\mathcal{P}_{k}'$ is a $(k^2,\tau,\gamma)$-partition of $G$ having fewer bad edges than the extremal $(k^2,\tau,\gamma)$-partition $\mathcal{P}_k$, which is a contradiction. As a result, for all $1\leq i',j'\leq k$, we have $d_{V_{*i'}}^{+}(w)\leq d_{V_{*i}}^{+}(w)$ and $d_{V_{j'*}}^{-}(w)\leq d_{V_{j*}}^{-}(w)$. The rest of the proof is immediate.
\end{proof}

\no For any dense regular oriented graph, we show that certain unions of sets in a $(k^2,\tau,\gamma)$-partition have strictly positive size.

\begin{proposition}\label{prop:quarterleadsnonempty}
Let $1/n<\gamma\ll\tau\ll\eps <1$ be constants, $k\in \mathbb{N}$, and $G$ be a $d$-regular oriented graph on $n$ vertices with $d \geq (1/4+\eps)n$. 
Suppose that $G$ has a $(k^2,\tau,\gamma)$-partition $\mathcal{P}_k=\{V_{ij}:i,j\in [k]\}$.
Then, for $i \in [k]$, we have $$\ds\big|\bigcup_{j\neq i}V_{ij}\big|,\ds\big|\bigcup_{j\neq i}V_{ji}\big|\geq \tau n.$$
\end{proposition}

\begin{proof}
First suppose that $k=2$. Without loss of generality, assume $|V_{11}|\leq |V_{22}|$, which gives $d-{|V_{11}|}/{2}\geq d-{n}/{4}\geq \eps n$. By Corollary~\ref{cor:Y-and-Z-equal-size}, we know that $|V_{12}|-|V_{21}|\leq \gamma n^2/d\leq \tau n$. Hence, it suffices to show that $|V_{12}|\geq 2\tau n$. By Proposition~\ref{prop:dense-regular-size}, we have $|V_{11}|+|V_{12}|=|V_{1*}|\geq (1/4+{\eps}/{2})n$. Since $\tau\ll \eps$, we may assume that $|V_{11}|\geq n/4$. Then, since $G$ is oriented and $d$-regular, we can write 
\begin{align*}
d|V_{11}|& =e(V(G),V_{11})=e(V_{11},V_{11})+e(V_{12},V_{11})+e(V_{2*},V_{11})\leq {|V_{11}|^2}/{2}+|V_{12}||V_{11}|+\gamma n^2\\
&\leq |V_{11}|\cdot ( {|V_{11}|}/2+|V_{12}|+4\gamma n).
\end{align*}
This implies $|V_{12}|\geq d-{|V_{11}|}/{2}-4\gamma n\geq (\eps-4\gamma)n\geq 2\tau n$ as required.
\\

\no Now, fix $k\geq 3$ and define $\mathcal{W}_i=\{W^{i}_{ab}:a,b\in[2]\}$ for all $i \in [k]$ where $$W^{i}_{11}=V_{ii},\,\, W^{i}_{12}=\ds\bigcup_{j\neq i}V_{ij},\,\, W^{i}_{21}=\ds\bigcup_{j\neq i}V_{ji},\,\, W^{i}_{22}=\ds\bigcup_{a,b\neq i}V_{ab}.$$ Notice that $\mathcal{W}_i$ is a $(4,\tau,\gamma)$-partition by Proposition~\ref{prop:partition-9-to-4}, so we get $|W^{i}_{12}|,|W^{i}_{21}|\geq\tau n$ from the case $k=2$. Then, we obtain $$\ds\Big|\bigcup_{j\neq i}V_{ij}\Big|=|W^{i}_{12}|\geq\tau n\,\text{ and }\,\ds\Big|\bigcup_{j\neq i}V_{ji}\Big|=|W^{i}_{21}|\geq\tau n,$$ so the result follows for any $k$.  
\end{proof}

\section{Balancing partitions}
\label{ch:balance}

Let $G$ be a regular digraph  or oriented graph and suppose $\mathcal{P}_k$ is a $(k^2,\tau,\gamma)$-partition of $G$ that is ``not balanced'', in the sense that $|V_{i*}| \not= |V_{*i}|$ for some $i \in [k] $. Then, Proposition~\ref{prop:partition-regular-graphs} implies that any Hamilton cycle $C$ must contain a number of bad edges (i.e.\ edges from $\mathcal{B}_k(\mathcal{P}_k,G)$) that depends on the extent of the ``imbalance'' of $\mathcal{P}_k$. Since $\mathcal{B}_k(\mathcal{P}_k,G)$ is small (at most $\gamma n^2$ edges), when constructing a Hamilton cycle of~$G$, it is necessary to first pick the edges of $\mathcal{B}_k(\mathcal{P}_k,G)$ that will be in~$C$. Let us write $\mathcal{Q}$ for the bad edges in our target Hamilton cycle, and note that $\mathcal{Q}$ is a path system. \\

\no
By Proposition~\ref{prop:partition-regular-graphs} (applied with $d=1$ and $G=C$), we must ensure that~$\mathcal{Q}$ satisfies that for all $i \in [k]$,
\begin{align*}
 \sum_{j\neq i}| E(\mathcal{Q})\cap E(V_{i*},V_{*j}) | -  \sum_{j\neq i}| E(\mathcal{Q})\cap E(V_{j*},V_{*i}) |=|V_{i*}|-|V_{*i}|.
\end{align*}
A naive approach to construct $\mathcal{Q}$ is to take a suitable size matching in each of~$G_{ij}$ for $i \ne j$, where as before~$G_{ij}=G[V_{i*},V_{*j}]$.
However, the union of these matchings may not be a path system since it might contain cycles or might satisfy $\Delta^0(\mathcal{Q}) \ge 2$.
The main purpose of this section is to adapt the naive approach to construct $\mathcal{Q}$; see Lemma~\ref{lem:BALANCING-PARTITION}. \\

\no Our first goal is to show that given several edge-disjoint subdigraphs of some  given digraph, we are able to pick a relatively large path system from each subdigraph such that the union of these path systems does not contain a directed cycle; this is Lemma~\ref{lem:CYCLE_FREE}. 
The first two lemmas below are technical results needed to prove this.

	\begin{lemma}\label{lem:MATHCING_LEMMA_I}
		Let $G$ be a digraph with $\Delta^{0}(G)\leq d$. Let $0<\theta<1$, and define the sets $W^{+}=\{w\in G:d^{+}(w)\geq \theta d\}$ and $W^{-}=\{w\in G:d^{-}(w)\geq \theta d\}$. Then, there exists a matching $M$ satisfying 
		\begin{itemize}
		\item[\rm(i)] $4\theta e(M)+|W^{+}|+|W^{-}|\geq e(G)/d $,
		\item[\rm(ii)] $x\notin W^{+}$ and $y\notin W^{-}$ for all $xy\in E(M)$,
		\item[\rm(iii)] $e(M)\leq e(G)/\theta d$.
		\end{itemize}
	\end{lemma}
	
	\begin{proof}
		If $\theta d<1$, then we obtain $W^{+}=\{w\in G:d^{+}(w)\geq 1\}$. Then, we have $x\in W^{+}$ for any $xy\in E(G)$, which, in particular, implies $d|W^{+}|\geq e(G)$. Therefore, we can set $M$ to be empty in that case. Hence, we may assume $\theta d\geq 1$.
		Let $H$ be the multigraph obtained from $G$ by deleting all the edges $ab$ with either $a\in W^{+}$ or $b\in W^{-}$, and by making all the edges undirected. Note that we have $\Delta(H)+\mu(H) \leq 2\theta d+2$ and
		\begin{align}\label{eqn:deleting-some-edges}
		e(H)\geq e(G)-d(|W^{+}|+|W^{-}|).	
		\end{align}
	Then, by Theorem \ref{thm:Vizing} (Vizing's theorem for multigraphs), there exists a matching $M_1$ in $H$ of size at least $e(H)/(2\theta d+2)$.
	Moreover, we can assume that $e(M_1)\leq e(H)/\theta d$ because otherwise we can remove some edges from $M_1$. Let $M$ be the corresponding matching in $G$. Clearly (ii) holds. By using $\theta d\geq 1$, we obtain $$\dfrac{e(G)}{\theta d}\geq \dfrac{e(H)}{\theta d}\geq e(M)\geq \dfrac{e(H)}{2\theta d+2}\geq \dfrac{e(H)}{4\theta d},$$ so (iii) holds. Hence, together with \eqref{eqn:deleting-some-edges}, we have $4\theta e(M)+|W^{+}|+|W^{-}|\geq e(G)/d$, proving~(i). 
	\end{proof}
	
\no Now, given some matchings in a graph, we show that one can pick a significant number of edges from each matching such that all the chosen edges form a matching. 	

\begin{lemma}\label{lem:MATCHING_LEMMA_II}
	Let $k,r\in\mathbb{N}$ and $M_1,M_2,\ldots,M_k$ be matchings with $\Delta\left(\bigcup_{i \in [k]} M_i\right)\leq r$. Suppose $e(M_i)>2(r^3+r)^2\ln k$ for all $i \in [k]$. Then, there exists a matching $H\subseteq \bigcup_{i \in [k]}M_i$ with $|E(H)\cap M_i|\geq e(M_i)/(r^2+1)$ for all $i \in [k]$. 
\end{lemma}
	
	\begin{proof}
		Letting $G=\bigcup_{i \in [k]}M_i$, we have $\Delta(G)\leq r$. We mark edges of $G$ randomly as follows. For each vertex $v\in G$, pick an edge incident to $v$ uniformly at random and mark all other edges incident to $v$. Do this independently for every vertex $v$ (so some edges may be marked twice). Then, let $H$ be the graph where all the marked edges are deleted. Note that $H$ is a matching. We now show that $H$ satisfies the desired property with positive probability. 
  Observe that the probability of an edge $uv$ surviving into $H$ is at least $1/r^2$ because the probability of $uv$ being marked due to~$u$ is at least $1/r$, and independently the probability of $uv$ being marked due to $v$ is at least $1/r$. Moreover, these events are independent for vertex-disjoint edges.
  Now, for any $i \in [k]$, let $X_i={\rm Bin}(e(M_i),1/r^2)$. Since $M_i$ is a matching, we have 
		\begin{align*}
		\mathbb{P}\left(|E(H)\cap M_i|\leq \dfrac{e(M_i)}{r^2+1}\right)\leq \mathbb{P}\left(X_i\leq \dfrac{e(M_i)}{r^2+1}\right).    
		\end{align*}
		Note that $\mathbb{E}[X_i]={e(M_i)}/{r^2}$. Hence, by Theorem \ref{thm:Chernoff} (Chernoff bound), we obtain 
		\begin{align*}
		\mathbb{P}\left(X_i\leq \dfrac{e(M_i)}{r^2+1}\right)= \mathbb{P}\left(X_i\leq \dfrac{r^2}{r^2+1}\cdot \mathbb{E}[X_i]\right)\leq \exp{\left(-\frac{\mathbb{E}[X_i]}{2(r^2+1)^2}\right)}=  \exp{\left(\frac{-e(M_i)}{2(r^3+r)^2}\right)}.
		\end{align*}
		
		\noindent Then, by using $e(M_i)> 2(r^3+r)^2\ln k$, we obtain $\mathbb{P}\left(|E(H)\cap M_i|\leq \dfrac{e(M_i)}{r^2+1}\right)< \dfrac{1}{k}$ for each $i \in [k]$. Hence, by the union bound, we have 
		\begin{align*}
		\mathbb{P}\left(|E(H)\cap M_i|\geq \dfrac{e(M_i)}{r^2+1}\text{ for all }i \in [k]\right)>0.    
		\end{align*}

		\noindent Therefore, there exists a matching $H\subseteq \bigcup_{i \in [k]}M_i$ with $|E(H)\cap M_i|\geq \dfrac{e(M_i)}{r^2+1}$ for all $i \in [k]$. 
		\end{proof}
	
	\no By using Lemmas~\ref{lem:MATHCING_LEMMA_I} and~\ref{lem:MATCHING_LEMMA_II}, we will prove an edge selection lemma which will be used in the proof of Lemma~\ref{lem:9-partition-path-selection}.
	
	\begin{lemma}\label{lem:CYCLE_FREE}
	Let $k\in\mathbb{N}$ with $k\leq 10$, let $0<\gamma\ll\alpha<1$ be constants, and let $G$ be a digraph on $n$ vertices. Let $G_1,G_2,\ldots,G_k$ be pairwise edge-disjoint subgraphs of $G$ with $\sum_{i \in [k]}e(G_i)\leq \gamma n^2$ and $\Delta^{0}(G_i)\leq \alpha n$ for each $i\in[k]$. Then, each $G_i$ contains a path system $\mathcal{Q}_i$ such that $\bigcup_{i \in [k]}\mathcal{Q}_i$ is cycle-free and $e(\mathcal{Q}_i)\geq \left\lfloor e(G_i) / \alpha n \right\rfloor$ for all $i \in [k]$.	

	\end{lemma}

\begin{proof}
Since $\gamma\ll \alpha$, we can choose a constant $\theta$ with $\sqrt{8\gamma}/\alpha <\theta<1/\left(8\ln k (k^3+k)^2\right)$. Then, let us define the sets $$W_i^{+}=\{w\in V(G_i):d_{G_i}^{+}(w)\geq\alpha\theta n\}\text{ and }W_i^{-}=\{w\in V(G_i):d_{G_i}^{-}(w)\geq\alpha\theta n\}.$$ By Lemma~\ref{lem:MATHCING_LEMMA_I}, for each $i \in [k]$, we can find a matching $M_i$ in $G_i$ with 
\begin{align*}
4\theta e(M_i)+|W_i^{+}|+|W_i^{-}|\geq e(G_i)/\alpha n,\,\, M_i\subseteq G_i[V-W_i^{+},V-W_i^{-}],\,\, e(M_i)\leq e(G_i)/\alpha \theta n.
\end{align*}
For each $i$, we have either $|W_i^{+}|+|W_i^{-}|>(e(G_i)/\alpha n)-1$ or $4\theta e(M_i)\geq 1$. In the latter case, we have $e(M_i)\geq 2(k^3+k)^2\ln k$ due to the definition of $\theta$. Let $R$ be the set of indices $i \in [k]$ satisfying $4\theta e(M_i)\geq 1$. By applying Lemma~\ref{lem:MATCHING_LEMMA_II} for the matchings $M_i$ with $i\in R$, we find a matching $M\subseteq\bigcup_{i\in R}M_i$ such that $e(M\cap M_i)\geq e(M_i)/(k^2+1)$ for all $i\in R$. Therefore, we have
\begin{align*}
e(M\cap M_i)+|W_i^{+}|+|W_i^{-}|&\geq e(M_i)/(k^2+1) +|W_i^{+}|+|W_i^{-}|\\
&\geq 4\theta e(M_i)+|W_i^{+}|+|W_i^{-}|\geq e(G_i)/\alpha n
\end{align*}
for all $i\in R$. On the other hand, if $i\notin R$, we know $|W_i^{+}|+|W_i^{-}|>(e(G_i)/\alpha n)-1$, which, in particular implies $|W_i^{+}|+|W_i^{-}|\geq \lfloor e(G_i)/\alpha n\rfloor$. Write $N_i=M\cap M_i$ if $i\in R$, write $N=\bigcup_{i\in R}N_i$, and set $N_i=\emptyset$ if $i\notin R$. Thus, we obtain $e(N_i)+|W_i^{+}|+|W_i^{-}|\geq \lfloor e(G_i)/\alpha n\rfloor$ for all $i \in [k] $. By deleting edges in $N_i$ or removing vertices from $W_i^{+}\cup W_i^{-}$, we may assume 
\begin{align*}
e(N_i)+|W_i^{+}|+|W_i^{-}|= \lfloor e(G_i)/\alpha n\rfloor  \text{ for all $i \in [k] $.}
\end{align*}
Let us write $W=\bigcup_{i \in [k] }(W_i^{+}\cup W_i^{-})$. Note that 
\begin{align}
    |V(N)\cup W| & \le 
	\sum_{i \in [k]}\left(2 e (N_i) + |W_i^{+}|+|W_i^{-}|\right)
	\leq 2\cdot \sum_{i \in [k]} (e(G_i)/\alpha\theta n )
	\leq 2\gamma n/\alpha \theta. \label{eqn:bound-non-selected}
\end{align}
We now construct the desired path systems $\mathcal{Q}_1, \ldots \mathcal{Q}_k$ by induction.
Suppose we have found path systems $\mathcal{Q}_1,\ldots,\mathcal{Q}_j$ for some $0\leq j\leq k$ such that $N\cup \left(\bigcup_{i \in [j]}\mathcal{Q}_i\right)$ is cycle-free, and the following hold for all $i \in [j]$: 
\begin{itemize}
	\item[(i)] $N_i\subseteq \mathcal{Q}_i\subseteq G_i$,
	\item[(ii)] $e(\mathcal{Q}_i)= \lfloor e(G_i)/\alpha n \rfloor $.
\end{itemize}

\no If $j=k$, then we are done. If $j<k$, then we construct $\mathcal{Q}_{j+1}$ as follows. First, we define $U=\left(V(N)\cup W\right)\cup V\left(\bigcup_{i \in [j]}\mathcal{Q}_i\right)$. By using \eqref{eqn:bound-non-selected}, we have 
$$|U|\leq 2\gamma n/\alpha\theta+2\sum_{i \in[j] }e(\mathcal{Q}_i)\leq (2\gamma n/\alpha\theta) +(2\gamma n/\alpha).$$
\no We construct the undirected bipartite graph $\mathcal{B}$ with bipartition $(A,B)$ as follows. Let $B=V(G)-U$, and let $A$ be the disjoint union of $W_{j+1}^{+}$ and $W_{j+1}^{-}$. We add the edge $ab$ for each $a\in W_{j+1}^{+}$ and $b\in B$ if $b\in N_G^{+}(a)$, and add the edge $ab$ for each $a\in W_{j+1}^{-}$ and $b\in B$ if $a\in N_{G}^{-}(b)$. Due to the choice of $\theta$, we have
\begin{align*}
d_{\mathcal{B}}(a)
\geq \alpha\theta n- (2\gamma n/\alpha\theta)-(2\gamma n/\alpha)\geq \gamma n/\alpha \geq e(G_{j+1})/\alpha n\geq |W_{j+1}^{+}|+|W_{j+1}^{-}|\geq |A|.
\end{align*}
Therefore, we can greedily pick a matching in $\mathcal{B}$ that covers $A$. Note that the corresponding edges in $G$ with respect to this matching give a path system $\mathcal{Q}_{j+1}'$ in $G_{j+1}$ containing paths of length one or two with $e(\mathcal{Q}_{j+1}')=|W_{j+1}^{+}|+|W_{j+1}^{-}|$ and $E(\mathcal{Q}_{j+1}')\cap E\left(N  \cup \bigcup_{i \in [j]}\mathcal{Q}_i\right)=\emptyset$. 
Moreover, each edge in $\mathcal{Q}_{j+1}'$ will contain a vertex in $W_{j+1}^{+}\cup W_{j+1}^{-}$ and one unique vertex not in~$U$.
Since $x\notin W_{j+1}^{+}$ and $y\notin W_{j+1}^{-}$ for all $xy\in N_{j+1}$, we can add $N_{j+1}$ into $\mathcal{Q}_{j+1}'$ to obtain another path system $\mathcal{Q}_{j+1}$ in $G_{j+1}$ with $e(\mathcal{Q}_{j+1})=\lfloor e(G_{j+1})/\alpha n \rfloor$.\\ 

\no Finally, suppose $N\cup \left(\bigcup_{i \in [j+1]}\mathcal{Q}_i\right)=N\cup \left(\bigcup_{i \in [j]}\mathcal{Q}_i\right)\cup \mathcal{Q}_{j+1}$ has a cycle $C$. Since $N\cup \left(\bigcup_{i \in [j]}\mathcal{Q}_i\right)$ has no cycle, $C$ contains an edge~$e$ in $\mathcal{Q}_{j+1}'$.
However, $e$ contains a unique vertex~$x$ not in~$U$, that is, $x$ has (total) degree~$1$ in~$N\cup \left(\bigcup_{i \in [j+1]}\mathcal{Q}_i\right)$, a contradiction. This completes the inductive construction of the $\mathcal{Q}_i$ and the proof of the lemma. 
As a result, $N\cup \left(\bigcup_{i \in [j+1]}\mathcal{Q}_i\right)$ is cycle-free, and we are done.  
\end{proof}

\no
Suppose $G$ is an oriented graph and consider a $9$-partition $\{V_{ij}:i,j \in [3]\}$ of $V(G)$. For $i,j \in [3]$, $i\neq j$, we say a path system $\mathcal{Q}$ is \emph{type-$ij$} if $E(\mathcal{Q})\subseteq E(V_{i*},V_{*j})$. Our next lemma describes the structure of the graph which is the union of several path systems that are of different types. First some further notation. \\

\no 
We denote the set of all type-$ij$ path systems by $\mathcal{Q}(i,j)$. Let $\mathscr{S}\subset \bigcup_{i\neq j}\mathcal{Q}(i,j)$ be a set consisting of three path systems of different types. We say $\mathscr{S}$ is a \emph{symmetric 3-set} if either $|\mathscr{S}\cap \mathcal{Q}(1,2)|=|\mathscr{S}\cap \mathcal{Q}(2,3)|=|\mathscr{S}\cap \mathcal{Q}(3,1)|=1$ or $|\mathscr{S}\cap \mathcal{Q}(2,1)|=|\mathscr{S}\cap \mathcal{Q}(3,2)|=|\mathscr{S}\cap \mathcal{Q}(1,3)|=1$. Otherwise, we say $\mathscr{S}$ is an \emph{anti-symmetric 3-set}. For an anti-symmetric 3-set $\mathscr{S}$, if $|\mathscr{S}\cap (\mathcal{Q}(1,2)\cup \mathcal{Q}(2,3)\cup \mathcal{Q}(3,1))|=2$ and $|\mathscr{S}\cap (\mathcal{Q}(2,1)\cup \mathcal{Q}(3,2)\cup \mathcal{Q}(1,3))|=1$, then we call the unique path system in $\mathscr{S}\cap (\mathcal{Q}(2,1)\cup \mathcal{Q}(3,2)\cup \mathcal{Q}(1,3))$ a \emph{special element} of $\mathscr{S}$. Similarly, if $|\mathscr{S}\cap (\mathcal{Q}(1,2)\cup \mathcal{Q}(2,3)\cup \mathcal{Q}(3,1))|=1$ and $|\mathscr{S}\cap (\mathcal{Q}(2,1)\cup \mathcal{Q}(3,2)\cup \mathcal{Q}(1,3))|=2$, then we call the unique path system in $\mathscr{S}\cap (\mathcal{Q}(1,2)\cup \mathcal{Q}(2,3)\cup \mathcal{Q}(3,1))$ as a \textit{special element} of $\mathscr{S}$. We will show that the graph induced by $\mathscr{S}$ has some structural properties if $\mathscr{S}$ is a symmetric or anti-symmetric 3-set. First, we need the definition of an anti-directed path. \\

\no
Let $G$ be a digraph.  A subgraph $P$ of $G$ is called an \emph{anti-directed path} in $G$ if its edges can be ordered as $E(P)=\{e_1,e_2,\ldots,e_k\}$ for some $k \in \mathbb{N}$ such that
	\begin{itemize}
	\item[\rm(i)] $(e_1,e_2,\ldots,e_k)$ induces an (undirected) path when we forgot the directions of the edges, and
	\item[\rm(ii)] $(e_1,e_2,\ldots,e_k)$ does not contain a directed path of length at least two.
	\end{itemize}
	An anti-directed path $P$ in $G$ is said to be \emph{maximal} if it is not entirely contained in any other anti-directed path.

\begin{lemma}\label{lem:ADE-BCF-Structure}
Let $\mathcal{P}_3=\{V_{ij}:i,j \in [3]\}$ be a $9$-partition of an oriented graph $G$. 
Let $\mathscr{S}$ be a set consisting of three path systems in $G$ of different types; thus either $\mathscr{S}$ is a symmetric $3$-set or an anti-symmetric $3$-set. Let $H$ be the graph induced by all the paths in $\mathscr{S}$.
If $\mathscr{S}$ is a symmetric $3$-set, then $H$ is the disjoint union of paths and cycles. If $\mathscr{S}$ is an anti-symmetric $3$-set with special element $S$, then $E(H)$ can be partitioned into maximal anti-directed paths $Q$ of length at most three with the following properties:
	\begin{itemize}
	\item[\rm(i)] If $Q$ is a maximal anti-directed path of length two, then $Q$ has a unique edge belonging to~$S$.
	\item[\rm(ii)] If $Q$ is a maximal anti-directed path of length three, then each edge of $Q$ belongs to a distinct path system in $\mathscr{S}$ where the middle edge belongs to~$S$.
	\end{itemize}

\end{lemma}

\begin{proof}
If $\mathscr{S}$ is a symmetric $3$-set, without loss of generality, assume $\mathscr{S}=\{\mathcal{Q}_{23},\mathcal{Q}_{31},\mathcal{Q}_{12}\}$ where the path system $\mathcal{Q}_{ij}$ is type-$ij$. Then, for any $x_{23}y_{23}\in E(\mathcal{Q}_{23})$, $x_{31}y_{31}\in E(\mathcal{Q}_{31})$, $x_{12}y_{12}\in E(\mathcal{Q}_{12})$, we obtain $x_{23}$, $x_{31}$, $x_{12}$ are all distinct since $x_{23}\in V_{2*}$, $x_{31}\in V_{3*}$, and $x_{12}\in V_{1*}$. Similarly, we have $y_{23}$, $y_{31}$, $y_{12}$ are all distinct. Therefore, for any vertex $v\in H$, we have $d^{+}(v), d^{-}(v)\leq 1$, which implies $H$ is the disjoint union of paths and cycles.\\

\no Let $\mathscr{S}$ be an anti-symmetric $3$-set. Without loss of generality, it is enough to examine the cases $\mathscr{S}=\{\mathcal{Q}_{23},\mathcal{Q}_{12},\mathcal{Q}_{13}\}$ and $\mathscr{S}=\{\mathcal{Q}_{23},\mathcal{Q}_{12},\mathcal{Q}_{21}\}$ where $\mathcal{Q}_{ij}$ is type-$ij$. Let us first examine the case $\mathscr{S}=\{\mathcal{Q}_{23},\mathcal{Q}_{12},\mathcal{Q}_{13}\}$. Note that $\mathcal{Q}_{13}$ is the special element of $\mathscr{S}$. For any $x_{23}y_{23}\in E(\mathcal{Q}_{23})$, $x_{13}y_{13}\in E(\mathcal{Q}_{13})$, $x_{12}y_{12}\in E(\mathcal{Q}_{12})$, we have $x_{23}\in V_{2*}$, $x_{13},x_{12}\in V_{1*}$, $y_{12}\in V_{*2}$, $y_{23},y_{13}\in V_{*3}$, which shows that $\Delta^{0}(H)\leq 2$. Thus, we can conclude that two different maximal anti-directed paths in $H$ are edge-disjoint, which implies every edge of $H$ lies in a unique maximal anti-directed path. Let $Q$ be a maximal anti-directed path in~$H$ of length at least two. 
Let $e_i,e_{i+1}$ be two consecutive edges in~$Q$.
It is easy to check that
\begin{align*}
(e_i,e_{i+1}) \in &\left(  E (\mathcal{Q}_{13}) \times   E (\mathcal{Q}_{12}) \right)
\cup 
   \left( E (\mathcal{Q}_{12}) \times E(\mathcal{Q}_{13}) \right)
\\ &\cup 
    \left( E(\mathcal{Q}_{13})    \times E (\mathcal{Q}_{23}) \right)
\cup 
\left( E(\mathcal{Q}_{23}) \times E (\mathcal{Q}_{13}) \right).
\end{align*}

\no Therefore, if $Q$ has two edges, property (i) follows. 
If $Q$ has three consecutive edges $e_i,e_{i+1},e_{i+2}$, then 
\begin{align*}
(e_i,e_{i+1},e_{i+2}) \in E(\mathcal{Q}_{12}) \times  E(\mathcal{Q}_{13}) \times E(\mathcal{Q}_{23}) \text{ or } 
(e_i,e_{i+1},e_{i+2}) \in E(\mathcal{Q}_{23} ) \times  E(\mathcal{Q}_{13}) \times E(\mathcal{Q}_{12}).
\end{align*}
This shows property (ii), and in particular that the middle of the three edges is in the special element~$\mathcal{Q}_{13}$. 
Finally, if $Q$ has at least four edges, take any four consecutive edges. 
These four edges contain two anti-directed paths of length three and the middle edge of each of these paths lies in $\mathcal{Q}_{13}$ from the argument above. Therefore we obtain two consecutive edges in $Q$ both in $\mathcal{Q}_{13}$, which is impossible since $\mathcal{Q}_{13}$ is a path system. \\

\no If $\mathscr{S}=\{\mathcal{Q}_{23},\mathcal{Q}_{12},\mathcal{Q}_{21}\}$, it is easy to check that we have $d_{H}^{+}(v)\leq 2$ and $d_{H}^{-}(v)\leq 1$ for all $v\in V(H)$. Note that $\mathcal{Q}_{21}$ is the special element of $\mathscr{S}$. As before, we see that $E(H)$ can be partitioned into maximal anti-directed paths since $\Delta^{0}(H)\leq 2$. Also, since each anti-directed path of length at least three has at least one vertex of indegree two, we have that all the maximal anti-directed paths in $H$ have at most two edges. Moreover, if $Q$ is a maximal anti-directed path of length two, say $e$ and $f$ are the edges of $Q$, then we have either $(e,f)\in E(Q_{23})\times E(Q_{21})$ or $(e,f)\in E(Q_{21})\times E(Q_{23})$, which completes the proof.\end{proof}

\no We need one more technical proposition before we prove the lemma that shows how to select the bad edges that will be part of our final Hamilton cycle.

\begin{proposition}\label{prop:selection-indicators}
Let $t,x_1,x_2,x_3,x_4,x_5\in\{0,1\}$ be such that \begin{align}\label{eqn:modulo-2-congruence}
	x_1+x_2+x_3\equiv t\equiv x_1+x_4+x_5\pmod{2}.
\end{align}
 Then, one can find $m_i\in\{-1,1\}$ for $i \in[5]$ with
$$m_1x_1+m_2x_2+m_3x_3=t=m_1x_1+m_4x_4+m_5x_5. $$
\end{proposition}

\begin{proof}
Without loss of generality, we can assume $x_2\leq x_3$, $x_4\leq x_5$, and $x_2+x_3\leq x_4+x_5$. By \eqref{eqn:modulo-2-congruence}, we must have $(x_2+x_3,x_4+x_5)\in\{(0,0),(1,1),(2,2),(0,2)\}$.
\begin{enumerate}
\item If $x_2+x_3=0=x_4+x_5$, then we have $t=x_1$ and $x_2=x_3=x_4=x_5=0$. Hence, we only need $m_1x_1=x_1$, which can be done by choosing $m_1=1$.
\item If $x_2+x_3=1=x_4+x_5$, then we have $t=1-x_1$, $x_2=x_4=0$ and $x_3=x_5=1$. Hence, we need $m_1x_1+m_3=1-x_1=m_1x_1+m_5$, which can be done by choosing $m_3=m_5=1$, and $m_1=-1$.
\item If $x_2+x_3=2=x_4+x_5$, then we have $t=x_1$ and $x_2=x_3=x_4=x_5=1$. Hence, we need $m_1x_1+m_2+m_3=x_1=m_1x_1+m_4+m_5$, which can be done by choosing $m_1=1$, $m_2=m_4=1$, and $m_3=m_5=-1$.
\item If $x_2+x_3=0$ and $x_4+x_5=2$, then we have $t=x_1$, $x_2=x_3=0$ and $x_4=x_5=1$. Hence, we need $m_1x_1=x_1=m_1x_1+m_4+m_5$, which can be done by choosing $m_4=1$, $m_5=-1$, and $m_1=1$.\end{enumerate}\end{proof}
\no We are now ready to prove the main result of this section.

\begin{lemma}\label{lem:9-partition-path-selection}
Let $1/n\ll\gamma\ll\tau\ll\alpha\ll1$ be some constants, let $G$ be a $d$-regular oriented graph on $n$ vertices with $d\geq \alpha n$ and an extremal $(9,\tau,\gamma)$-partition~$\mathcal{P}_3=\{V_{ij}:i,j \in [3]\}$.
Then, there exists a path system $\mathcal{Q}$ in $\mathcal{B}_3(\mathcal{P}_3,G)$ such that, writing $a_{ij}= |E(\mathcal{Q})\cap E(V_{i*},V_{*j})|$ for all $i\neq j$, we have

\begin{itemize}
\item[(i)] $e\left(\mathcal{Q}\right)\leq 2\gamma n/\alpha$, and
\item[(ii)] $a_{i*}-a_{*i}=|V_{i*}|-|V_{*i}|$ for all $i \in [3]$, where $a_{i*}=\sum_{j\neq i}a_{ij}$ and $a_{*i}=\sum_{j\neq i}a_{ji}$. 
\end{itemize}
\end{lemma}

\begin{proof}
We first give the main idea of the proof. Note that by Proposition~\ref{prop:partition-regular-graphs}, it suffices to find a path system $\mathcal{Q}$ satisfying 
\begin{align}
a_{i*}-a_{*i}= \sum_{j\neq i} e(G_{ij})/d-\sum_{j\neq i}e(G_{ji})/d \label{eqn:balancing-condition}
\end{align}
for each $i\in[3]$. By Proposition~\ref{prop:degree-property-extremal}, we know $\Delta(G_{ij})\leq d/2$, so by using Lemma~\ref{lem:CYCLE_FREE}, we can find path systems $\mathcal{Q}_{ij}$ in $G_{ij}$ such that $\bigcup_{i\neq j}\mathcal{Q}_{ij}$ is cycle-free and $e(\mathcal{Q}_{ij})$ has roughly $2e(G_{ij})/d$ edges. Therefore, we need 
only (roughly) half of the edges from each $\mathcal{Q}_{ij}$ to satisfy \eqref{eqn:balancing-condition}. Moreover, for each $i\neq j$, it makes sense to include edges from only one of $G_{ij}$ and $G_{ji}$. We will choose $\mathcal{Q}^0_{ij}\subseteq \mathcal{Q}_{ij}$, where $\mathcal{Q}^0_{ij}$ has size (roughly) $e(\mathcal{Q}_{ij})/2$ for three different pairs $(i,j)$ and is empty for the remaining three pairs, by using the structural properties of $\bigcup_{i\neq j}\mathcal{Q}_{ij}$ (ensured by Lemma~\ref{lem:ADE-BCF-Structure}) so that $\Delta^0(\bigcup_{i\neq j}\mathcal{Q}^0_{ij})=1$. Since $\bigcup_{i\neq j}\mathcal{Q}^0_{ij}\subseteq \bigcup_{i\neq j}\mathcal{Q}_{ij}$ is cycle-free,  $\Delta^0(\bigcup_{i\neq j}\mathcal{Q}^0_{ij})=1$ guarantees that $\bigcup_{i\neq j}\mathcal{Q}^0_{ij}$ is a path system. Also, $e(\bigcup_{i\neq j}\mathcal{Q}^0_{ij})$ is small enough by the construction since $e(\bigcup_{i\neq j}\mathcal{Q}_{ij})\leq 2|\mathcal{B}_3(\mathcal{P}_3,G)|/d\leq 2\gamma n/\alpha$.\\

\no
Let us write $n_i=|V_{i*}|-|V_{*i}|$ for $i \in [3]$. Since $n_1+n_2+n_3=0$, without loss of generality, we can assume $n_1,n_2\geq 0$. 
Recall $G_{ij} = G[V_{i*},V_{*j}]$, and write $m_{ij}=e(G_{ij})-e(G_{ji})$ for $i,j \in [3]$, $i\neq j$.
Note that $m_{ij} = - m_{ji}$. Without loss of generality, we can assume $m_{12}\geq 0$.
By Proposition~\ref{prop:partition-regular-graphs}, we have
\begin{align}
d n_1&=m_{12}+m_{13}=m_{12}-m_{31},\label{eqn:theta-1-equality}\\
d n_2&=m_{21}+m_{23}=m_{23}-m_{12}.\label{eqn:theta-2-equality}
\end{align}
Since $n_1,n_2\geq 0$, \eqref{eqn:theta-1-equality} and \eqref{eqn:theta-2-equality} imply that $m_{23}\geq m_{12}\geq m_{31}$.
So, it suffices to consider the cases $$m_{23}\geq m_{12}\geq m_{31}\geq 0\,\text{ and }\,m_{23}\geq m_{12}\geq 0\geq m_{31}.$$
Let $m_{12}=dx$ for some $x\geq 0$, and write $x=s+t$ where $s=\lfloor x\rfloor$ and $0\leq t<1$.\\

\no \underline{\textit{Case 1:}} Suppose we have $m_{23}\geq m_{12}\geq m_{31}\geq 0$. Then, we can write $m_{31}=d(x-n_1)$ and $m_{23}=d(x+n_2)$ by using \eqref{eqn:theta-1-equality} and~\eqref{eqn:theta-2-equality}. Let $H = G_{23} \cup G_{31} \cup G_{12} $.
Notice that $e(H)\leq |\mathcal{B}_3(\mathcal{P}_3,G)| \leq\gamma n^2$. Also, by Proposition~\ref{prop:degree-property-extremal}, we know $\Delta^{0}(G_{23}),\Delta^{0}(G_{31}),\Delta^{0}(G_{12})\leq d/2$. By Lemma~\ref{lem:CYCLE_FREE}, we can find path systems $\mathcal{Q}_{23} \subseteq G_{23}$, $\mathcal{Q}_{31} \subseteq G_{31}$, $\mathcal{Q}_{12} \subseteq G_{12}$ such that $\mathcal{Q}_{23}\cup \mathcal{Q}_{31}\cup \mathcal{Q}_{12}$ has no cycle and 
\begin{align*}
e(\mathcal{Q}_{23}) &= \left\lfloor \dfrac{e(G_{23})}{d/2}  \right\rfloor 
\geq \left\lfloor \dfrac{m_{23}}{d/2} \right\rfloor = \left\lfloor 2x+2n_2  \right\rfloor \geq s+n_2,\\
e(\mathcal{Q}_{31}) &= \left\lfloor \dfrac{e(G_{31})}{d/2}  \right\rfloor 
\geq \left\lfloor \dfrac{m_{31}}{d/2} \right\rfloor = \left\lfloor 2x-2n_1  \right\rfloor \geq s-n_1,\\
e(\mathcal{Q}_{12}) &= \left\lfloor \dfrac{e(G_{23})}{d/2}  \right\rfloor 
\geq \left\lfloor \dfrac{m_{12}}{d/2} \right\rfloor = \left\lfloor 2x  \right\rfloor \geq s.
\end{align*}

\no Moreover, by Lemma~\ref{lem:ADE-BCF-Structure}, we have $\mathcal{Q}_{23}\cup \mathcal{Q}_{31}\cup \mathcal{Q}_{12}$ is a disjoint union of paths and cycles since $\{\mathcal{Q}_{23},\mathcal{Q}_{31},\mathcal{Q}_{12}\}$ is a symmetric $3$-set. However, we know it is cycle-free, which implies it is a path system.
Note that \eqref{eqn:theta-1-equality} implies that $d(x-n_1)=m_{31}\geq 0$ by assumption, so we have $s-n_1\geq 0$.
We now define $\mathcal{Q}$ by choosing $s+n_2$ edges from~$\mathcal{Q}_{23}$, $s-n_1$ edges from~$\mathcal{Q}_{31}$, and $s$ edges from~$\mathcal{Q}_{12}$.
Note that 
$$ e\left( \mathcal{Q}\right)\leq \dfrac{e(G_{12})+e(G_{23})+e(G_{31})}{d/2}\leq \dfrac{|\mathcal{B}_3(\mathcal{P}_3,G)|}{d/2}\leq \dfrac{2\gamma n}{\alpha }.$$
Since $a_{23}=s+n_2$, $a_{31}=s-n_1$, $a_{12}=s$ and $a_{21}=a_{32}=a_{13}=0$, the result follows. \\

\no \underline{\textit{Case 2:}} Suppose we have $m_{23} \ge m_{12}\geq 0\geq m_{31}$. 
Recall $m_{12}=dx$. 
Then, we can write $m_{13}=d(n_1-x)$ and $m_{23}=d(n_2+x)$ by using~\eqref{eqn:theta-1-equality} and~\eqref{eqn:theta-2-equality}. 
As with the previous case, we can find path systems $\mathcal{Q}_{23}\subseteq G_{23}$, $\mathcal{Q}_{13}\subseteq G_{13}$, $\mathcal{Q}_{12}\subseteq G_{12}$ such that $\mathcal{Q}_{23}\cup \mathcal{Q}_{13}\cup \mathcal{Q}_{12}$ is cycle-free and 
\begin{align}
\label{eqn:size-path-systems}
e(\mathcal{Q}_{13})= 2n_1 + \lfloor-2x\rfloor,\, e(\mathcal{Q}_{23})&= \lfloor2x\rfloor +2n_2 ,\, e(\mathcal{Q}_{12})= \lfloor2x\rfloor,\\
e(\mathcal{Q}_{13} \cup \mathcal{Q}_{23} \cup \mathcal{Q}_{12}) &\le 2 \gamma n/ \alpha.
\label{eqn:size-path-systems2}
\end{align}
Let $H$ be the graph induced by $\mathcal{Q}_{23}\cup \mathcal{Q}_{13}\cup \mathcal{Q}_{12}$.
Note that $\mathcal{Q}_{13}$ is the special element of the anti-symmetric $3$-set $\{\mathcal{Q}_{23},\mathcal{Q}_{13},\mathcal{Q}_{12}\}$. 
For simplicity, we write $A = 13$, $B = 23$ and $C = 12$ (so e.g.\ $m_A = m_{13}$ and $G_A = G_{13}$).
By Lemma~\ref{lem:ADE-BCF-Structure}, we can decompose $E(H)$ into six sets~$\mathscr{S}_{T}$ with $T\in\{ ABC,AB,AC,A,B,C\}$ such that $\mathscr{S}_T$ is the set of maximal anti-directed paths of length~$|T|$ containing an edge in each $\mathcal{Q}_S$ for $S \in T$, e.g.\ $\mathscr{S}_{ABC}$ is the set of anti-directed paths of length three with one edge in each of $\mathcal{Q}_{13},\mathcal{Q}_{23},\mathcal{Q}_{12}$.\\

\no From the definition of the decomposition, clearly we have
\begin{align*}
   e(\mathcal{Q}_{13}) & = e(\mathcal{Q}_A) = |\mathscr{S}_{ABC}|+|\mathscr{S}_{AB}|+|\mathscr{S}_{AC}|+|\mathscr{S}_{A}|,\\
   e(\mathcal{Q}_{23}) & =  e(\mathcal{Q}_{B}) = |\mathscr{S}_{ABC}|+|\mathscr{S}_{AB}|+|\mathscr{S}_{B}|,\\
   e(\mathcal{Q}_{12}) & = e(\mathcal{Q}_{C}) =
   |\mathscr{S}_{ABC}|+|\mathscr{S}_{AC}|+|\mathscr{S}_{C}|.
\end{align*}
By \eqref{eqn:size-path-systems}, we obtain
\begin{align}
2(|\mathscr{S}_{ABC}|+|\mathscr{S}_{AC}|)+|\mathscr{S}_{AB}|+|\mathscr{S}_A|+|\mathscr{S}_C|&
= e(\mathcal{Q}_{A})+e(\mathcal{Q}_{C})
=2n_1+\lfloor-2t\rfloor+\lfloor2t\rfloor\label{eqn:theta-1-related}\\
|\mathscr{S}_{AB}|+|\mathscr{S}_B|-|\mathscr{S}_{AC}|-|\mathscr{S}_C| &
= e(\mathcal{Q}_{B})-e(\mathcal{Q}_{C})
=
2n_2\label{eqn:theta-2-related}.
\end{align}
Hence, letting $|\mathscr{S}_{T}|\equiv r_T\pmod{2}$ for $T\in\{ABC,AB,AC,A,B,C\}$ where $r_T\in\{0,1\}$ (so $|\mathscr{S}_{T}|\pm r_T$ is even), we have the following equivalence by summing \eqref{eqn:theta-1-related} and \eqref{eqn:theta-2-related}: $$r_{AB}+r_A+r_C\equiv -\lfloor2t\rfloor-\lfloor-2t\rfloor\equiv r_{AC}+r_A+r_B\pmod{2}.$$
Since $0\leq t<1$, we have $-\lfloor2t\rfloor-\lfloor-2t\rfloor\in\{0,1\}$. Then, by Proposition~\ref{prop:selection-indicators}, we can find $i_{AB},i_{AC},i_A,i_B,i_C\in\{-1,1\}$ such that 
\begin{align}
i_{AB}r_{AB}+i_Ar_A+i_Cr_C= -\lfloor2t\rfloor-\lfloor-2t\rfloor= i_{AC}r_{AC}+i_Ar_A+i_Br_B. \label{eqn:indicator-selection}
\end{align}

\no
We now construct $\mathcal{Q} \subseteq H$ as follows. 
Initializing $\mathcal{Q}=\emptyset$, we will add some edges into $\mathcal{Q}$ as follows:
\begin{enumerate}
\item Choose $ \left(|\mathscr{S}_{ABC}|+r_{ABC}\right)/2 $ many paths from $\mathscr{S}_{ABC}$, $\left(|\mathscr{S}_{AB}|+i_{AB}r_{AB}\right)/2$ many paths from $\mathscr{S}_{AB}$, and $\left(|\mathscr{S}_{AC}|+i_{AC}r_{AC}\right)/2$ many paths from $\mathscr{S}_{AC}$. 
For each such path, we add the unique edge from $\mathcal{Q}_{A} \subseteq G_{13}$ to~$\mathcal{Q}$.

\item Take the remaining $\left(|\mathscr{S}_{ABC}|-r_{ABC}\right)/2$ many paths from $\mathscr{S}_{ABC}$. 
For each such path, we add the unique edge from $\mathcal{Q}_{B} \subseteq G_{23}$ and the unique edge from $\mathcal{Q}_{C} \subseteq G_{12}$ to~$\mathcal{Q}$.

\item Take the remaining $\left(|\mathscr{S}_{AB}|-i_{AB}r_{AB}\right)/2$ many paths from $\mathscr{S}_{AB}$. 
For each such path, we add the unique edge from $\mathcal{Q}_{B} \subseteq G_{23}$ to~$\mathcal{Q}$.

\item Take the remaining $\left(|\mathscr{S}_{AC}|-i_{AC}r_{AC}\right)/2$ many paths from $\mathscr{S}_{AC}$. 
For each such path, we add the unique edge from $\mathcal{Q}_{C} \subseteq G_{12}$ to~$\mathcal{Q}$.

\item For each $T\in\{A,B,C\}$, take $\left(|\mathscr{S}_{T}|+i_{T}r_{T}\right)/2$ many paths from $\mathscr{S}_T$. 
Add them to~$\mathcal{Q}$. 
\end{enumerate}
If $\Delta^0(\mathcal{Q}) \ge 2$, then there exists an anti-directed path~$Q'$ of length~2 in~$\mathcal{Q}$.
This path~$Q'$ must be contained in some maximal anti-directed path~$Q^*$ in $\mathscr{S}_{ABC} \cup \mathscr{S}_{AB} \cup \mathscr{S}_{AC}$.
Only in Step~2 do we add more than one edge from a maximal anti-directed path to~$\mathcal{Q}$.
However, the two edges added in that case are not incident by~Lemma~\ref{lem:ADE-BCF-Structure}(ii) as $\mathcal{Q}_{13} = \mathcal{Q}_A$ is the special element. Therefore no such $Q'$ exists, and so $\Delta^0(\mathcal{Q}) \leq 1$.
Recall that $\bigcup \mathcal{Q} \subseteq H$ is cycle-free and so $\mathcal{Q}$ is a path system. 
By~\eqref{eqn:size-path-systems2}
\begin{align*}
e\left(\mathcal{Q}\right) \le e(H) \le 2 \gamma n /\alpha.
\end{align*}
Note that 
\begin{align*}
2(a_{1*} - a_{*1}) & = 
2\left(e\left( \mathcal{Q} \cap G_{12}\right) + e\left( \mathcal{Q} \cap G_{13}\right)\right)
=
2\left(e\left( \mathcal{Q} \cap \mathcal{Q}_{A}\right) + e\left( \mathcal{Q} \cap \mathcal{Q}_{C}\right)\right)
\\ &
= 
2(|\mathscr{S}_{ABC}|+|\mathscr{S}_{AC}|)+|\mathscr{S}_{AB}|+|\mathscr{S}_A|+|\mathscr{S}_C| +{i_{AB}r_{AB}+i_Ar_A+i_Cr_C} = 2n_1,
\end{align*}
where the last equality is due to~\eqref{eqn:theta-1-related} and~\eqref{eqn:indicator-selection}.
Similarly, 
\begin{align*}
2(a_{2*} - a_{*2}) & = 
 2(e( \mathcal{Q} \cap G_{23}) - e(\mathcal{Q} \cap G_{12}))
= 
2(e( \mathcal{Q} \cap G_{B}) - e( \mathcal{Q} \cap G_{C}))
\\
&= (|\mathscr{S}_{AB}|+|\mathscr{S}_B|-|\mathscr{S}_{AC}|-|\mathscr{S}_C|)+({i_{AC}r_{AC}+i_Br_B-i_{AB}r_{AB}-i_Cr_C})
= 2 n_2,
\end{align*}
where the last equality is due to~\eqref{eqn:theta-2-related} and~\eqref{eqn:indicator-selection}. 
So we have $a_{1*} - a_{*1} = n_1$ and  $a_{2*} - a_{*2} = n_2$.
Since $n_1 + n_2 +n_3 = 0$, we deduce that $a_{3*} - a_{*3} = n_3$ as required. 
 \end{proof}

\no 
The previous lemma shows how to obtain the (path system of) bad edges that will be part of our final Hamilton cycle. It will be convenient to suitably contract this path system because the resulting contracted graph will have a ``balanced'' partition and finding a Hamilton cycle in the contracted graph will give us a Hamilton cycle in the original graph by ``uncontracting'' the path system. We now define the right notion of contraction and establish some of its properties.

\begin{definition}\label{def:path-contraction}
Let $G$ be a digraph, $k\in \mathbb{N}$, and $\mathcal{P}_k=\{V_{ij}:i,j\in [k]\}$ be a $k^2$-partition of $V(G)$. Let $\mathcal{Q}$ be a path system in $G$. We define the \emph{contraction of $\mathcal{Q}$ in $G$ with respect to $\mathcal{P}_k$} as follows: for each $Q\in\mathcal{Q}$, create a new vertex $x$ associated to $Q$ such that $N^{-}(x)=N_{G}^{-}(u)$ and $N^{+}(x)=N_{G}^{+}(v)$ where $Q$ goes from $u$ to $v$. If $u\in V_{ij}$ and $v\in V_{i'j'}$, put $x$ into $V_{i'j}$. Then, we delete all the vertices in $\mathcal{Q}$. We call $\mathcal{P}_k'=\{V_{ij}':i,j \in [k]\}$ the \emph{resulting partition} where $V_{ij}'$ is the updated version of $V_{ij}$ for all $i,j \in [k]$, and we denote the \emph{resulting} graph by $G'$.
\end{definition}

\no Since we often use the following fact, we state it as a proposition.

\begin{proposition}
\label{prop:contract-hamilton}
Let $G$ be a digraph, $\mathcal{Q}$ be a path system in $G$, and $\mathcal{P}_k=\{V_{ij}:i,j\in [k]\}$ be a $k^2$-partition of $V(G)$. If $G'$ is the graph obtained from $G$ by contracting $\mathcal{Q}$  with respect to $\mathcal{P}_k$, and $G'$ is Hamiltonian, then so is $G$.
\end{proposition}

\no Next we see that the number of bad edges cannot increase from contracting a path system with respect to the given partition.

\begin{proposition}\label{prop:after-contraction}
Let $1/n\ll \theta,\gamma\ll \tau \le 1$ and $k\in \mathbb{N}$ be constants.
Let $G$ be a digraph on $n$ vertices,  $\mathcal{Q}$ be a path system in $G$, and $\mathcal{P}_k=\{V_{ij}:i,j\in [k]\}$ be a $k^2$-partition of $V(G)$. Let us contract $\mathcal{Q}$ with respect to the partition $\mathcal{P}_k$. Then, we have $|\mathcal{B}_k(\mathcal{P}_k',G')|\leq |\mathcal{B}_k(\mathcal{P}_k,G)|$. 
Moreover, if $\mathcal{P}_k$ is a $(k^2,\tau,\gamma)$-partition of $G$ and $e(\mathcal{Q}) \le \theta n$, then $\mathcal{P}_{k}'$ is a $(k^2,\tau/2,2\gamma)$-partition of $G'$.
\end{proposition}
\begin{proof}
Consider a path $P\in\mathcal{Q}$ that goes from $u$ to $v$, let $x$ be the created vertex corresponding to $P$ during the contraction process with $x\in V_{cb}'$. In particular, we have $v \in V_{c*}$. If $xy\in \mathcal{B}_k(\mathcal{P}_{k}',G')$, then $y\notin V_{*c}$ and $y\in N^{+}(v)$, which shows $vy\in \mathcal{B}_k(\mathcal{P}_k,G)$. Similarly, for any bad edge in $G'$ with respect to~$\mathcal{P}_{k}'$, we can find a different bad edge in $G$ with respect to $\mathcal{P}_k$, which shows $|\mathcal{B}_k(\mathcal{P}_{k}',G')|\leq |\mathcal{B}_k(\mathcal{P}_k,G)|$.\\

\no 
Notice that we have $|G'|\geq (1-\theta)n$. Also, since we deleted at most $2\theta n$ vertices and $\theta\ll \tau$, we have $|V_{i*}'|\geq \tau n-2\theta n\geq \tau |G'|/2$ for all $i \in [k]$. Moreover, we get $2(1-\theta)^2>1$ since $\theta\ll1$, which implies $|\mathcal{B}_k(\mathcal{P}_{k}',G')|\leq\gamma n^2\leq 2\gamma (1-\theta)^2n^2\leq 2\gamma|G'|^2$.  
\end{proof}

\no We end this section with a lemma which states that if a path system $\mathcal{Q}$ in $\mathcal{B}(\mathcal{P}_k,G)$ satisfies condition~(ii) of Lemma~\ref{lem:9-partition-path-selection}, then the contraction of $\mathcal{Q}$ with respect to $\mathcal{P}_k$ balances the partition.

\begin{lemma}\label{lem:BALANCING-PARTITION}
Let $k \in \mathbb{N}$, and let $\mathcal{P}_{k}=\{V_{ij}:i,j\in [k]\}$ be a $k^2$-partition for a digraph $G$. Let $\mathcal{Q}$ be a path system in $\mathcal{B}_{k}(\mathcal{P}_{k},G)$ such that, for all $i \in [k]$, $$\sum_{j\neq i}a_{ij}-\sum_{j\neq i}a_{ji}=|V_{i*}|-|V_{*i}|,$$  where $a_{ij}$ denotes the number of edges in $E(\mathcal{Q})\cap E(V_{i*},V_{*j})$ for all $i\neq j$.
Then, the contraction of $\mathcal{Q}$ with respect to $\mathcal{P}_{k}$ results in a digraph $G'$ with a $k^2$-partition $\mathcal{P}'_{k}=\{V'_{ij}:i,j \in [k]\}$ such that $|V_{i*}'|=|V_{*i}'|$ for all $i \in [k]$.
\end{lemma}
\begin{proof}
Let $\mathcal{Q}=\{Q_1,Q_2,\ldots,Q_t\}$. Let $a_{ij}^{p}$ denote the number of edges in $E(Q_p)\cap E(V_{i*},V_{*j})$ for all $1\leq p \leq t$ and $i\neq j$. 
Consider a path $Q_p$, say from $u\in V_{xy}$ to $v\in V_{zt}$. Recall that we delete all the vertices in $Q_p$ and add a new vertex into $V_{zy}$ (see Definition~\ref{def:path-contraction}). By applying Proposition~\ref{prop:partition-regular-graphs} with $d=1$ and $G=Q_p\cup \{vu\}$, we obtain 
\begin{align*}
|V_{i*}\cap V(Q_p)|-|V_{*i}\cap V(Q_p)|=\sum_{j\neq i}a_{ij}^{p}-\sum_{j\neq i}a_{ji}^{p}+\mathbbm{1}\{i=z\}-\mathbbm{1}\{i=y\}    
\end{align*}
for each $i\in[k]$ since $v\in V_{z*}$ and $u\in V_{*y}$. By considering the new vertex added into $V_{zy}$, we see that the contraction of the path $Q_p$ leads to a decrease in $|V_{i*}|-|V_{*i}|$ by $\sum_{j\neq i}a_{ij}^{p}-\sum_{j\neq i}a_{ji}^{p}$.
Since all the paths in $\mathcal{Q}$ can be contracted independently, we have 
\begin{align*}
|V_{i*}'|-|V_{*i}'|=\left(|V_{i*}|-|V_{*i}|\right)-\sum_{p \in [t]}\left( \sum_{j\neq i}a_{ij}^{p}-\sum_{j\neq i}a_{ji}^{p} \right)=\left(|V_{i*}|-|V_{*i}|\right)-\left(\sum_{j\neq i}a_{ij}-\sum_{j\neq i}a_{ji}\right).
\end{align*}
Since we have $\sum_{j\neq i}a_{ij}-\sum_{j\neq i}a_{ji}=|V_{i*}|-|V_{*i}|$, the result follows.  
\end{proof}

\section{Hamilton cycles from partitions}\label{ch:partition-to-hamilton}

The main goal of this section is to prove that regular directed or oriented graphs of suitably high degree that admit a $(k^2, \tau, \gamma)$-partition for suitable $k, \tau, \gamma$ have a Hamilton cycle. We begin by formally defining
 certain contracted graphs associated with $4$-partitions (i.e.\ the graphs $J_i$ discussed in the sketch proof). These will be used in this and the next section. \\

\no
Let $H$ be a (undirected) bipartite graph with bipartition $(A,B)$ and $|A|=|B|=n$. Given a set $K$ of size $n$ and bijections $\phi_A:K\to A$ and $\phi_B:K\to B$, the \emph{identification of $H$ with respect to $(K,\phi_A,\phi_B)$} is defined to be the digraph $G$, where $V(G)=K$ and for each $a,b\in K$, we have $ab\in E(G)$ if and only if $\phi_A(a)\phi_B(b)\in E(H)$.\footnote{If $\phi_A(a)\phi_B(a) \in E(H)$ for some $a \in K$, then we will have a loop $aa \in E(G)$. The small number of loops in $G$ play no role in our arguments, but we keep them for convenience so that $H$ and $G$ have the same number of edges.}  \\

\no	Let $G$ be a digraph and $\mathcal{P}=\{V_{ij}:i,j \in [2]\}$ be a $4$-partition of $V(G)$. 
For each $i \in [2]$,  we define $\mathcal{B}^{i}(\mathcal{P},G)$ to be the (undirected) bipartite graph with bipartition $(V_{i*},V_{*i})$, where, for each $a\in V_{i*}$ and $b\in V_{*i}$, we have $ab\in E(\mathcal{B}^{i}(\mathcal{P},G))$ if and only if $ab\in E(G)$. 
(Although $V_{i*}$ and $V_{*i}$ are not disjoint as subsets of $V(G)$, namely $V_{i*} \cap V_{*i} = V_{ii}$, we duplicate any vertices in~$V_{ii}$, so $\mathcal{B}^{i}(\mathcal{P},G)$ has $|V_{i*}| + |V_{*i}|$ vertices.)  \\
	
\no Let $G$ be a digraph and $\mathcal{P}=\{V_{ij}:i,j \in [2]\}$ be a $4$-partition of $V(G)$ such that $|V_{12}|=|V_{21}|=t>0$. For $i \in [2]$, we call  $\phi^i=(\phi_{i*},\phi_{*i})$ a \emph{proper $i$-pair with respect to $\mathcal{P}$} if  $\phi_{i*}:[t]\cup V_{ii}\to V_{i*}$ and $\phi_{*i}:[t]\cup V_{ii}\to V_{*i}$ are bijections satisfying $\phi_{i*}(x)=\phi_{*i}(x)=x$ for all $x\in V_{ii}$.
In this case we define $\mathcal{J}^{i}(\mathcal{P},G,\phi^i)$ to be the identification of $\mathcal{B}^{i}(\mathcal{P},G)$ with respect to $([t]\cup V_{ii},\phi_{i*},\phi_{*i})$.
Formally, $V(\mathcal{J}^{i}(\mathcal{P},G,\phi^i))=[t]\cup V_{ii}$ and $xy\in E(\mathcal{J}^{i}(\mathcal{P},G,\phi^i))$ if and only if $\phi_{i*}(x)\phi_{*i}(y)\in E(\mathcal{B}^{i}(\mathcal{P},G))$.
One can think of $\mathcal{J}^{i}(\mathcal{P},G,\phi^i)$ as the digraph obtained from $G[V_{i*} \cup V_{*i}]$ by pairing vertices in $V_{i*} \setminus V_{ii}$ with vertices in $V_{*i} \setminus V_{ii}$ and identifying them, where the pairing is determined by $\phi_{i*}$ and $\phi_{*i}$; if we pair $x \in V_{i*} \setminus V_{ii}$ with $y \in V_{*i} \setminus V_{ii}$, the identified vertex has the same outneighbours as $x$ and the same inneighbours as $y$. Note that there is a one-to-one correspondence between the edges in $\mathcal{J}^{i}(\mathcal{P},G,\phi^i)$ and those in $G[V_{i*}, V_{*i}]$. Figure~\ref{fig:CONSTRUCTION-J-FROM-PHI} illustrates this construction by a small example.
\\

\begin{figure}[h]
    \centering

\begin{tikzpicture}[scale=0.8]

\draw[dashed]  (3.7,6.6) -- (3.7,-0.6) -- (13.8,-0.6) -- (13.8,6.6) -- cycle ;

\node at (8.75,-1.8) {\small a proper $1$-pair $\phi^1=(\phi_{1*},\phi_{*1})$ with respect to $\mathcal{P}$};
\node at (8.75,-1.2) {\small Partition $\mathcal{P}$ with $|V_{12}|=|V_{21}|=t>0$, and};

\draw  (4,4.5) -- (4,1.5) -- (7,1.5) -- (7,4.5) -- cycle ;
\node at (5.5,4.8) {$V_{11}$};
\node [style=blackcircle,scale=0.4] (x1) at (5,2) {};
\node [style=blackcircle,scale=0.4] (x2) at (5,4) {};
\node [style=blackcircle,scale=0.4] (x3) at (6,2) {};
\node [style=blackcircle,scale=0.4] (x4) at (6,4) {};

\draw  (7,6) -- (7,5) -- (10,5) -- (10,6) -- cycle ;
\node at (8.5,6.3) {$V_{12}$};
\node [style=blackcircle,scale=0.4] (V12-1) at (7.75,5.5) {};
\node [style=blackcircle,scale=0.4] (V12-2) at (8.5,5.5) {};
\node [style=blackcircle,scale=0.4] (V12-3) at (9.25,5.5) {};

\draw  (7,1) -- (7,0) -- (10,0) -- (10,1) -- cycle ;
\node at (8.5,-0.3) {$V_{21}$};
\node [style=blackcircle,scale=0.4] (V21-1) at (7.75,0.5) {};
\node [style=blackcircle,scale=0.4] (V21-2) at (8.5,0.5) {};
\node [style=blackcircle,scale=0.4] (V21-3) at (9.25,0.5) {};

\draw  (10,4.25) -- (10,1.75) -- (13,1.75) -- (13,4.25) -- cycle ;
\node at (13.3,3) {$[t]$};
\node [style=blackcircle,scale=0.4] (1) at (10.25,3) {};
\node at (10.25,3.4) {$1$};
\node [style=blackcircle,scale=0.4] (2) at (11.5,3) {};
\node at (11.5,3.4) {$2$};
\node [style=blackcircle,scale=0.4] (3) at (12.75,3) {};
\node at (12.75,3.4) {$3$};

\draw[dashed]  (15.7,6.3) -- (15.7,-0.3) -- (23.8,-0.3) -- (23.8,6.3) -- cycle ;

\node at (19.75,-1.2) {\small The digraph $\mathcal{J}^{1}(\mathcal{P},G,\phi^1)$ on the vertex set $[t]\cup V_{11}$};

\draw  (16,4.5) -- (16,1.5) -- (19,1.5) -- (19,4.5) -- cycle ;
\node at (17.5,4.8) {$V_{11}$};
\node [style=blackcircle,scale=0.4] (y1) at (17,2) {};
\node [style=blackcircle,scale=0.4] (y2) at (17,4) {};
\node [style=blackcircle,scale=0.4] (y3) at (18,2) {};
\node [style=blackcircle,scale=0.4] (y4) at (18,4) {};

\draw  (20,4.25) -- (20,1.75) -- (23,1.75) -- (23,4.25) -- cycle ;
\node at (23.3,3) {$[t]$};
\node [style=blackcircle,scale=0.4] (t1) at (20.25,3) {};
\node at (20.25,3.4) {$1$};
\node [style=blackcircle,scale=0.4] (t2) at (21.5,3) {};
\node at (21.5,3.4) {$2$};
\node [style=blackcircle,scale=0.4] (t3) at (22.75,3) {};
\node at (22.75,3.4) {$3$};

\draw [style=REDARROW] (V12-1) to (x4);
\draw [style=BLUEARROW] (x3) to (V21-2);
\draw [style=GREENARROW] (V12-2) to (V21-3);
\draw [style=GRAYARROW] (x2) to (x1);
\draw [style=ORANGEARROW] (x1) to (V21-1);

\node at (11.3,4.75) {$\phi_{1*}\vert_{[t]}$};
\node at (11.3,1.25) {$\phi_{*1}\vert_{[t]}$};


\begin{scope}[thick,decoration={
    markings,
    mark=at position 0.25 with {\arrow{>}}}
] 
\draw[dashed,postaction={decorate}] (1) to (V12-1);
\draw[dashed,postaction={decorate}] (2) to (V12-2);
\draw[dashed,postaction={decorate}] (3) to (V12-3);
\draw[dashed,postaction={decorate}] (1) to (V21-1);
\draw[dashed,postaction={decorate}] (2) to (V21-2);
\draw[dashed,postaction={decorate}] (3) to (V21-3);

\end{scope}

\draw [style=REDARROW] (t1) to (y4);
\draw [style=BLUEARROW] (y3) to (t2);
\draw [style=GREENARROW] (t2) to (t3);
\draw [style=GRAYARROW] (y2) to (y1);
\draw [style=ORANGEARROW] (y1) to (t1);

\end{tikzpicture}

    \caption{An illustration for how $\mathcal{J}^{1}(\mathcal{P},G,\phi^1)$ is constructed.}
    \label{fig:CONSTRUCTION-J-FROM-PHI}
\end{figure}

\no	
The first proposition shows how Hamiltonicity of $\mathcal{J}^i$ translates into Hamiltonicity for $G$.

\begin{proposition}\label{prop:identification-Hamilton-implies-Hamilton}
Let $G$ be a digraph on $n$ vertices, and let $\mathcal{P}=\{V_{ij}:i,j \in [2]\}$ be a $4$-partition of~$V(G)$ with $|V_{12}|=|V_{21}|>0$. Suppose that for every $i \in [2]$ and every proper $i$-pair $\phi^i$ with respect to $\mathcal{P}$, we have that $\mathcal{J}^{i}(\mathcal{P},G,\phi^i)$ is Hamiltonian. Then, $G$ is Hamiltonian.
\end{proposition}

\begin{proof}
Let $|V_{12}|=|V_{21}|=t$ and $\phi^1$ be a proper 1-pair with respect to $\mathcal{P}$. Consider a Hamilton cycle~$C$ in $\mathcal{J}^{1}(\mathcal{P},G,\phi^1)$.
Recall that the vertex set of $\mathcal{J}^{1}(\mathcal{P},G,\phi^1)$ is $[t] \cup V_{11}$.
Let $p_1,\ldots,p_t$ be the order in which the vertices in $[t]$ are visited by $C$ so that $C$ can be partitioned into paths $P_1, \ldots P_t$ where $P_r$ is a path from $p_r$ to $p_{r+1}$ (with the convention that $p_{t+1}=p_1$). Each $P_r$ corresponds to a path $P_r^1$ in $G[V_{11} \cup V_{12} \cup V_{21}]$ from $\phi_{1*}(p_r) \in V_{12}$ to $\phi_{*1}(p_{r+1}) \in V_{21}$, and moreover the paths $P_1^1, \ldots, P_t^1$ are vertex-disjoint and span $V_{11} \cup V_{12} \cup V_{21}$.\\

\no Let $\phi^2$ be the proper $2$-pair with respect to $\mathcal{P}$ satisfying $\phi_{2*}(p_r) = \phi_{*1}(p_{r+1}) \in V_{21}$ and $\phi_{*2}(p_r)=\phi_{1*}(p_r) \in V_{12}$ for all $r \in [t]$. Note that $\mathcal{J}^{2}(\mathcal{P},G,\phi^2)$ can be obtained from $G[V_{2*} \cup V_{*2}]$ by identifying the start and end points of $P_r^1$ for each $r$ and calling the resulting vertex $p_r$ (here we keep only the inedges of the start point $\phi_{1*}(p_r)$  and the outedges of the end point $\phi_{*1}(p_{r+1})$). Since $\mathcal{J}^{2}(\mathcal{P},G,\phi^2)$ has some Hamilton cycle $H$, we see that $G$  also has a Hamilton cycle, obtained by replacing each vertex $p_r$ in $H$ with the path $P^1_r$. 
\end{proof}

\no Next, we will prove that digraphs admitting a $(4,1/3,\gamma)$-partition with additional degree conditions are Hamiltonian.
Recall that for a $k^2$-partition $\mathcal{P}_k=\{V_{ij}:i,j\in[k]\}$ of $V(G)$, the set of good edges was defined as $\mathcal{G}_k(\mathcal{P}_k,G)=\bigcup_{i}E(V_{i*},V_{*i})$ (see Definition~\ref{def:good-edges}), and we also think of $\mathcal{G}_k(\mathcal{P}_k,G)$ as the subdigraph of $G$ with the vertex set consisting of those vertices incident to edges in $\mathcal{G}_k(\mathcal{P}_k,G)$.

\begin{lemma}\label{lem:digraph-hamilton}
Let $1/n \ll \gamma,\rho\ll \eps\ll1$ be constants. 
Let $G$ be a digraph on $n$ vertices with a $(4,1/3,\gamma)$-partition $\mathcal{P}=\{V_{ij}:i,j \in [2]\}$. Suppose that
\begin{itemize}
\item[\rm(i)] $d^{+}_{\mathcal{G}_2(\mathcal{P},G)}(v),d^{-}_{\mathcal{G}_2(\mathcal{P},G)}(v)\geq (1/3+\eps)n$ holds for all but at most $\rho n$ vertices $v \in V(G)$,
\item[\rm(ii)]$\delta^0 ( \mathcal{G}_2(\mathcal{P},G) ) \geq n/20$, 
\item[\rm(iii)] $|V_{12}|=|V_{21}|>0$.
\end{itemize}
	
	\no Then $G$ is Hamiltonian.

\end{lemma}

\begin{proof}
Let $|V_{12}|=|V_{21}|=t$. For $i \in [2]$, let $\phi^i$ be a proper $i$-pair with respect to $\mathcal{P}$. Let $J_i:=\mathcal{J}^{i}(\mathcal{P},G,\phi^i)$. Since $|V(J_i)|=|V_{i*}|\geq n/3$ (the inequality holds because $\mathcal{P}$ is a $(4, 1/3, \gamma)$-partition), we obtain $n/3\leq |J_i|\leq 2n/3$. On the other hand, for any $v\in V_{ii}$, we have $d_{J_i}^{+}(v)=d_{\mathcal{G}_2(\mathcal{P},G)}^{+}(v)$ and $d_{J_i}^{-}(v)=d_{\mathcal{G}_2(\mathcal{P},G)}^{-}(v)$. Similarly, for any $r\in[t]$, we have $d_{J_i}^{+}(r)=d_{\mathcal{G}_2(\mathcal{P},G)}^{+}(\phi_{i*}(r))$ and $d_{J_i}^{-}(r)=d_{\mathcal{G}_2(\mathcal{P},G)}^{-}(\phi_{*i}(r))$. Then, $d_{J_i}^{+}(x),d_{J_i}^{-}(x)\geq (1/2+\eps)|J_i|$ holds for all but at most $3\rho|J_i|$ vertices $x$ in $J_i$ by (i). Moreover, (ii) implies $\delta^{0}(J_i)\geq |J_i|/20$. Therefore, $J_i$ is Hamiltonian for $i \in [2]$ by Corollary~\ref{cor:Higher-Than-Half-For-All-But-Few}. Hence, the result follows from Proposition~\ref{prop:identification-Hamilton-implies-Hamilton}.	 
\end{proof}

\no We end this section by showing that every regular oriented graph of sufficiently high degree that admits a $(9,\tau,\gamma)$-partition is Hamiltonian.

	\begin{lemma}\label{lem:9partitionHamiltonian}
		Let $1/n <\gamma\ll \tau\ll\varepsilon<1$ be constants. Then every $d$-regular oriented graph~$G$ on $n$ vertices with $d\geq (1/4+\eps)n$ and that admits a $(9,\tau,\gamma)$-partition is Hamiltonian. 
	\end{lemma}
	
	\begin{proof}
		Let $\mathcal{P}=\{V_{ij}:i,j \in [3]\}$ be an extremal $(9,\tau,\gamma)$-partition of $G$. 
		Firstly, we claim at least two of the following are true:
		\begin{itemize}
		\item[(a)] $\tau n+|V_{11}|\leq |V_{22}|+|V_{33}|+|V_{23}|+|V_{32}|$,
		\item[(b)] $\tau n+|V_{22}|\leq |V_{33}|+|V_{11}|+|V_{31}|+|V_{13}|$,
		\item[(c)] $\tau n+|V_{33}|\leq |V_{11}|+|V_{22}|+|V_{12}|+|V_{21}|$.
		\end{itemize}
	\no If not, then without loss of generality, say (a) and (b) are false. By adding up those inequalities, we obtain $2\tau n > 2|V_{33}|+|V_{31}\cup V_{32}|+|V_{13}\cup V_{23}|$. However, by Proposition~\ref{prop:quarterleadsnonempty}, we know $|V_{31}\cup V_{32}|,|V_{13}\cup V_{23}|\geq  \tau n$,  so we have a contradiction. Similarly, it can be easily shown that at least two of the following are true:
	\begin{align*}
		\text{(a$'$) } \tau n &\leq |V_{12}|+|V_{21}|,&
		\text{(b$'$) } \tau n & \leq |V_{13}|+|V_{31}|,&
		\text{(c$'$) } \tau n &\leq |V_{23}|+|V_{32}|.
	\end{align*}
Thus, without loss of generality, we can assume that (c) and (a$'$) hold, that is, 
\begin{align}
\label{eq:WLOG}    
\tau n +|V_{33}|\leq |V_{11}|+|V_{22}|+|V_{12}|+|V_{21}|\text{ and } \tau n \leq |V_{12}|+|V_{21}|.
\end{align}
\no By Lemma~\ref{lem:9-partition-path-selection}, there exists a path system $\mathcal{Q}$ in $\mathcal{B}_3(\mathcal{P},G)$ containing at most $8\gamma n$ edges such that $\sum_{j\neq i}a_{ij} - \sum_{j\neq i}a_{ji}=|V_{i*}|-|V_{*i}|$ for all $i \in [3]$, where $a_{ij} = |E(V_{i*}, V_{*j}) \cap \mathcal{Q}|$. 
We contract $\mathcal{Q}$ with respect to $\mathcal{P}$ and write $G'$ for the resulting graph and $\mathcal{P}'=\{V_{ij}':i,j \in [3]\}$ for the resulting partition. By Proposition~\ref{prop:dense-regular-size}, $\mathcal{P}$ is actually a $(9, 1/4+\eps/2 ,\gamma)$-partition of~$G$, so Proposition~\ref{prop:after-contraction} implies that $\mathcal{P}'$ is a $(9,1/8,2\gamma)$-partition for~$G'$. By Lemma~\ref{lem:BALANCING-PARTITION}, we have 
  \begin{align}
  \label{eqn:balpart}
  |V_{i*}'|=|V_{*i}'| \ge |V_{*i}| - |V(\mathcal{Q})|  \ge n/4 \text{ for all }i \in [3].    
  \end{align}
  Moreover, by using Proposition~\ref{prop:quarterleadsnonempty}, we have \begin{align}
\label{eq:newpartition1}
 \sum_{j\neq i}|V_{ij}'|=\sum_{j\neq i}|V_{ji}'|\geq \tau n \,\text{ for all }i \in [3].
\end{align}
Also, using~\eqref{eq:WLOG} and the facts that $\gamma \ll \tau$ and $ e(\mathcal{Q}) \le 8\gamma n$, we have \begin{align}
\label{eq:newpartition2}
|V_{33}'|\leq |V_{11}'|+|V_{22}'|+|V_{12}'|+|V_{21}'|\,\text{ and }\, |V_{12}'|+|V_{21}'| \geq  \tau n /2.
\end{align}
Since $| V ( \mathcal{Q} ) |\le 16\gamma n$, we have
\begin{align}
\label{eq:degreeG'}
\delta^0(G')\geq d-16\gamma n \geq \left( 1/4 + {\varepsilon}/{2} \right)n.
\end{align}
Similarly, by Proposition~\ref{prop:degree-property-extremal},
\begin{align}
\label{eq:degreee}
\text{for any }  v\in V(G')\text{, if } v \in  V_{ab}' \text{ for some }   a,b \in [3], \text{ then } d_{V_{*a}'}^{+}(v), d_{V_{b*}'}^{-}(v)\geq d/3-16\gamma n. \footnotemark
\end{align}
\no\footnotetext{It is clear that all the vertices of $G'$ inherited from $G$ satisfy these degree conditions; for the new vertices in $G'$ (created from contracting paths), one can easily check in the definition of contraction that the vertices are placed in such a way that the degree conditions hold.}In other words, we have 
\begin{align}
d_{\mathcal{G}_3(\mathcal{P}',G')}^{+}(v),d_{\mathcal{G}_3(\mathcal{P}',G')}^{-}(v)\geq d/3-16\gamma n. \label{eq:G'-min-good-degree}    
\end{align}

\no
Let 
\begin{align*}
W_{11}&=V_{33}', &
W_{12}&=V_{32}'\cup V_{31}', &
W_{21}&=V_{23}'\cup V_{13}', &
W_{22}&=V_{11}'\cup V_{22}'\cup V_{12}'\cup V_{21}'.
\end{align*}
By Proposition~\ref{prop:partition-9-to-4}, we have $\mathcal{W}=\{W_{ij}:i,j \in [2]\}$ is a $(4,1/8,2\gamma)$-partition for $G'$. Furthermore, \eqref{eq:newpartition2} and \eqref{eq:newpartition1} imply that 
\begin{align*}
    |W_{11}|\leq |W_{22}| \text{ and }|W_{12}|=|W_{21}|\geq \tau n/2.
\end{align*}

\no By Proposition~\ref{prop:contract-hamilton}, if $G'$ is Hamiltonian then so is $G$. For $i \in [2]$, let $\phi^i$ be a proper $i$-pair with respect to $\mathcal{W}$. 
In order to prove the lemma, 
it is enough to show that $J_i:=\mathcal{J}^{i}(\mathcal{W},G',\phi^i)$ is Hamiltonian for $i \in [2]$ by Proposition~\ref{prop:identification-Hamilton-implies-Hamilton}.  \\

\no
First, for $J_1$, \eqref{eqn:balpart} and the fact that $|W_{11}|\leq |W_{22}|$ imply that
\begin{align}
    \label{eqn:|V_{3*}|}
    n/4 \le |V_{3*}'|= |J_1| \leq |G'|/2\leq n/2.
\end{align}
Let $B^{+}(J_1)$ be the set of vertices in $J_1$ satisfying $d_{J_1}^{+}(x)<(1/2+\eps/2)|J_1|$. 
Similarly, define $B^{-}(J_1)$.
For any vertex $x\in V(J_1)$, we have $\phi_{1*}(x)\in V_{3*}'$ and $d_{J_1}^{+}(x)=d_{V_{*3}'}^{+}(\phi_{1*}(x))$.
Together with~\eqref{eq:degreeG'} and \eqref{eqn:|V_{3*}|}, we deduce that
\begin{align*}
    2 \gamma n^2 &\ge |\mathcal{B}_2(\mathcal{W},G')| \ge e( \phi_{1*}(B^{+}(J_1)) , V(G') \setminus V_{*3}')
    \ge \sum_{x \in B^{+}(J_1)} \left( d_{G'}^{+}(\phi_{1*}(x)) -  d_{V_{*3}'}^{+}(\phi_{1*}(x)) \right)
        \\
    & \ge \sum_{x \in B^{+}(J_1)} \left( \delta^0(G') -  d_{J_1}^{+}(x) \right)
    \ge \sum_{x \in B^+(J_1)} \left( (1/4 + \eps / 2)n - (1/2 + \eps/2)|J_1| \right)
    \\
    &\ge |B^{+}(J_1)| \eps n/4.
\end{align*}
So $|B^{+}(J_1)| \le 8  \gamma n/\eps$ and, similarly, $|B^{-}(J_1)| \le  8  \gamma n/\eps$.
By~\eqref{eqn:|V_{3*}|},
\begin{align*}
     |B^{+}(J_1)|+|B^{-}(J_1)| & \leq 16\gamma n/\eps \le 64 \gamma |J_1| / \eps \le \sqrt{\gamma}|J_1|.
\end{align*}
Thus $d_{J_1}^{+}(x),d_{J_1}^{-}(x)\geq(1/2+\eps/2)|J_1|$ holds for all but at most $\sqrt{\gamma}|J_1|$ vertices. Also, by \eqref{eq:degreee}, we have $\delta^0 (J_1)
\geq d/3-16\gamma n \geq |J_1|/10$. Therefore, by Corollary~\ref{cor:Higher-Than-Half-For-All-But-Few}, $J_1$ is Hamiltonian.\\

\no For $J_2$, we first show that $J_2$ has a $(4,1/3,8\gamma)$-partition. 
By~\eqref{eqn:|V_{3*}|}
\begin{align}
\label{eqn:|J_2|}
n/2 \le |J_2|=|G'|-|V_{3*}'| \le 3n/4,
\end{align}
so \eqref{eqn:balpart} implies that
\begin{align*}
 |V'_{*i}| =|V_{i*}'| \ge n/4 \ge |J_2|/3 \text{ for } i \in [2].
\end{align*}
Let $t:=|W_{12}|=|W_{21}|$. Recall that $\phi_{2*}:[t]\cup W_{22}\to W_{2*}$ and $\phi_{*2}:[t]\cup W_{22}\to W_{*2}$ are bijections satisfying $\phi_{2*}(x)=\phi_{*2}(x)=x$ for all $x\in W_{22}$, so we have $\phi_{2*}(q)\in V_{32}'\cup V_{31}'$ and $\phi_{*2}(q)\in V_{23}'\cup V_{13}'$ for any $q\in[t]$ since $W_{12}=V_{32}'\cup V_{31}'$ and $W_{21}=V_{23}'\cup V_{13}'$.
Then, we partition $[t]$ into parts $\{T_{ij}:i,j \in [2]\}$ as follows:
\begin{align*}
T_{11}&=\{q \in [t]:\phi_{2*}(q)\in V_{13}', \;\phi_{*2}(q)\in V_{31}'\}, &  T_{12}&=\{q \in [t]:\phi_{2*}(q)\in V_{13}', \; \phi_{*2}(q)\in V_{32}'\},\\
 T_{21}&=\{q \in [t]:\phi_{2*}(q)\in V_{23}', \; \phi_{*2}(q)\in V_{31}'\}, &
 T_{22}&=\{q\in [t]:\phi_{2*}(q)\in V_{23}', \; \phi_{*2}(q)\in V_{32}'\}.
\end{align*}
Then, let us write $Z_{ij}=V_{ij}'\cup T_{ij}$ for $i,j \in [2]$, and $\mathcal{Z}=\{Z_{ij}:i,j \in [2]\}$.
Notice that $\mathcal{Z}$ is a partition of $V(J_2)$.
By using $|T_{11}|+|T_{12}|=|V_{13}'|$, we deduce that $|Z_{1*}|=|V_{1*}'|\geq |J_2|/3$. 
More generally, for $i \in \{1,2\}$
\begin{align*}
|Z_{i*}| =  |Z_{*i}| = |V'_{i*}| \geq |J_2|/3.
\end{align*}
Note that $ Z_{12}\cup Z_{21} \supseteq V_{12}'\cup V_{21}'  \ne \emptyset$ by~\eqref{eq:newpartition2}.
We deduce that $|Z_{12}|=|Z_{21}|>0$.
On the other hand, for $i\in\{1,2\}$, we have $Z_{i*}=V_{i1}'\cup V_{i2}'\cup \{q\in[t]:\phi_{2*}(q)\in V_{i3}' \}$ and $Z_{*i}=V_{1i}'\cup V_{2i}'\cup \{q\in[t]:\phi_{*2}(q)\in V_{3i}' \}$. Since $\phi_{2*}(x)=\phi_{*2}(x)=x$ for all $x\in V_{11}'\cup V_{12}'\cup V_{21}'\cup V_{22}'$, we see that $\phi_{2*}(V_{i1}'\cup V_{i2}')=V_{i1}'\cup V_{i2}'$ and $\phi_{*2}(V_{1i}'\cup V_{2i}')=V_{1i}'\cup V_{2i}'$. Therefore,
\begin{align}
\label{eq:GJ-correspondence}
xy\in E(Z_{i*},Z_{*j}) \:\:\text{ if and only if }\:\: \phi_{2*}(x)\phi_{*2}(y)\in E(V_{i*}',V_{*j}') \:\:\text{ for all } i,j \in [2].
\end{align} 
Then, we have 
\begin{align}
\label{eqn:e(Z_i*,Z*j)}
e(Z_{i*},Z_{*j})=e(V_{i*}',V_{*j}') \text{ for }i,j \in [2].
\end{align}
Hence, we obtain $$|\mathcal{B}_2(\mathcal{Z},J_2)|=e(V_{1*}',V_{*2}')+e(V_{2*}',V_{*1}')\leq  |\mathcal{B}_3(\mathcal{P}',G')| \leq  2\gamma |G'|^2\leq 8\gamma |J_2|^2.$$ As a result, $\mathcal{Z}$ is a $(4,1/3,8\gamma)$-partition for $J_2$ with $|Z_{12}|=|Z_{21}|>0$.\\

\noindent Let $B^{+}(J_2)$ be the set of vertices in $J_2$ satisfying $d_{\mathcal{G}_2(\mathcal{Z},J_2)}^{+}(x)<(1/3+\eps/3)|J_2|$. Similarly, define $B^{-}(J_2)$. Note that $d_{\mathcal{G}_2(\mathcal{Z},J_2)}^{+}(x)= d_{\mathcal{G}_3(\mathcal{P}',G')}^{+}(\phi_{2*}(x))$ for any vertex $x\in V(J_2)$ by \eqref{eq:GJ-correspondence}. Moreover, by~\eqref{eq:degreeG'} and \eqref{eqn:|J_2|}, we have $\delta^{0}(G')\geq (1/3+\eps/2)|J_2|$. Hence, by \eqref{eqn:|J_2|}, for any vertex $x\in B^{+}(J_2)$, we obtain
\begin{align*}
d_{\mathcal{B}_3(\mathcal{P}',G')}^{+}(\phi_{2*}(x))> (1/3+\eps/2)|J_2|- (1/3+\eps/3)|J_2|\geq \eps|J_2|/6\geq \eps n/12. \end{align*}
\noindent Since $|\mathcal{B}_3(\mathcal{P}',G')|\leq |\mathcal{B}_3(\mathcal{P},G)| \leq \gamma n^2$, we find $\gamma n^2\geq |B^{+}(J_2)| \eps n/12$. So $|B^{+}(J_2)|\leq 12\gamma n /\eps$ and, similarly, $|B^{-}(J_2)|\leq 12\gamma n/\eps $. As a result, by \eqref{eqn:|J_2|}, we have
\begin{align*}
|B^{+}(J_2)|+|B^{-}(J_2)|\leq 24\gamma n/\eps\leq 48\gamma |J_2|/\eps \leq \sqrt{\gamma}|J_2|. 
\end{align*}

\noindent On the other hand, by \eqref{eq:G'-min-good-degree} and \eqref{eqn:|J_2|}, for any vertex $x\in V(J_2)$, we obtain \begin{align*}
d_{\mathcal{G}_2(\mathcal{Z},J_2)}^{+}(x)=d_{\mathcal{G}_3(\mathcal{P}',G')}^{+}(\phi_{2*}(x))\geq d/3-16\gamma n\geq |J_2|/20.
\end{align*}
Similarly, we have $d_{\mathcal{G}_2(\mathcal{Z},J_2)}^{-}(x)\geq |J_2|/20$. As a result, the partition $\mathcal{Z}$ for the digraph $J_2$ satisfies all the conditions of Lemma~\ref{lem:digraph-hamilton}, so we are done.\end{proof}

\section{Proofs of main results}\label{ch:mainproofs}

In this section, we give the proofs of Theorems~\ref{thm:MAIN-1-3-CASE} and~\ref{thm:MAIN-1-4-CASE}. 

\begin{proof}[Proof of Theorem~\ref{thm:MAIN-1-3-CASE}]
Let $\eps>0$ be a constant. Let $G$ be a strongly well-connected $d$-regular digraph on $n$ (sufficiently large) vertices with $d\geq (1/3+\eps)n$. We will show that $G$ is Hamiltonian.
Let $\nu$ and $\tau$ be constants satisfying $1/n \ll \nu\ll\tau\ll \eps$.
If $G$ is a robust $(\nu, \tau)$-outexpander, then we are done by Theorem~\ref{thm:RobustExpanderImpliesHamilton}. Assume not. Then, $G$ admits a $(4,\tau,4\nu)$-partition by Lemma~\ref{lem:structure_not_expander}. Let $\mathcal{P}=\{V_{ij}:i,j \in [2]\}$ be an extremal $(4,\tau,4\nu)$-partition for $G$. Notice that $|V_{1*}|,|V_{2*}|\geq (1/3+\eps/2)n$ by Proposition~\ref{prop:dense-regular-size}. Without loss of generality, assume $|V_{12}|\geq |V_{21}|$. We will choose a path system $\mathcal{Q}$ in $\mathcal{B}_2(\mathcal{P},G)$ satisfying 
\[
|E(\mathcal{Q})\cap E(V_{1*},V_{*2})|-|E(\mathcal{Q})\cap E(V_{2*},V_{*1})|=|V_{12}|-|V_{21}|
\]
as follows.
\begin{itemize}
\item[(i)] If $V_{12}=V_{21}=\emptyset$, then $|V_{11}|,|V_{22}|\geq (1/3+\eps/2)n$. Since $G$ is strongly well-connected, we can find disjoint edges $ab\in E(V_{11},V_{22})$ and $cd\in E(V_{22},V_{11})$. Then, set $\mathcal{Q}=\{ab,cd\}$.
\item[(ii)] If $|V_{12}|\geq |V_{21}|>0$, then  $d(|V_{12}|-|V_{21}|)=e(V_{1*},V_{*2})-e(V_{2*},V_{*1})$ by Proposition~\ref{prop:partition-regular-graphs}. Hence, we have $e(V_{1*},V_{*2})\geq d(|V_{12}|-|V_{21}|)$. By Proposition~\ref{prop:degree-property-extremal}, $E(V_{1*},V_{*2})$ induces a subgraph $H$ in $G$ with $\Delta^{0}(H)\leq d/2$. Since $e(V_{1*},V_{*2})\leq 4\nu n^2$, by Lemma~\ref{lem:CYCLE_FREE}, we can find a path system $\mathcal{Q}'$ in $H$ with $e(\mathcal{Q}')\geq 2e(V_{1*},V_{*2})/d\geq 2(|V_{12}|-|V_{21}|)$. Then, we remove all but exactly $|V_{12}|-|V_{21}|$ edges in $\mathcal{Q}'$ to obtain $\mathcal{Q}$.     

\item[(iii)] If $|V_{12}|\geq 2$ and $|V_{21}|=0$, then as with the previous case, $E(V_{1*},V_{*2})$ has a path system $\mathcal{Q}'$ containing $2|V_{12}|$ edges. We claim $\mathcal{Q}'$ has at least one path that starts in $V_{11}$ and ends in $V_{22}$. If not, then any path in $\mathcal{Q}'$, with $s$ edges say, is incident to at least $s$ vertices in $V_{12}$, but since $\mathcal{Q}'$ contains more than $|V_{12}|$ edges, we have a contradiction. 
Next we claim that any path in $\mathcal{Q}'$ from $V_{11}$ to $V_{22}$ has at most $|V_{12}|$ edges. Indeed, if not, then $\mathcal{Q}'$ has a unique path which has $|V_{12}| + 1$ edges. But then $\mathcal{Q}'$ has $2|V_{12}| = |V_{12}| + 1$ edges, contradicting $|V_{12}| \geq 2$.
 Using the claims, we can remove all but exactly $|V_{12}|$ edges in $\mathcal{Q}'$ to obtain a path system $\mathcal{Q}$ with exactly $|V_{12}| = |V_{12}| - |V_{21}|$ edges and where at least one path starts in $|V_{11}|$ and ends in $|V_{22}|$. 
\item[(iv)] If $|V_{12}|=1$ and $|V_{21}|=0$, let $x$ be the unique vertex in $V_{12}$. By Proposition~\ref{prop:partition-regular-graphs}, we have $d=d_{V_{11}}^{-}(x)+d_{V_{22}}^{+}(x)+e(V_{11},V_{22})-e(V_{22},V_{11})$. Note that, by Proposition~\ref{prop:degree-property-extremal}, we know $d_{V_{11}}^{-}(x),d_{V_{22}}^{+}(x)\leq d/2$. If $d_{V_{11}}^{-}(x)=d_{V_{22}}^{+}(x)=d/2$, then we obtain another extremal $(4,\tau,4\nu)$-partition by moving $x$ into $V_{11}$, which results in case (i). If we have either $d_{V_{11}}^{-}(x)<d/2$ or $d_{V_{22}}^{+}(x)<d/2$, then we have $e(V_{11},V_{22})\geq 1$. We can take an arbitrary edge $ab\in E(V_{11},V_{22})$, and set $\mathcal{Q}=\{ab\}$.
\end{itemize}

\no Now we contract this path system $\mathcal{Q}$ in $G$ with respect to partition $\mathcal{P}$ to obtain a graph $G'$ with resulting partition $\mathcal{P}'=\{V_{ij}':i,j \in [2]\}$.
By Lemma~\ref{lem:BALANCING-PARTITION}, we have $|V_{12}'|=|V_{21}'|$. 
Moreover, the choice of $\mathcal{Q}$ ensures that both $V_{12}'$ and $V_{21}'$ are nonempty as follows: In cases (i), (iii), and (iv) we include at least one path from $V_{11}$ to $V_{22}$ so that the vertex created when contracting this path is placed in $V'_{21}$; see Definition~\ref{def:path-contraction}. In case (ii), $V_{21}$ is nonempty and we do not use any vertices from $V_{21}$ in the path system. Therefore, $V_{12}'$ is nonempty after the contraction, which also means that $V_{12}'$ is nonempty as $|V_{12}'|=|V_{21}'|$.
We note that $\mathcal{Q}$ has at most $12\nu n$ edges since, by construction, $\mathcal{Q}$ has at most $||V_{12}| - |V_{21}||$ edges (except in case (i) where $\mathcal{Q}$ has two edges) and $|V_{12}|-|V_{21}|\leq 12\nu n$ by Corollary~\ref{cor:Y-and-Z-equal-size}. Therefore, we delete at most $24\nu n$ vertices, which implies $\delta^{0}(G')\geq d-24\nu n$. On the other hand, by Proposition~\ref{prop:after-contraction}, we have that $\mathcal{P}'$ is a $(4,\tau/2,8\nu)$-partition.  Also, by Proposition~\ref{prop:dense-regular-size}, we have 
\begin{align*}
|V_{i*}'|\geq |V_{i*}|-24\nu n\geq (1/3+\eps-24\nu)n\geq |G'|/3
\end{align*}
for $i \in [2]$. Similarly, we obtain $|V_{*i}'|\geq |G'|/3$, so $\mathcal{P}'$ is a $(4,1/3,8\nu)$-partition. Let $B^{+}(G')$ be the set of vertices in $G'$ satisfying $d_{\mathcal{G}_2(\mathcal{P}',G')}^{+}(x)<(1/3+\eps/3)|G'|$. Similarly, define $B^{-}(G')$. Note that $\delta^{0}(G')\geq d-24\nu n\geq (1/3+\eps/2)|G'|$. Hence, for any vertex $x\in B^{+}(G')$, we obtain
\begin{align*}
d_{\mathcal{B}_2(\mathcal{P}',G')}^{+}(x)>(1/3+\eps/2)|G'|-(1/3+\eps/3)|G'|\geq \eps |G'|/6.    
\end{align*}
Since $|\mathcal{B}_2(\mathcal{P}',G')|\leq 8\nu |G'|^2$, we have $8\gamma |G'|^2\geq |B^{+}(G')|\cdot \eps |G'|/6$. So, $|B^{+}(G')|\leq 48\nu |G'|/\eps$ and, similarly, $|B^{-}(G')|\leq 48\nu |G'|/\eps$. As a result, we obtain
\begin{align*}
|B^{+}(G')|+|B^{-}(G')|\leq 96\nu |G'|/\eps\leq \sqrt{\nu}|G'|.    
\end{align*}

\noindent Moreover, since $\mathcal{P}$ is an extremal $(4,\tau,4\nu)$-partition and we deleted at most $24\nu n$ vertices, Proposition~\ref{prop:degree-property-extremal} implies that $d_{\mathcal{G}_2(\mathcal{P}',G')}^{+}(v),d_{\mathcal{G}_2(\mathcal{P}',G')}^{-}(v)\geq d/2-24\nu n\geq |G'|/20$ for all $v\in V(G')$. As a result, $\mathcal{P}'$ satisfies the properties in Lemma~\ref{lem:digraph-hamilton}, so $G'$ is Hamiltonian. Hence, the result follows by Proposition~\ref{prop:contract-hamilton}.
\end{proof}

\begin{proof}[Proof of Theorem~\ref{thm:MAIN-1-4-CASE}]
Let $\eps>0$ be a constant. Let $G$ be a $d$-regular oriented graph on $n$ (sufficiently large) vertices with $d\geq (1/4+\eps)n$. We will show that $G$ is Hamiltonian.
Fix constants $\nu$ and $\tau$ satisfying $1/n \ll \nu\ll\tau\ll \eps$. 
By Theorem~\ref{thm:RobustExpanderImpliesHamilton}, we are done if $G$ is a robust $(\nu, \tau)$-outexpander. Assume not. Then, by Lemma~\ref{lem:structure_not_expander}, $G$ admits an extremal $(4,\tau,4\nu)$-partition~$\mathcal{P}=\{V_{ij}:i,j \in [2]\}$.
Then, for each $i \in [2]$, 
\begin{align}
        |V_{i*}|,|V_{*i}|\geq (1/4+\eps/2)n \label{eqn:123}
\end{align}
by Proposition~\ref{prop:dense-regular-size}. Also, we have $|V_{12}|,|V_{21}|\geq \tau n$ by Proposition~\ref{prop:quarterleadsnonempty}. 
Without loss of generality, assume $|V_{11}|\leq |V_{22}|$. 
Furthermore, by reversing the edges if necessary, we may assume that $|V_{12}| \ge |V_{21}|$.
Let $r= |V_{12}|-|V_{21}|$.
By Corollary~\ref{cor:Y-and-Z-equal-size} and the fact that $|\mathcal{B}_2(\mathcal{P},G)|\leq 4\nu n^2$, we obtain 
\begin{align}
\label{eqn:r}
r\leq 4\nu n^2/d\leq 16\nu n.
\end{align}
Fix a subset~$R$ of~$V_{12}$ of size~$r$.
Let $\mathcal{W}=\{W_{ij}:i,j \in [2]\}$ where $W_{ij}=V_{ij} \setminus R$ for $i,j \in [2]$.
Note that $|V_{11}|\leq |V_{22}|$ and $|V_{12}|\geq|V_{21}|$ imply that $|V_{*2}|\geq n/2$. Hence
\begin{align}
\label{eqn:|W_2*|}
    |W_{2*}|&=|W_{*2}|\geq (n-r)/2 \\
\label{eqn:|W_1*|}    
    \text{ and } |W_{1*}| &= |W_{*1}| \geq |V_{21}| \geq \tau n.
\end{align}
We now split into cases depending on whether, for all proper $2$-pairs~$\phi^2$ with respect to~$\mathcal{W}$, the digraph $\mathcal{J}^{2}(\mathcal{W},G-R,\phi^2)$ is a robust $(\nu^{1/2}, \tau)$-outexpander or not. 
\\

\noindent\underline{Case 1:} Suppose that, for all $2$-pairs~$\phi^2$ with respect to~$\mathcal{W}$, $\mathcal{J}^{2}(\mathcal{W},G-R,\phi^2)$ is a robust $(\nu^{1/2}, \tau)$-outexpander. 
Recall that $G_{ij} = G[V_{i*},V_{*j}]$. 
We have by Proposition~\ref{prop:partition-regular-graphs} that
\begin{align*}
    e(G_{12}) \ge e(G_{12})-e(G_{21}) =d(|V_{12}|-|V_{21}|)= d r.
\end{align*}
By~Proposition~\ref{prop:degree-property-extremal}, $\Delta^{0}(G_{12})\leq d/2$.
Moreover, we have $e(G_{12}) \le |\mathcal{B}_2(\mathcal{P},G)|\leq 4\nu n^2$. Hence, by Lemma~\ref{lem:CYCLE_FREE}, $G_{12}$ has a path system $\mathcal{Q}$ with $r$ edges.\\

\no We contract $\mathcal{Q}$ in~$G$ with respect to $\mathcal{P}$ to obtain $G'$ with resulting partition $\mathcal{P}'=\{V_{ij}':i,j \in [2]\}$. 
Since $|E(\mathcal{Q})\cap E(G_{12})|-|E(\mathcal{Q})\cap E(G_{21})|=|V_{12}|-|V_{21}|$, Lemma~\ref{lem:BALANCING-PARTITION} implies that $|V_{12}'|=|V_{21}'|$.
Moreover, by Proposition~\ref{prop:quarterleadsnonempty}, we have $|V_{12}|,|V_{21}|\geq \tau n$. Since $r\leq 16\nu n\ll \tau n$, we conclude that $|V_{12}'|=|V_{21}'|>0$. 
By Proposition~\ref{prop:contract-hamilton}, it is enough to show that $G'$ is Hamiltonian. \\

\no Consider a proper $i$-pair~$\psi^i$ with respect to $\mathcal{P}'$ for $i \in [2]$.
Let $\mathcal{J}_i = \mathcal{J}^{i}(\mathcal{P}',G',\psi^i)$.
To show that $G'$ is Hamiltonian, by Proposition~\ref{prop:identification-Hamilton-implies-Hamilton}, it suffices to show that $\mathcal{J}_1$ and $\mathcal{J}_2$ are Hamiltonian.
\\

\noindent
We first prove that $\mathcal{J}_2$ is a robust $(\nu^{1/2}/2,2\tau)$-outexpander by showing it is a small perturbation of $\mathcal{J}^{2}(\mathcal{W},G-R,\phi^2)$ for a suitable proper $2$-pair $\phi^2$ with respect to $\mathcal{W}$, chosen as follows. Let $t = |W_{12}|$ and $t' = |V_{12}'|$. 
Recall that $\psi^2$ is a function from $[t'] \cup V_{22}'$ to $V_{2*}' \times V_{*2}'$. Pick $\phi^2$ among all proper $2$-pairs with respect to $\mathcal{W}$ such that $|X|$ is as large as possible where $X$ is the set of $x\in ([t] \cup W_{22}) \cap ([t'] \cup V_{22}')$ satisfying $\phi^2(x)=\psi^2(x)$. We define $\mathcal{J}^2 = \mathcal{J}^{2}(\mathcal{W},G-R,\phi^2)$.
\\

\no
We have that $V(\mathcal{J}_2) = [t'] \cup V_{22}'$ and $V(\mathcal{J}^2) = [t] \cup W_{22}$ and $X \subseteq V(\mathcal{J}_2) \cap V(\mathcal{J}^2)$. Moreover $\mathcal{J}_2[X] = \mathcal{J}^2[X]$; to see this, note that for $Y:=\phi^2(X) = \psi^2(X) \subseteq V(G)$, the partitions $\mathcal{P}'$ and $\mathcal{W}$ are the same on $G'[Y] = (G-R)[Y]$.\\

\no
First note that 
\begin{align*}
    |\mathcal{J}^2|=|W_{*2}|\geq (n-r)/2.
\end{align*}

\no
Since 
\begin{align*}
    |W_{2*} \triangle V'_{2*} |, |W_{*2} \triangle V'_{*2} | \le 3e(\mathcal{Q}) + |R| = 4r ,
\end{align*}
we deduce that 
\begin{align*}
\left| ([t] \cup W_{22}) \setminus X \right|, \left| ([t'] \cup V'_{22}) \setminus X \right| \le 16r.
\end{align*}
Therefore
\begin{align*}    
    | V(\mathcal{J}_2) \triangle V(\mathcal{J}^2) | \le | ([t] \cup W_{22}) \setminus X | + | ([t'] \cup V'_{22}) \setminus X | \le 32r \le 512 \nu n \le  \nu^{1/2}  |\mathcal{J}^2|/2.
\end{align*}
Since $\mathcal{J}^2$ is a robust $(\nu^{1/2} , \tau)$-outexpander by assumption, we conclude that $\mathcal{J}_2$ is a robust $(\nu^{1/2}/2,2\tau)$-outexpander by Lemma~\ref{lem:SYMMETRIC-DIFFERENCE} (where $\mathcal{J}_2 \cup \mathcal{J}^2$ plays the role of $G$). Also, by Proposition~\ref{prop:degree-property-extremal}, we know that $d^+_{\mathcal{G}_2(\mathcal{P},G)},d^-_{\mathcal{G}_2(\mathcal{P},G)}\geq d/2$. Then, since $e(\mathcal{Q})=r$, we have $d^{+}_{\mathcal{G}_2(\mathcal{P}',G')}(v),d^{-}_{\mathcal{G}_2(\mathcal{P}',G')}(v)\geq d/2-2r$ for all $v\in V(G')$, which shows 
\[
\delta^{0}(\mathcal{J}_2)\geq d/2-2r\geq n/10 \geq |\mathcal{J}_2|/10
\]
by using $d\geq (1/4+\eps)n$, $r\leq 16\nu n$ and $\nu \ll \eps$.
Hence, $\mathcal{J}_2$ is Hamiltonian by Theorem~\ref{thm:RobustExpanderImpliesHamilton}.\\

\no
We now show that $\mathcal{J}_1$ is Hamiltonian.
By \eqref{eqn:123} and \eqref{eqn:r}, we have $ n/4 \le |V_{1*}|-2r \le |V_{1*}'|=|\mathcal{J}_1|$. Also, since $|V_{11}|\leq |V_{22}|$ and $|V_{12}|-|V_{21}|=r$, we have $|V_{1*}'|\leq |V_{1*}|+r\leq (n+r)/2+r\leq(1/2+\tau)n$ as $\nu\ll \tau$. Then, we obtain 
\begin{align*}
n/4\leq |V_{1*}'|=|\mathcal{J}_1|\leq (1/2+\tau)n.
\end{align*}
Similarly as above, by Proposition~\ref{prop:degree-property-extremal},
we have $\delta^{0}(\mathcal{J}_1)\geq d/2-2r\geq |\mathcal{J}_1|/10$.
By Proposition~\ref{prop:after-contraction}, $\mathcal{P}'$ is a $(4,\tau/2,8\nu)$-partition of~$G'$. 
Also
\begin{align*}
    e(\mathcal{J}_1)&= e(V_{1*}',V_{*1}')
    \geq \delta^{0}(G')|V_{1*}'|- |\mathcal{B}_2(\mathcal{P}',G')| 
    \\
    &\geq (d-2r)|V_{1*}'|-  8\nu n^2 \ge (d -64\nu n)|\mathcal{J}_1|
\end{align*}
as $|V_{1*}'|=|\mathcal{J}_1|\geq n/4$.
Let $B^{+}(\mathcal{J}_1)$ be the set of vertices in $\mathcal{J}_1$ satisfying $d_{\mathcal{J}_1}^{+}(v)<d -\eps n/4$. Similarly define $B^{-}(\mathcal{J}_1)$. Since $\Delta^0(\mathcal{J}_1) \le d$, we obtain
\begin{align*}
\left(|\mathcal{J}_1|-|B^{+}(\mathcal{J}_1)|\right)d+|B^{+}(\mathcal{J}_1)|\left(d-\eps n/4\right)\geq e(\mathcal{J}_1)\geq  (d -64\nu n)|\mathcal{J}_1|,
\end{align*}
which implies $64\nu n|\mathcal{J}_1|\geq |B^{+}(\mathcal{J}_1)|\eps n/4$. Hence, we have $|B^{+}(\mathcal{J}_1)|\leq 256\nu |\mathcal{J}_1|/\eps$, and similarly, $|B^{-}(\mathcal{J}_1)|\leq 256\nu |\mathcal{J}_1|/\eps$. As a result, we obtain
\begin{align*}
|B^{+}(\mathcal{J}_1)|+ |B^{-}(\mathcal{J}_1)| \leq 512\nu |\mathcal{J}_1|/\eps \leq \sqrt{\nu}|\mathcal{J}_1|  
\end{align*}
as $\nu \ll \eps$. Hence, for all but at most $\sqrt{\nu} |\mathcal{J}_1|$ vertices $v \in V(\mathcal{J}_1)$, we have 
\begin{align*}
    d_{\mathcal{J}_1}^{+}(v),d_{\mathcal{J}_1}^{-}(v)\geq (d -\eps n/4) \ge (1/2+\eps/3)|\mathcal{J}_1|.
\end{align*}

\noindent Therefore, $\mathcal{J}_1$ satisfies the conditions of Corollary~\ref{cor:Higher-Than-Half-For-All-But-Few}, so it is Hamiltonian.\\

\noindent\underline{Case 2:} 
Suppose that there exists a $2$-pair~$\phi^2 = (\phi_{2*},\phi_{*2})$ with respect to~$\mathcal{W}$ such that  $\mathcal{J}^{2}(\mathcal{W},G-R,\phi^2)$ is not a robust $(\nu^{1/2}, \tau)$-outexpander. 
Let $\mathcal{J}^2 = \mathcal{J}^{2}(\mathcal{W},G-R,\phi^2)$.
We now show that there is a $(9, \tau/6, 20 \nu^{1/2})$-partition for~$G$ (so that we can apply Lemma~\ref{lem:9partitionHamiltonian}).\\

\no 
Note that it suffices to show that $G-R$ admits a $(9,\tau/3,10\nu^{1/2})$-partition since $|R|=r\leq 16\nu n$ (so we can arbitrarily add the vertices of $R$ into those 9 parts, which would cause a small amount of increase in the number of bad edges). Recall that $\mathcal{J}^2$ is a digraph on $[t]\cup W_{22}$ where $t=|W_{12}|=|W_{21}|$, and $\phi_{2*}:[t]\cup W_{22}\to W_{2*}$, $\phi_{*2}:[t]\cup W_{22}\to W_{*2}$ are bijections satisfying $\phi_{2*}(x)=\phi_{*2}(x)=x$ for all $x\in W_{22}$. First we show $\mathcal{J}^2$ is almost regular, so it admits a $(4,\tau,4\gamma^{1/2})$-partition by using Lemma~\ref{lem:structure_not_expander} since we assumed it is not a robust $(\nu^{1/2},\tau)$-outexpander. Note that any partition $\{U_{ij}:i,j\in[2]\}$ of $V(\mathcal{J}^2)$ also gives a $4$-partition for $W_{22}$. Similarly, $\{U_{ij}:i,j\in[2]\}$ partitions $W_{12}$ (resp. $W_{21}$) into 2 parts depending on $y\in U_{1*}$ or $y\in U_{2*}$ (resp. $y\in U_{*1}$ or $y\in U_{*2}$) for each $y\in[t]$, so we obtain a $9$-partition of $G-R$. Then we show bad edges in this $9$-partition (almost) correspond to $\mathcal{B}_2(\mathcal{W},G-R)$, so we can find an upper bound for the number of them. \\

\no
Let $\theta=d/|W_{2*}|$. 
Remove any loops in $\mathcal{J}^2$.
Notice that we have $|\mathcal{J}^2| = |W_{2*}|$ and $\Delta^{0}(\mathcal{J}^2)\leq d=\theta |W_{2*}|$. 
Since $W_{2*} = V_{2*}$ and $W_{*2} = V_{*2}-R$,  we have by~\eqref{eqn:|W_2*|}
\begin{align*}
    e(\mathcal{J}^2) & \ge e(W_{2*}, W_{*2}) - n = e(V_{2*}, W_{*2}) -n \geq d|W_{*2}| - e(\mathcal{B}_2(\mathcal{P}, G)) -n
    \\
    &\ge  d|W_{2*}|-\nu n^2 -n
    \geq (\theta-\nu^{1/2})|W_{2*}|^2.
\end{align*}
By Lemma~\ref{lem:structure_not_expander}, $\mathcal{J}^2$ admits a $(4, \tau, 4 \nu^{1/2})$-partition $\mathcal{P}^*_2 = \{U_{ij}:i,j \in [2]\}$.
Let $X_i = \phi_{2*}(U_{i*})$ and $Y_i = \phi_{*2}(U_{*i})$ for $i \in [2]$.
Hence we have 
\begin{align}
\label{eqn:X_i,Y_i}
|X_1|,|X_2|, |Y_1|, |Y_2| & \geq \tau |W_{2*}| \ge \tau(n-r)/2\geq
\tau n/3, \\
\label{eqn:e(XY)}
e_{G}(X_1,Y_2)+e_{G}(X_2,Y_1) & \leq \mathcal{B}_2(\mathcal{P}^*_2,\mathcal{J}^2)+|W_{2*}| \le 5\nu^{1/2} |W_{2*}|^2,
\end{align}
where we have used~\eqref{eqn:|W_2*|} and \eqref{eqn:r} for the first line.
Then, let us define the partition $\mathcal{Z}=\{Z_{ij}:i,j \in [3]\}$ for $G-R$ as follows:
\[
\begin{array}{lll}
 Z_{11}=W_{22}\cap X_1\cap Y_1, &Z_{12}=W_{22}\cap X_1\cap Y_2, &Z_{13}=W_{21}\cap X_1, \\
 Z_{21}=W_{22}\cap X_2\cap Y_1, &Z_{22}=W_{22}\cap X_2\cap Y_2, &Z_{23}=W_{21}\cap X_2, \\
  Z_{31}=W_{12}\cap Y_1, &Z_{32}=W_{12}\cap Y_2, &Z_{33}=W_{11}.
\end{array} 
\]
Notice that, for $i \in [2]$
\begin{align*}
    |Z_{i*}|=|X_i| \geq \tau n/3 \text{ and } |Z_{*i}|=|Y_i| \ge \tau n/3
\end{align*}
by~\eqref{eqn:X_i,Y_i}.
Also, by \eqref{eqn:|W_1*|}, we have $|Z_{3*}|=|W_{1*}|\geq  \tau n/3$ and $|Z_{*3}|=|W_{*1}|\geq \tau n/3$.
Note that
\begin{align*}
    \mathcal{B}_3(\mathcal{Z},G-R) & \subseteq E_{G}(X_1,Y_2)  \cup E_{G}(X_2,Y_1) \cup \bigcup_{i,j\neq 3} \left( E_G(Z_{i*},Z_{*3}) \cup E_G(Z_{3*},Z_{*j})\right) \\
    & = E_{G}(X_1,Y_2)  \cup E_{G}(X_2,Y_1) \cup \mathcal{B}_2(\mathcal{W},G-R) \\
    &\subseteq E_{G}(X_1,Y_2)  \cup E_{G}(X_2,Y_1) \cup \mathcal{B}_2(\mathcal{P},G),
\end{align*}
so \eqref{eqn:e(XY)} implies that $|\mathcal{B}_3(\mathcal{Z},G-R)|  \le 5 \nu^{1/2} |W_{2*}|^2 + 4 \nu n^2 \le 10 \nu^{1/2} |G-R|^2$.
Therefore, $\mathcal{Z}$ is a $(9,\tau/3,10\nu^{1/2})$-partition for~$G-R$. 
Let us distribute the vertices of~$R$ into elements of~$\mathcal{Z}$ arbitrarily.
Since $r\leq 16\nu n\ll \tau n$, the modified version of $\mathcal{Z}$ becomes a $(9,\tau/6, 20\nu^{1/2})$-partition for~$G$. Hence, by Lemma~\ref{lem:9partitionHamiltonian}, $G$ is Hamiltonian, as required.\end{proof}

\section{Conclusion}
\label{ch:conclusion}

\no The main result of this paper is a proof of the approximate version of Jackson's conjecture, namely Conjecture~\ref{conj:JacksonMain}. It remains an open problem to prove this conjecture exactly. Similarly, it would be interesting (and probably easier) to obtain an exact version of Theorem~\ref{thm:MAIN-1-3-CASE}, namely to show that every strongly well-connected $n$-vertex $d$-regular digraph with $d \geq n/3$ is Hamiltonian. \\

\noindent
Another natural question is to ask for the analogue of Theorem~\ref{thm:MAIN-1-3-CASE} for oriented graphs.
By suitably orienting the edges in a  non-Hamiltonian $2$-connected regular graph on $n$ vertices with degree
close to~$n/3$ (see e.g. \cite{kuhn2016solution}), there exist non-Hamiltonian strongly well-connected regular oriented graphs on~$n$ vertices with $d$ close to~$n/6$. 

\begin{proposition}  \label{prop:wellconncetedexample}
	For $n \in \mathbb{N}$, there exists a strongly well-connected $3n$-regular oriented graph on $18n+5$ vertices with no Hamilton cycle.
\end{proposition}

\begin{proof}
    Let $G_1$, $G_2$ and $G_3$ be vertex-disjoint regular tournaments each on $(6n+1)$ vertices. 
    For $i \in [3]$, let $M_i = \{ x^{i}_j y^{i}_j : j \in [2n] \}$ be a matching of size~$2n$ in~$G_i$. 
    Define $G$ to be the oriented graph obtained from $\bigcup_{i \in [3]} (G_i - M_i)$ by adding two new vertices $z$ and $z'$ and edge set $\{ x^{i}_j z, z y^{i}_j, x^{i}_{j+n} z', z' y^{i}_{j+n} : i \in [3], j \in [n] \}$.
    Note that $G$ is a $3n$-regular oriented graph on $18 n +5$ vertices. 
    We claim that $G$ is strongly well-connected. Indeed, $G$ has a cycle with vertex set $V(G_1) \cup V(G_2) \cup \{z, z' \}$ and another cycle with vertex set $V(G_3) \cup \{z, z'\}$. The union of these two cycles (which is a subdigraph of $G$) is already strongly well-connected; hence $G$ is strongly well-connected.
    However $G$ is not Hamiltonian because deleting the two vertices $z$ and $z'$ from $G$ disconnects it into $3$ components (whereas deleting any $2$ vertices from a Hamilton cycle disconnects it into at most $2$ components). 
\end{proof}

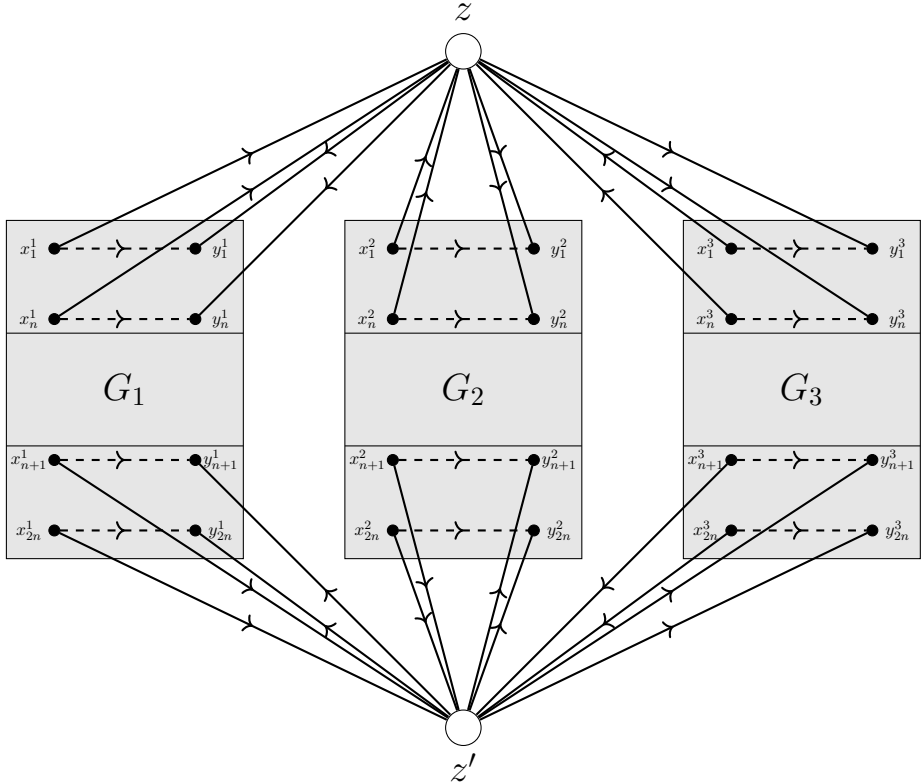
\begin{figure}[h]
    \centering
    \begin{tikzpicture}[scale=0.75]

\draw[fill=SHADEDGRAY]   (-2.1,11) -- (-2.1,5) -- (2.1,5) -- (2.1,11) -- cycle ;
\draw (-2.1,9) -- (2.1,9);
\draw (-2.1,7) -- (2.1,7);
\draw[fill=SHADEDGRAY]   (-8.1,11) -- (-8.1,5) -- (-3.9,5) -- (-3.9,11) -- cycle ;
\draw (-8.1,9) -- (-3.9,9);
\draw (-8.1,7) -- (-3.9,7);
\draw[fill=SHADEDGRAY]   (3.9,11) -- (3.9,5) -- (8.1,5) -- (8.1,11) -- cycle ;
\draw (3.9,9) -- (8.1,9);
\draw (3.9,7) -- (8.1,7);

\node [style=whitecircle,scale=1.3] (1) at (0, 14) {};
\node [style=whitecircle,scale=1.3] (2) at (0, 2) {};

\node at (0,14.7) {\Large$z$};
\node at (0,1.3) {\Large$z'$};

\node at (-6,8) {\Large$G_1$};
\node at (0,8) {\Large$G_2$};
\node at (6,8) {\Large$G_3$};

\node [style=blackcircle,scale=0.4] (3) at (-1.25, 10.5) {};
\node [style=blackcircle,scale=0.4] (4) at (1.25, 10.5) {};
\node [style=blackcircle,scale=0.4] (7) at (-1.25, 9.25) {};
\node [style=blackcircle,scale=0.4] (8) at (1.25, 9.25) {};
\node [scale=0.6] at (-1.7,10.5) {$x_1^2$};
\node [scale=0.6] at (1.7,10.5) {$y_1^2$};
\node [scale=0.6] at (-1.7,9.25) {$x_n^2$};
\node [scale=0.6] at (1.7,9.25) {$y_n^2$};
\node [style=blackcircle,scale=0.4] (9) at (-1.25, 6.75) {};
\node [style=blackcircle,scale=0.4] (10) at (1.25,6.75) {};
\node [style=blackcircle,scale=0.4] (13) at (-1.25, 5.5) {};
\node [style=blackcircle,scale=0.4] (14) at (1.25,5.5) {};
\node [scale=0.6] at (-1.7,6.75) {$x_{n+1}^2$};
\node [scale=0.6] at (1.7,6.75) {$y_{n+1}^2$};
\node [scale=0.6] at (-1.7,5.5) {$x_{2n}^2$};
\node [scale=0.6] at (1.7,5.5) {$y_{2n}^2$};

\node [style=blackcircle,scale=0.4] (333) at (-7.25, 10.5) {};
\node [style=blackcircle,scale=0.4] (334) at (-4.75, 10.5) {};
\node [style=blackcircle,scale=0.4] (337) at (-7.25, 9.25) {};
\node [style=blackcircle,scale=0.4] (338) at (-4.75, 9.25) {};
\node [scale=0.6] at (-7.7,10.5) {$x_1^1$};
\node [scale=0.6] at (-4.3,10.5) {$y_1^1$};
\node [scale=0.6] at (-7.7,9.25) {$x_n^1$};
\node [scale=0.6] at (-4.3,9.25) {$y_n^1$};
\node [style=blackcircle,scale=0.4] (339) at (-7.25, 6.75) {};
\node [style=blackcircle,scale=0.4] (3310) at (-4.75,6.75) {};
\node [style=blackcircle,scale=0.4] (3313) at (-7.25, 5.5) {};
\node [style=blackcircle,scale=0.4] (3314) at (-4.75,5.5) {};
\node [scale=0.6] at (-7.7,6.75) {$x_{n+1}^1$};
\node [scale=0.6] at (-4.3,6.75) {$y_{n+1}^1$};
\node [scale=0.6] at (-7.7,5.5) {$x_{2n}^1$};
\node [scale=0.6] at (-4.3,5.5) {$y_{2n}^1$};

\node [style=blackcircle,scale=0.4] (1333) at (4.75, 10.5) {};
\node [style=blackcircle,scale=0.4] (1334) at (7.25, 10.5) {};
\node [style=blackcircle,scale=0.4] (1337) at (4.75, 9.25) {};
\node [style=blackcircle,scale=0.4] (1338) at (7.25, 9.25) {};
\node [scale=0.6] at (7.7,10.5) {$y_1^3$};
\node [scale=0.6] at (4.3,10.5) {$x_1^3$};
\node [scale=0.6] at (7.7,9.25) {$y_n^3$};
\node [scale=0.6] at (4.3,9.25) {$x_n^3$};
\node [style=blackcircle,scale=0.4] (1339) at (4.75, 6.75) {};
\node [style=blackcircle,scale=0.4] (13310) at (7.25,6.75) {};
\node [style=blackcircle,scale=0.4] (13313) at (4.75, 5.5) {};
\node [style=blackcircle,scale=0.4] (13314) at (7.25,5.5) {};
\node [scale=0.6] at (7.7,6.75) {$y_{n+1}^3$};
\node [scale=0.6] at (4.3,6.75) {$x_{n+1}^3$};
\node [scale=0.6] at (7.7,5.5) {$y_{2n}^3$};
\node [scale=0.6] at (4.3,5.5) {$x_{2n}^3$};


\begin{scope}[thick,decoration={
    markings,
    mark=at position 0.5 with {\arrow{>}}}
    ] 

\draw[postaction={decorate},dashed] (3)--(4);
\draw[postaction={decorate},dashed] (7)--(8);
\draw[postaction={decorate},dashed] (9)--(10);
\draw[postaction={decorate},dashed] (13)--(14);

\draw[postaction={decorate},dashed] (333)--(334);
\draw[postaction={decorate},dashed] (337)--(338);
\draw[postaction={decorate},dashed] (339)--(3310);
\draw[postaction={decorate},dashed] (3313)--(3314);

\draw[postaction={decorate},dashed] (1333)--(1334);
\draw[postaction={decorate},dashed] (1337)--(1338);
\draw[postaction={decorate},dashed] (1339)--(13310);
(13312);
\draw[postaction={decorate},dashed] (13313)--(13314);
\end{scope}

\begin{scope}[thick,decoration={
    markings,
    mark=at position 0.5 with {\arrow{>}}}
    ] 

\draw[postaction={decorate}] (3)--(1);
\draw[postaction={decorate}] (7)--(1);
\draw[postaction={decorate}] (1)--(4);
\draw[postaction={decorate}] (1)--(8);

\draw[postaction={decorate}] (9)--(2);
\draw[postaction={decorate}] (13)--(2);
\draw[postaction={decorate}] (2)--(10);
\draw[postaction={decorate}] (2)--(14);

\draw[postaction={decorate}] (333)--(1);
\draw[postaction={decorate}] (337)--(1);
\draw[postaction={decorate}] (1)--(334);
\draw[postaction={decorate}] (1)--(338);

\draw[postaction={decorate}] (339)--(2);
\draw[postaction={decorate}] (3313)--(2);
\draw[postaction={decorate}] (2)--(3310);
\draw[postaction={decorate}] (2)--(3314);

\draw[postaction={decorate}] (1333)--(1);
\draw[postaction={decorate}] (1337)--(1);
\draw[postaction={decorate}] (1)--(1334);
\draw[postaction={decorate}] (1)--(1338);

\draw[postaction={decorate}] (1339)--(2);
\draw[postaction={decorate}] (13313)--(2);
\draw[postaction={decorate}] (2)--(13310);
\draw[postaction={decorate}] (2)--(13314);

\end{scope}

    \end{tikzpicture}
    \caption{A strongly well-connected $3n$-regular oriented graph $G$ on $18n+5$ vertices}
    \label{fig:EXTREMAL-3n-regular}
\end{figure}

\noindent 
Are all strongly well-connected $d$-regular oriented graphs on $n$ vertices with $d \ge n/6$ Hamiltonian? We note that a version of this question with ``strongly $2$-connected'' in place of ``strongly well-connected'' was asked in~\cite{SurveyDirected}, but Proposition~\ref{thm:counterexample2connected} provides a counterexample for that. \\

\noindent
Another interesting direction is to obtain an analogue of the Bollob{\'a}s--H{\"a}ggkvist Conjecture (which is discussed in the introduction) for oriented graphs. 
That is, given $t \ge 3$, determine the minimum value for $d$ such that any strongly $t$-connected $d$-regular $n$-vertex oriented graph is Hamiltonian.
For any choice of $t$, we must have $d \ge n/8$ by considering a suitable orientation of the example of Jung and of Jackson, Li, and Zhu (mentioned in the Section~\ref{sec:INTRO}), as shown below.

\begin{proposition} \label{prop:3-conncetedexample}
	For $n \in \mathbb{N}$, there exists a strongly $n$-connected $2n$-regular oriented graph on $16n+1$ vertices with no Hamilton cycle.
\end{proposition}

\begin{proof}
    Consider a $2n$-regular oriented bipartite graph~$H$ with vertex classes $A$ and $B$ each of size~$4n$.
    Fix $b \in B$ and let $N^+_H(b) = \{a^+_1, \dots, a^+_{2n}\}$ and $N^-_H(b) = \{a^-_1, \dots, a^-_{2n}\}$.
    Let $G_1$ and $G_2$ be regular tournaments each on $(4n+1)$ vertices.
    Suppose that $V(H)$, $V(G_1)$ and $V(G_2)$ are pairwise disjoint. 
    For $i \in [2]$, let $M_i = \{ x^{i}_j y^{i}_j : j \in [n] \}$ be a matching of size~$n$ in~$G_i$. 
    Define $G$ to be the oriented graph obtained from $(H - \{b\}) \cup G_1 \cup G_2$ by removing the edges from $M_1 \cup M_2$ and adding the edges $\{ x^1_j  a^+_j,  a^-_j y^1_j , x^2_j  a^+_{j+n},  a^-_{j+n} y^2_j : j \in [n]\}$.
    Note that $G$ is a strongly $n$-connected $2n$-regular oriented graph on $16n+1$ vertices. 
    However $G$ is not Hamiltonian as removing $A$ will create $2+(|B|-1) > |A|$ components.
\end{proof}

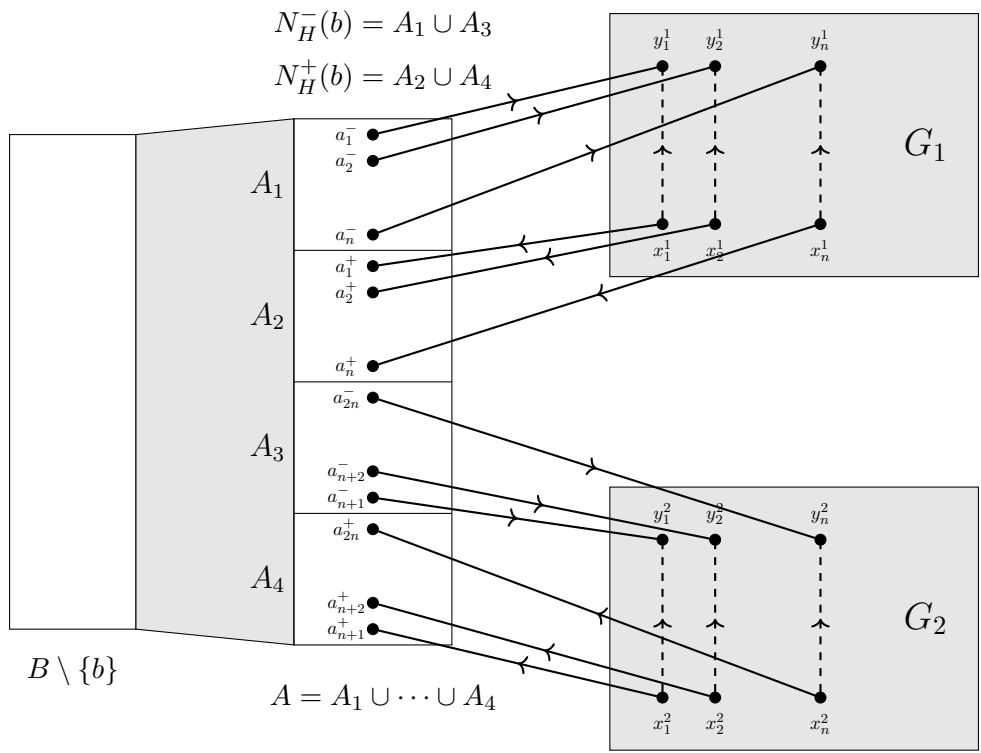
\begin{figure}[h]
    \centering

    \begin{tikzpicture}[scale=0.7]

\node at (6.7,10.8) {$N_H^+(b)=A_2\cup A_4$};
\node at (6.7,11.8) {$N_H^-(b)=A_1\cup A_3$};
\node at (6.7,-1) {$A=A_1\cup \cdots \cup A_4$};
    
\draw  (-0.4,9.7) -- (-0.4,0.3) -- (2,0.3) -- (2,9.7) -- cycle ;
\node at (0.8,-0.5) {$B\setminus\{b\}$};

\draw[fill=SHADEDGRAY] (5,10) -- (2,9.7) -- (2,0.3) -- (5,0) -- cycle ;

\draw  (5,10) -- (5,0) -- (8,0) -- (8,10) -- cycle ;
\draw (5,7.5) -- (8,7.5);
\draw (5,5) -- (8,5);
\draw (5,2.5) -- (8,2.5);
\node at (4.5,8.75) {\large$A_1$};
\node at (4.5,6.25) {\large$A_2$};
\node at (4.5,3.75) {\large$A_3$};
\node at (4.5,1.25) {\large$A_4$};


\draw[fill=SHADEDGRAY]   (11,12) -- (11,7) -- (18,7) -- (18,12) -- cycle ;
\draw[fill=SHADEDGRAY]   (11,3) -- (11,-2) -- (18,-2) -- (18,3) -- cycle ;

\node at (17,9.5) {\Large$G_1$};
\node at (17,0.5) {\Large$G_2$};

\node [style=blackcircle,scale=0.4] (a1-) at (6.5,9.7) {};
\node [scale=0.65] at (6.0,9.7) {$a_{1}^-$};
\node [style=blackcircle,scale=0.4] (a2-) at (6.5,9.2) {};
\node [scale=0.65] at (6.0,9.2) {$a_{2}^-$};
\node [style=blackcircle,scale=0.4] (an-) at (6.5,7.8) {};
\node [scale=0.65] at (6.0,7.8) {$a_{n}^-$};
\node [style=blackcircle,scale=0.4] (an1-) at (6.5,2.8) {};
\node [scale=0.65] at (6.0,2.8) {$a_{n+1}^-$};
\node [style=blackcircle,scale=0.4] (an2-) at (6.5,3.3) {};
\node [scale=0.65] at (6.0,3.3) {$a_{n+2}^-$};
\node [style=blackcircle,scale=0.4] (ann-) at (6.5,4.7) {};
\node [scale=0.65] at (6.0,4.7) {$a_{2n}^-$};
\node [style=blackcircle,scale=0.4] (a1+) at (6.5,7.2) {};
\node [scale=0.65] at (6.0,7.2) {$a_{1}^+$};
\node [style=blackcircle,scale=0.4] (a2+) at (6.5,6.7) {};
\node [scale=0.65] at (6.0,6.7) {$a_{2}^+$};
\node [style=blackcircle,scale=0.4] (an+) at (6.5,5.3) {};
\node [scale=0.65] at (6.0,5.3) {$a_{n}^+$};
\node [style=blackcircle,scale=0.4] (an1+) at (6.5,0.3) {};
\node [scale=0.65] at (6.0,0.3) {$a_{n+1}^+$};
\node [style=blackcircle,scale=0.4] (an2+) at (6.5,0.8) {};
\node [scale=0.65] at (6.0,0.8) {$a_{n+2}^+$};
\node [style=blackcircle,scale=0.4] (ann+) at (6.5,2.2) {};
\node [scale=0.65] at (6.0,2.2) {$a_{2n}^+$};

\node [style=blackcircle,scale=0.4] (x1-1) at (12,8) {};
\node [scale=0.65] at (12,7.5) {$x_{1}^1$};
\node [style=blackcircle,scale=0.4] (y1-1) at (12,11) {};
\node [scale=0.65] at (12,11.5) {$y_{1}^1$};
\node [style=blackcircle,scale=0.4] (x2-1) at (13,8) {};
\node [scale=0.65] at (13,7.5) {$x_{2}^1$};
\node [style=blackcircle,scale=0.4] (y2-1) at (13,11) {};
\node [scale=0.65] at (13,11.5) {$y_{2}^1$};
\node [style=blackcircle,scale=0.4] (xn-1) at (15,8) {};
\node [scale=0.65] at (15,7.5) {$x_{n}^1$};
\node [style=blackcircle,scale=0.4] (yn-1) at (15,11) {};
\node [scale=0.65] at (15,11.5) {$y_{n}^1$};

\node [style=blackcircle,scale=0.4] (x1-2) at (12,-1) {};
\node [scale=0.65] at (12,-1.5) {$x_{1}^2$};
\node [style=blackcircle,scale=0.4] (y1-2) at (12,2) {};
\node [scale=0.65] at (12,2.5) {$y_{1}^2$};
\node [style=blackcircle,scale=0.4] (x2-2) at (13,-1) {};
\node [scale=0.65] at (13,-1.5) {$x_{2}^2$};
\node [style=blackcircle,scale=0.4] (y2-2) at (13,2) {};
\node [scale=0.65] at (13,2.5) {$y_{2}^2$};
\node [style=blackcircle,scale=0.4] (xn-2) at (15,-1) {};
\node [scale=0.65] at (15,-1.5) {$x_{n}^2$};
\node [style=blackcircle,scale=0.4] (yn-2) at (15,2) {};
\node [scale=0.65] at (15,2.5) {$y_{n}^2$};

\begin{scope}[thick,decoration={
    markings,
    mark=at position 0.5 with {\arrow{>}}}
    ] 
    \draw[postaction={decorate},dashed] (x1-1)--(y1-1);
    \draw[postaction={decorate},dashed] (x2-1)--(y2-1);
    \draw[postaction={decorate},dashed] (xn-1)--(yn-1);
    \draw[postaction={decorate},dashed] (x1-2)--(y1-2);
    \draw[postaction={decorate},dashed] (x2-2)--(y2-2);
    \draw[postaction={decorate},dashed] (xn-2)--(yn-2);

    \draw[postaction={decorate}] (x1-1)--(a1+);
    \draw[postaction={decorate}] (x2-1)--(a2+);
    \draw[postaction={decorate}] (xn-1)--(an+);
    \draw[postaction={decorate}] (x1-2)--(an1+);
    \draw[postaction={decorate}] (x2-2)--(an2+);
    \draw[postaction={decorate}] (xn-2)--(ann+);

    \draw[postaction={decorate}] (a1-) to (y1-1);
    \draw[postaction={decorate}] (a2-) to (y2-1);
    \draw[postaction={decorate}] (an-) to (yn-1);
    \draw[postaction={decorate}] (an1-) to (y1-2);
    \draw[postaction={decorate}] (an2-) to (y2-2);
    \draw[postaction={decorate}] (ann-) to (yn-2);
\end{scope}

    \end{tikzpicture}

    \caption{A strongly $n$-connected $2n$-regular oriented graph $G$ on $16n+1$ vertices}
    \label{fig:EXTREMAL-N-CONNECTED-2N-REGULAR}
\end{figure}

\noindent
Are all strongly $3$-connected $d$-regular oriented graphs on $n$ vertices with $d \ge n/8$  Hamiltonian?\\

\noindent
For digraphs, one can similarly ask whether all strongly well connected (or $3$-connected) $d$-regular digraphs on $n$ vertices with $d \ge n/3$ (or $d \ge n/4$, respectively) Hamiltonian?
If the answer is yes, then the value of $d$ is best possible by considering the digraph analogues of the examples given by Propositions~\ref{prop:wellconncetedexample} and~\ref{prop:3-conncetedexample}.


\begin{thebibliography}{1}
	
	    \bibitem{Diestel} R. Diestel, ``Graph Theory (4th Edition)",  \textit{Graduate texts in mathematics 173, Springer} (2012).
	
		\bibitem{Dirac}	G.A. Dirac, ``Some theorems on abstract graphs", \textit{Proc. Lond. Math. Soc.} \textbf{2} (1952), 69–81.
		
		\bibitem{GhouilaHouri} A. Ghouila-Houri, ``Une condition suffisante d'existence d'un circuit hamiltonien", \textit{C.R. Acad. Sci. Paris} \textbf{25} (1960), 495-497.
		
		\bibitem{JacksonDirected} B. Jackson, ``Hamilton cycles in regular 2-connected graphs", \textit{J. Combin. Theory Ser.~B} \textbf{29}	(1980), 27–46.
		
		\bibitem{JacksonConjecture} B. Jackson, ``Long paths and cycles in oriented graphs", \textit{J. Graph Theory} \textbf{5} (1981), 245-252.
		
	
		
		\bibitem{JacksonLiZhu} B. Jackson, H. Li, and Y. Zhu, ``Dominating cycles in regular 3-connected graphs", \textit{Discrete Math.} \textbf{102} (1991), 163–176.
		
		\bibitem{Jung} H. A. Jung, ``Longest circuits in 3-connected graphs", \textit{Finite and infinite sets, Vol I, II, Colloq. Math. Soc. J´anos Bolyai} \textbf{37} (1984), 403–438.
		
		\bibitem{OrientedExactMinDegree} P. Keevash, D. Kühn, and D. Osthus, ``An exact minimum degree condition for Hamilton cycles in oriented graphs", \textit{J. Lond. Math. Soc.} \textbf{79} (2009), 144-166.
		
		\bibitem{oriented-3n-over-8} L. Kelly, D.~Kühn, and D. Osthus, ``A Dirac type result on Hamilton cycles in oriented graphs", \textit{Combin. Probab. Comput.} \textbf{17} (2014), 689-709.
        
		
		\bibitem{IntoTwoBipartiteExpander} D. Kühn, A. Lo, D. Osthus, and K. Staden, ``The Robust Component Structure of Dense Regular Graphs and Applications", \textit{Proc. Lond. Math. Soc.} \textbf{110(1)} (2015), 19-56.
		
		 \bibitem{kuhn2016solution} D. Kühn, A. Lo, D. Osthus, and K. Staden, ``Solution to a problem of Bollobas and Häggkvist on Hamilton cycles in regular graphs", \textit{J. Combin. Theory Ser.~B}, \textbf{121} (2016), 85-145.
        
       
		\bibitem{SurveyDirected} D. Kühn, and D. Osthus, ``A survey on Hamilton cycles in directed graphs", \textit{European J. Combin.}, \textbf{33(5)} (2012), 750-766.
		
		
		\bibitem{On-Kelly-Conjecture} D. Kühn, and D. Osthus, ``Hamilton decompositions of regular expanders: a proof of Kelly's conjecture for large tournaments", \textit{	
Adv. Math.}, \textbf{237} (2013), 62-146.
		
		
		\bibitem{NEWW} D. K\"uhn, and D. Osthus, ``Hamilton decompositions of regular expanders: applications", \textit{J. Combin. Theory Ser.~B}, \textbf{104} (2014), 1–27.
		
		\bibitem{RobustExpanderHamilton} D. Kühn, D. Osthus, and A. Treglown, ``Hamiltonian degree sequences in digraphs", \textit{J. Combin. Theory Ser.~B}, \textbf{100} (2010), 367-380.

		
		\bibitem{LoPatel} A. Lo, and V. Patel, ``Hamilton Cycles in Sparse Robustly Expanding Digraphs'', \textit{Electron. J. Comb.}, \textbf{25(3)} (2018), P3.44 
	
		\bibitem{Chernoff} M. Mitzenmacher, and E. Upfal, ``Probability and Computing: Randomized Algorithms and Probabilistic Analysis", \textit{Cambridge University Press} (2005).
		
		
		
		\bibitem{PatStr} V. Patel and F. Stroh, ``A polynomial-time algorithm to determine (almost) Hamiltonicity of dense regular graphs'',  arXiv:2007.14502, (2020).
		
        \bibitem{Vizing} V. G. Vizing, ``On an estimate of the chromatic class of a p-graph'', \textit{Diskret. Analiz.}, \textbf{3} (1964), 25-30.
		
	\end{thebibliography}
\end{document}